\documentclass[letterpaper, 10pt, dvipsnames]{article}

\usepackage{type1ec}
\usepackage[T1]{fontenc}
\usepackage[utf8]{inputenc}    
\usepackage[english]{babel}
\usepackage[normalem]{ulem}
\usepackage{dsfont, bm, pifont, mathrsfs}
\usepackage{stmaryrd}
\usepackage{bbm}
\usepackage{enumitem}
\usepackage{bold-extra}
\usepackage{float, graphicx, wrapfig, tikz}
\usetikzlibrary{positioning, fit}
\usepackage[margin=0.8in]{subcaption}
\usepackage[font=small, margin=0.6in]{caption}

\usepackage{amsthm, amsmath, amsfonts, amssymb, mathrsfs, mathtools}
\usepackage[alphabetic]{amsrefs}

\usepackage{hyperref, xcolor, titlesec} 
\hypersetup{colorlinks = true, urlcolor = RoyalBlue!70!Black,linkcolor=BrickRed, citecolor=NavyBlue, bookmarksopen = true}
\usepackage[nomarginpar]{geometry}
\numberwithin{equation}{section}
\usepackage[boxsize = .3em]{ytableau}
\usepackage{cleveref}
\renewcommand{\ge}{\geqslant}
\renewcommand{\le}{\leqslant}

\DeclareMathOperator{\Id}{Id}

\renewcommand{\Im}{\operatorname{Im}}
\renewcommand{\Re}{\operatorname{Re}}
\DeclareMathOperator{\Tr}{Tr}

\newcommand{\e}{\mathrm{e}}
\renewcommand{\i}{\mathrm{i}}
\let\d\relax
\newcommand{\d}{\mathrm{d}}

\renewcommand{\a}{\mathbf{a}}
\renewcommand{\b}{\mathbf{b}}
\newcommand{\eb}{\mathbf{e}}

\renewcommand{\k}{\mathbf{k}}

\renewcommand{\q}{\mathbf{q}}

\renewcommand{\u}{\mathbf{u}}
\renewcommand{\v}{\mathbf{v}}
\newcommand{\w}{\mathbf{w}}
\newcommand{\x}{\mathbf{x}}
\newcommand{\y}{\mathbf{y}}
\newcommand{\z}{\mathbf{z}}

\newcommand{\E}{\mathbf{E}}
\newcommand{\F}{\mathbf{F}}

\newcommand{\Hbf}{\mathbf{H}}

\renewcommand{\O}{\mathbf{O}}

\newcommand{\Cbb}{\mathbb{C}}

\newcommand{\Nbb}{\mathbb{N}}
\newcommand{\Obb}{\mathbb{O}}

\newcommand{\Rbb}{\mathbb{R}}

\newcommand{\Ccal}{\mathcal{C}}

\newcommand{\Fcal}{\mathcal{F}}
\newcommand{\Gcal}{\mathcal{G}}

\newcommand{\Jcal}{\mathcal{J}}
\newcommand{\Kcal}{\mathcal{K}}

\newcommand{\Mcal}{\mathcal{M}}
\newcommand{\Ncal}{\mathcal{N}}

\newcommand{\Scal}{\mathcal{S}}
\newcommand{\Tcal}{\mathcal{T}}
\newcommand{\Ucal}{\mathcal{U}}

\newcommand{\Dscr}{\mathscr{D}}

\newcommand{\Lscr}{\mathscr{L}}

\newcommand{\Oscr}{\mathscr{O}}
\newcommand{\Pscr}{\mathscr{P}}

\newcommand{\Xscr}{\mathscr{X}}

\newcommand{\Afr}{\mathfrak{A}}

\newcommand{\Efr}{\mathfrak{E}}

\newcommand{\Lfr}{\mathfrak{L}}

\newcommand{\Qfr}{\mathfrak{Q}}

\renewcommand{\epsilon}{\varepsilon}
\renewcommand{\phi}{\varphi}
\renewcommand{\hat}{\widehat}
\renewcommand{\tilde}{\widetilde}

\newcommand{\scp}[2]{\langle #1,#2\rangle}
\let\O\relax
\newcommand{\O}[1]{\mathcal{O}\left(#1\right)}

\newcommand{\Pds}{\mathds{P}}
\newcommand{\Eds}{\mathds{E}}

\newcommand{\unn}[2]{\left[\!\left[#1,#2\right]\!\right]}

\newtheorem{ccounter}{ccounter}[section]
\newtheorem{theorem}[ccounter]{Theorem}
\newtheorem{lemma}[ccounter]{Lemma}
\newtheorem{corollary}[ccounter]{Corollary}
\newtheorem{definition}[ccounter]{Definition}
\newtheorem{proposition}[ccounter]{Proposition}

\theoremstyle{definition}
\newtheorem{remark}[ccounter]{Remark}

\titleformat{\section}[block]{\normalfont\filcenter}{\large\Roman{section}.}{.7em}{\large\scshape}
\newcommand{\appendixswitch}{\titleformat{\section}[block]{\normalfont\filcenter}{\large\Alph{section}.}{.7em}{\large\scshape}}
\titleformat{\subsection}[runin]{\normalfont}{\it\bf \thesubsection .}{.5em}{\it}[.]
\titleformat{\subsubsection}[runin]{\normalfont}{\bf \thesubsubsection .}{.5em}{\bf}[.]
\titleformat*{\paragraph}{\itshape\mdseries}
\renewcommand{\abstractname}{\normalsize\normalfont\scshape Abstract}

\begin{document}
\tikzset{every node/.style={circle, minimum size=.1cm, inner sep = 2pt, scale=1}}

\title{\vspace{-5ex}\bfseries\scshape{Quantitative eigenvector universality for generalized Wigner matrices}}
\author{\textsc{L. Benigni}\thanks{Supported by NSERC RGPIN 2023-03882 \& DGECR 2023-00076 and a FRQNT \'{E}tablissement de la rel\`eve professorale 364387.}\\\vspace{-0.15cm}\footnotesize{\itshape Universit\'e de Montr\'eal}\\\footnotesize{\itshape lucas.benigni@umontreal.ca}\and  \textsc{P. Lopatto}\thanks{Partially supported by NSF grant DMS-2450004}\\\vspace{-0.15cm}\footnotesize{\itshape University of North Carolina at Chapel Hill}\\\footnotesize{\itshape lopatto@unc.edu}}
\date{}
\maketitle
\renewcommand{\abstractname}{\normalsize\normalfont\scshape Abstract}

\begin{abstract}
    
We present a novel approach to eigenvector universality for generalized
Wigner matrices. Our main consequences are asymptotic normality of joint
eigenvector projections everywhere in the spectrum as well as, under
a uniform subexponential decay assumption, convergence of the extremal
process of eigenvector entries and the resulting Gumbel law for the
largest entry of an eigenvector. In the case of smooth entries, we are
able to obtain joint normality of an explicit growing number of
eigenvector projections, and we are also able to obtain explicit rates
of convergence in Kolmogorov distance, which are nearly optimal for
a fixed number of projections. This is based on a new analysis of
the Dyson vector flow which does not rely on the eigenvector moment flow.
  
\end{abstract}
\section{Introduction and main results}
The analysis of the distribution of wave functions of large self-adjoint Hamiltonians was spurred by Porter and Thomas \cite{porterthomas} to understand nuclear reaction widths. While actual nuclear Hamiltonians are beyond mathematical scope, the breakthrough work of Wigner \cite{wigner1955characteristic} introduced the idea of modeling them by random matrices. Wigner matrices, self-adjoint matrices with independent entries, are a natural model for such Hamiltonians which is amenable to mathematical analysis.
The most central example of a Wigner matrix is the Gaussian Orthogonal Ensemble\footnote{in the real symmetric class and the Gaussian Unitary Ensemble (GUE) in the complex Hermitian class}. If $H$ is an $N\times N$ symmetric matrix, the model is defined by
\[
H\sim\mathrm{GOE}_N:\quad 
H_{ii}\sim \mathcal{N}\left(0,\frac{2}{N}\right),\quad H_{ij}\sim \mathcal{N}\left(0,\frac{1}{N}\right)\text{ for }i<j,
\quad (H_{ij})_{i\leqslant j}\text{ independent}.
\]
Due to the orthogonal invariance of the Gaussian distribution, $\mathrm{GOE}_N$ is invariant with respect to orthogonal conjugation, namely if $O\in \Obb_N(\Rbb)$ is an orthogonal matrix independent of $H$, then $OHO^\top$ has the same distribution as $H$. This invariance has many consequences: each $\ell^2$-normalized eigenvector is uniformly distributed on the unit sphere, the whole eigenbasis is Haar-distributed on the orthogonal group and the eigenvectors are independent of the eigenvalues. This exact distribution of eigenvectors for finite dimension $N$ allows for several quantitative or distributional statements on their entries. A classical consequence, going back to Borel \cite{borel}, is that for a fixed number of indices $\alpha_1,\dots, \alpha_q$, if $\u_k$ denotes the $k$-th eigenvector of $H$ then 
\begin{equation}\label{eq:dist_cv}
\left(\sqrt{N}\u_{k}(\alpha_1),\dots, \sqrt{N}\u_k(\alpha_q)\right)
\xrightarrow[N\to\infty]{(d)} (\mathcal{N}_1,\dots,\mathcal{N}_q) \sim \mathcal{N}(0,\Id_q).
\end{equation}
This result was then refined to understand the asymptotic distribution of a growing number of coordinates as well as a growing minor of a Haar-distributed orthogonal matrix in \cites{diaconis1987dozen, jiang2006entries, jiang2019distances, stewart2020total} which show that for $k_1,\dots,k_{p_N}$ and $\alpha_1,\dots,\alpha_{q_N}$ such that $p_Nq_N=o(N)$, we have
\[
\d_{\mathrm{TV}}\left(     
    \left(
        \sqrt{N}\u_{k_i}(\alpha_j)
    \right)_{\substack{1\leqslant i\leqslant p_N\\1\leqslant j\leqslant q_N}},\,
    \mathcal{N}(0,\Id_{p_Nq_N}) \right) \xrightarrow[N\to\infty]{} 0.
\]
If $\u$ is a vector uniformly distributed on the unit sphere, we have
\begin{equation}\label{eq:max_entry}
\Pds\left(N\Vert \u\Vert_\infty^2 -2\log N + \log\log N + \log(\pi) \leqslant x\right)\xrightarrow[N\to\infty]{} \e^{-\e^{-\frac{x}{2}}}
\end{equation}
and this result has been refined in \cite{jiang2005maximal} to obtain Gumbel convergence of the largest entry of a Haar-distributed orthogonal matrix.

More recently, there has been a large program proving universality of eigenvalue and eigenvector statistics of large random matrices. The goal is to show that these quantitative or distributional statements go beyond Gaussian ensembles and are valid irrespective of the distribution of the entries of the matrix, as long as they satisfy some moment conditions. This program was launched by the works \cites{erdHos2011universality,tao2011random} for universality of eigenvalue statistics. For eigenvector statistics, the first quantitative result is eigenvector delocalization, which states that $\sqrt{N}\Vert \u_k\Vert_\infty \leqslant C\sqrt{\log N}$ with very high probability. There were many subsequent improvements on this result in the context of Wigner matrices \cites{erdos2009local, vuwang,alex2014isotropic,BL22LInfinity}. For distributional convergence as in \eqref{eq:dist_cv}, the first results were obtained in \cites{knowles2013eigenvector,tao2012random} for Wigner matrices whose entries matched the first four moments with Gaussian random variables. The moment matching condition was then removed in the breakthrough work \cite{bourgade2017eigenvector} which introduced a dynamical approach to eigenvector universality in analogy with the earlier dynamical approach to eigenvalue statistics from \cite{erdHos2011universality}. This approach was then refined in \cite{marcinek2022high} to obtain the joint universality of a fixed number of projections of eigenvectors in the bulk for many random matrix models, including Wigner matrices. 
    
For the largest entry of an eigenvector as in \eqref{eq:max_entry}, the optimal first-order upper bound was obtained in \cite{BL22LInfinity}, while the recent work \cite{bao2026order} established the Gumbel law and the extremal Poisson process for edge eigenvectors. Recently, the Gumbel fluctuations for the maximal entry for individual eigenvectors was obtained in \cite{osman26} for Wigner matrices both at the bulk and at the edge. However, their result holds at \emph{fixed energy} instead of \emph{fixed index} meaning that it describes the behavior of eigenvectors corresponding to eigenvalues in a small spectral window rather than a specific eigenvector indexed by its position in the ordered spectrum. Additionally, their method allows for the Gumbel fluctuations of the whole eigenbasis for complex Wigner matrices or real Wigner matrices under a four moment matching condition.

\subsection{Main results}In this work, we introduce a new dynamical approach to universality of eigenvector statistics. This approach is based on the Dyson Brownian motion introduced in \cite{dyson1962brownian}, which is now central to the proof of universality of local statistics in random matrix theory, see \cite{erdHos2017dynamical} for a recent monograph. The \emph{eigenvector moment flow} introduced in \cite{bourgade2017eigenvector} has been a very successful program to obtain universality results for eigenvectors, ranging from studying other models beyond Wigner matrices \cites{bourgade2017sparse, benigni2020eigenvectors, aggarwal2021eigenvector, marcinek2022high}, to studying the eigenstate thermalization hypothesis (ETH) which is related to the quantum unique ergodicity of eigenvectors \cites{bourgade2018random, benigni2021fermionic, cipolloni2022normal, BenLop22QUE, benigni2024fluctuations, benigni2025convergence} or studying delocalization properties through the analysis of high moments in \cite{BL22LInfinity}. One caveat of this method is that it is difficult to extract quantitative results from moments of random variables. Thus, while the approach of \cites{bourgade2017eigenvector,marcinek2022high} is based on the analysis of eigenvector moments along the flow, our approach is based on a direct analysis of the Dyson vector flow. We start by introducing the model of random matrices we consider in this article, which is the class of generalized Wigner matrices.
\begin{definition}[Generalized Wigner matrix]\label{def:genwigner}
A generalized Wigner matrix $W$ is a real symmetric or complex Hermitian matrix of size $N\times N$ whose entries $\{w_{ij}\}_{i\leqslant j}$ are independent centered random variables of variances $s_{ij}^2 = \Eds\vert w_{ij}\vert^2$ which satisfy, for some $c,C>0$, 
\[
\sum_{i=1}^N s_{ij}^2 =1\quad \text{for all }j\in\unn{1}{N}\quad\text{and}\quad   
\frac{c}{N}\leqslant s_{ij}^2 \leqslant \frac{C}{N}\quad\text{for all }i,j\in\unn{1}{N}.
\]
Moreover, we assume that the entries have uniformly bounded moments, i.e. for every $p\in\Nbb$ there exists $C_p>0$ such that for every $i,j\in\unn{1}{N}$, we have $\Eds\vert w_{ij}\vert^p \leqslant C_p N^{-\frac{p}{2}}$.
\end{definition}

If $W$ is a generalized Wigner matrix we denote $\lambda_1\leqslant \dots\leqslant \lambda_N$ its eigenvalues and $(\u_1,\dots,\u_N)$ its corresponding eigenvectors. We now state our different results obtained through our novel approach. We note that we state only the case of real symmetric matrices for simplicity and we give the corresponding results for complex Hermitian matrices in Appendix \ref{app:hermitian}. For real symmetric matrices, since eigenvectors are only defined up to a phase, we set the problem by introducing $(\varepsilon_i)_{i\in\unn{1}{N}}$ an i.i.d. family of Rademacher variables independent of $W$ and we study the distribution of $\v_i\coloneqq \sqrt{N}\varepsilon_i\u_i$. We start by stating the universality of joint eigenvector projections for generalized Wigner matrices.
\begin{theorem}
\label{theo:finite}
Let $W$ be a real symmetric generalized Wigner matrix. Let $p\geqslant 1$ and $a_1,\dots,a_p\geqslant 1$ be fixed integers, and set $Q=\sum_{i=1}^p a_i.$
Let $k_1,\dots,k_p\in\unn{1}{N}$ be distinct indices, and for every
$1\leqslant i\leqslant p$, let $\q_{i1},\dots,\q_{ia_i}$
be deterministic unit vectors in $\Rbb^N$. We set
\[
    X_N
    =
    \left(
        \langle \q_{i\alpha},\v_{k_i}\rangle
    \right)_{\substack{1\leqslant i\leqslant p\\1\leqslant \alpha\leqslant a_i}}
    \in\Rbb^Q.
    \]
Let $G_N=(G_{i\alpha})_{1\leqslant i\leqslant p,\,1\leqslant \alpha\leqslant a_i}$ be the centered Gaussian vector in $\Rbb^Q$ with covariance $\Eds\left[G_{i\alpha}G_{j\beta}\right]=    \delta_{ij}\scp{\q_{i\alpha}}{\q_{j\beta}}   $.
Then there exists $\varepsilon=\varepsilon(Q)>0$ such that for every smooth $h$ for which there exist $K_1,\,K_2>0$ with
\[
\Vert h\Vert_\infty,\, \Vert \nabla^2 h\Vert_\infty \leqslant K_1 \quad\text{and}\quad \vert \partial^n  h(\x)\vert \leqslant K_2(1+\vert \x\vert)^{K_2}
\] 
for $\vert n\vert \leqslant 5$, we have
\[
    \left\vert
    \Eds\left[h(X_N)\right]
    -
    \Eds\left[h(G_N)\right]
    \right\vert
    \leqslant C_hN^{-\varepsilon}
\]
for some $C_h = C_h(K_1,K_2,Q)$.
\end{theorem}
We note that this joint universality of eigenvector projections is new, even in the context of generalized Wigner matrices. Indeed, the first proofs of eigenvector universality were obtained in \cites{knowles2013eigenvector,tao2012random} for random matrices whose entries matched the first four moments with Gaussian random variables. We note that only two moments are needed when considering edge eigenvector indices in \cite{knowles2013eigenvector}. In \cite{bourgade2017eigenvector}, the authors introduced a dynamical approach to eigenvector universality but only considered the joint distribution of a fixed projection of eigenvectors, though considering the whole spectrum. In the context of our theorem above, this would correspond to the case $a_i=1$ and $\q_{i1}=\q$ for a fixed unit vector $\q$ for every $i\in\unn{1}{p}$. Finally, \cite{marcinek2022high} obtained the joint universality of eigenvector projections for many random matrix models, including generalized Wigner matrices, but only for a fixed number of projections in the bulk, in the context of our theorem, this would restrict $k_1,\dots, k_p \in \unn{\kappa N}{(1-\kappa)N}$ for any small $\kappa>0$. Our result above gives the first joint eigenvector universality for a finite number of arbitrary projections of arbitrary eigenvectors for generalized Wigner matrices without matching the first four moments.

The quantitative nature of our dynamical result, Theorem \ref{theo:maindyn}, allows us to identify the extremal process of the coordinates of an eigenvector under a uniform subexponential decay assumption. We say that the entries have \emph{uniformly subexponential decay} if there exist $c,C>0$ such that
\[
    \Pds\left(\sqrt N\vert w_{ij}\vert>x\right)
    \leqslant C\e^{-x^c},
    \qquad x\geqslant0,
\]
uniformly in $i,j,N$. Let $\Mcal_p(\Rbb)$ denote the space of locally finite point measures on $\Rbb$, with the vague topology.
\begin{theorem}\label{theo:extremal_process}
Let $W$ be a real symmetric generalized Wigner matrix whose entries have uniformly subexponential decay. For $k\in\unn{1}{N}$, define
\[
    \Xi_{N,k}
    =
    \sum_{\alpha=1}^N
    \delta_{N\vert \u_k(\alpha)\vert^2-b_N},
    \qquad
    b_N=2\log N-\log\log N-\log\pi.
\]
Uniformly in $k$, $\Xi_{N,k}$ converges in distribution in $\Mcal_p(\Rbb)$ to a Poisson point process $\Xi$ with intensity
\[
    \nu(\d x)=\frac12\e^{-x/2}\,\d x.
\]
More precisely, for every nonnegative $f\in\Ccal_c(\Rbb)$,
\[
\sup_{k\in\unn{1}{N}}
\left\vert
    \Eds\left[\exp\left(-\int f\,\d\Xi_{N,k}\right)\right]
    -
    \exp\left(-\int_{\Rbb}(1-\e^{-f(x)})\frac12\e^{-x/2}\,\d x\right)
\right\vert
\xrightarrow[N\to\infty]{}0.
\]
Moreover, for every fixed $m\in\Nbb$, the $m$ largest atoms of
$\Xi_{N,k}$ converge jointly in distribution to the $m$ largest
atoms of $\Xi$, uniformly in $k$.
\end{theorem}
As a direct consequence, we obtain the Gumbel distribution for the largest entry of an eigenvector.
\begin{corollary}\label{theo:gumbel}
Under the assumptions of Theorem \ref{theo:extremal_process}, for every $x\in\Rbb$,
\[
\sup_{k\in\unn{1}{N}}
\left\vert
    \Pds\left(
        N\Vert \u_k\Vert_\infty^2
        -2\log N+\log\log N+\log\pi
        \leqslant x
    \right)
    -\e^{-\e^{-x/2}}
\right\vert
\xrightarrow[N\to\infty]{}0.
\]
\end{corollary}
This corollary implies that under the assumptions of Theorem \ref{theo:extremal_process} and for every $\varepsilon>0$,  
\[
    \sup_{k\in\unn{1}{N}}
    \Pds\left(
        \left\vert
        \sqrt{\frac N{\log N}}\Vert \u_k\Vert_\infty-\sqrt2
        \right\vert>\varepsilon
    \right)
    \xrightarrow[N\to\infty]{}0.
\]
This improves the sharp upper bound of \cite{BL22LInfinity}*{Theorem 1.3} to convergence in probability.

There are more quantitative results we can obtain using our dynamical approach through Theorem \ref{theo:maindyn}. Indeed, we keep the dependence on the number of eigenvector projections explicit in the theorem allowing us to obtain universality results for a growing number of eigenvector projections. However, while these statistics can be obtained dynamically, the comparison step to remove the Gaussian component through Green function comparison as in \cite{knowles2013eigenvector} is not quantitative enough to obtain the corresponding result for generalized Wigner matrices. We thus use the reverse heat flow technique which needs a smoothness assumption on the entries of the matrix given by the following definition.
\begin{definition}[Smooth generalized Wigner matrices]\label{def:smooth}   We say that a real symmetric generalized Wigner matrix $W$ is smooth
if the rescaled entries $\sqrt{N}w_{ij}$ have densities    proportional to $\e^{-V_{ij}(x)}$ where $V_{ij}$  are smooth functions such that for every $k\in\Nbb$, there exists $C_k$ such that 
    \[
    \sup_{i,j\in\unn{1}{N}}\vert V_{ij}^{(k)}(x)\vert \leqslant C_k(1+\vert x\vert)^{C_k}.
    \]
\end{definition}
    The smoothness assumption for complex Hermitian matrices is stated
in Appendix \ref{app:hermitian}.    
We have the following analogue of Theorem \ref{theo:finite} for a growing number of eigenvector projections.
\begin{theorem}
\label{theo:growing}
Let $W$ be a smooth real symmetric generalized Wigner matrix. Let $p\geqslant 1$ and $a_1,\dots,a_p\geqslant 1$ and set $Q=\sum_{i=1}^p a_i.$
Let $k_1,\dots,k_p\in\unn{1}{N}$ be distinct indices, and for every
$1\leqslant i\leqslant p$, let $\q_{i1},\dots,\q_{ia_i}$ be deterministic unit vectors in $\Rbb^N$. We denote 
\[
R = \dim\left(\mathrm{Span}\{\q_{i\alpha},\,i\in\unn{1}{p},\alpha\in\unn{1}{a_i}\}\right).
\]
We set
\[
    X_N
    =
    \left(
        \langle \q_{i\alpha},\v_{k_i}\rangle
    \right)_{\substack{1\leqslant i\leqslant p\\1\leqslant \alpha\leqslant a_i}}
    \in\Rbb^Q.
    \]
Let $G_N=(G_{i\alpha})_{1\leqslant i\leqslant p,\,1\leqslant \alpha\leqslant a_i}$ be the centered Gaussian vector in $\Rbb^Q$ with covariance $\Eds\left[G_{i\alpha}G_{j\beta}\right]=    \delta_{ij}\scp{\q_{i\alpha}}{\q_{j\beta}}$.    Assume that there exists $b\in(0,\frac{1}{2}]$ and $C>0$ such that 
\[
Q\leqslant RQ \leqslant CN^{\frac{1}{2}-b}.
\]
    
Let $h$ be a smooth function such that, for some $\delta_1>0$ and $\delta_2<b$,
\[
\Vert h\Vert_\infty\leqslant N^{\delta_1}\quad\text{and}\quad
\Vert \nabla^2 h\Vert_\infty \leqslant N^{\delta_2}.
\]
For every $\kappa>0$, there exists $C_0>0$ such that
\[
    \left\vert
    \Eds\left[h(X_N)\right]
    -
    \Eds\left[h(G_N)\right]
    \right\vert
    \leqslant C_0 N^{-b+\delta_2+\kappa}.
\]
   
\end{theorem}

To our knowledge, Theorem \ref{theo:growing} gives the first eigenvector universality result for an explicit growing number of eigenvector projections. We note that \cite{tao2012random} developed a Green function comparison theorem which can be used to compare \(N^\delta\) eigenvector entries with i.i.d. Gaussian random variables, for some implicit \(\delta>0\) under a four moment matching condition. The point of Theorem \ref{theo:growing} is that the admissible number of projections is explicit and depends on the two natural parameters \(Q\), the total number of projections, and \(R\), the dimension of their linear span.

Let us spell out a few consequences.

\begin{itemize}
    \item \textbf{Eigenbasis row.}
    Fix \(\alpha\in\unn{1}{N}\) and consider $\left(
            \sqrt{N}\u_{k_1}(\alpha),\dots,\sqrt{N}\u_{k_p}(\alpha)
        \right).$
    This corresponds to taking \(a_i=1\) and \(\q_{i1}=\eb_\alpha\) for every \(i\in\unn{1}{p}\). In this case \(R=1\) and \(Q=p\). The condition of Theorem \ref{theo:growing} allows
    \[
        p\leqslant N^{\frac12-\varepsilon}
    \]
    for any fixed $\varepsilon>0$, after choosing $b>0$ appropriately. Thus we obtain joint Gaussianity for \(N^{\frac{1}{2}-\varepsilon}\) entries of a fixed row of the eigenbasis. This is not optimal in the GOE case: by \cite{diaconis1987dozen}, for a Haar-distributed orthogonal matrix, any \(o(N)\) entries of a fixed row are asymptotically jointly Gaussian. We note that the result was stated in \cite{diaconis1987dozen} for vectors uniformly distributed on the sphere which is the law of the rows (and the columns) of a Haar-distributed orthogonal matrix.

    \item \textbf{Eigenbasis minor.}
    Consider a minor formed by \(p\) eigenvectors and \(r\) coordinates. Then $Q=pr$ and $R=r$.
    If \(p\) and \(r\) are of the same order, say \(p=r=m\), then the condition \(RQ\leqslant N^{\frac12-\varepsilon}\) allows
    \[
        m\leqslant N^{\frac{1}{6}-\varepsilon}.
    \]
    Hence Theorem \ref{theo:growing} gives joint universality for minors of size  $N^{\frac16-\varepsilon}\times N^{\frac16-\varepsilon}.$
    In the GOE case, the optimal Haar result is stronger since \cite{jiang2006entries} proves joint Gaussianity for minors up to size
        $N^{\frac12-\varepsilon}\times N^{\frac12-\varepsilon}$, and shows that this scale is optimal. If we consider a rectangular minor of $p$ eigenvectors and $r$ coordinates then our condition becomes $pr^2\leqslant N^{\frac{1}{2}-\varepsilon}$ while it was shown in \cites{jiang2019distances, stewart2020total} that for a Haar-distributed orthogonal matrix, the optimal condition is $pr=o(N)$.

    \item \textbf{Eigenvector entries.}
    If we consider $Q$ entries of a single eigenvector, then $p=1$ and $a_1=Q$ and $R=Q$.   The condition $RQ=Q^2\leqslant N^{\frac12-\varepsilon}$ allows    
    \[
        Q\leqslant N^{\frac14-\varepsilon}.
    \]
    Thus we obtain joint universality for \(N^{\frac14-\varepsilon}\) arbitrary entries of one eigenvector. Again this is not expected to be optimal since in the Haar case, \cite{diaconis1987dozen} gives asymptotic Gaussianity for any \(o(N)\) entries.
\end{itemize}
    
We show in Appendix \ref{app:rate} that the exponent $\frac{1}{2}$
in Theorems \ref{theo:growing} and \ref{theo:kolmogorov} for a fixed
number of eigenvector projections cannot be improved in general.

Our dynamical approach also gives bounds on the Kolmogorov distance
between the distribution of eigenvector projections and the Gaussian
distribution.    Define the Kolmogorov distance for two vectors $X$ and $Y$ in $\Rbb^Q$ as
\[
    \d_{\mathrm K}(X,Y)
    =
    \sup_{\y\in\Rbb^Q}
    \left|
    \Pds\left(X\leqslant \y\right)
    -
    \Pds\left(Y\leqslant \y\right)
    \right|,
\]
where the inequality is coordinatewise. We can obtain the following bound valid for a growing number of eigenvector projections under the same assumption on the smoothness of the entries of the matrix.

\begin{theorem}
   
\label{theo:kolmogorov}
Under the assumptions of Theorem \ref{theo:growing}, for every
$\kappa>0$ there exists $C>0$ such that
\[
    \d_{\mathrm K}(X_N,G_N)
    \leqslant
    C N^{-b+\kappa}.
\]
In particular, for a fixed number of eigenvector projections, i.e.
$b=\frac{1}{2}$, we get, for every $\kappa>0$,
\[
    \d_{\mathrm K}(X_N,G_N)
    \leqslant
    C N^{-\frac{1}{2}+\kappa}.
\]
  
\end{theorem}
    
We note that a faster rate is possible in the Haar case. For a single
entry of a Haar-distributed orthogonal matrix, the Kolmogorov distance
between its rescaled distribution and the standard Gaussian
distribution is of order $N^{-1}$, which can be seen directly from
its density    
\[
f_{\Obb_N(\Rbb)}(x) = \frac{\Gamma\left(\frac{N}{2}\right)}{\sqrt{\pi N}\Gamma\left(\frac{N-1}{2}\right)}\left(1-\frac{x^2}{N}\right)^{\frac{N-3}{2}}\mathbf{1}_{\vert x\vert \leqslant \sqrt{N}}
=
\frac{1}{\sqrt{2\pi}}\e^{-\frac{x^2}{2}}\left(1+\frac{1}{N}\left(-\frac{3}{4}+\frac{3x^2}{2}-\frac{x^4}{4}\right)+\O{\frac{1+\vert x\vert^6}{N^2}}\right).
\]
However, to our knowledge, this is the first explicit quantitative bound on convergence of eigenvector entries and projections for smooth generalized Wigner matrices.

\subsection{Sketch of the proof} The proof is based on the three-step dynamical approach to universality. The first step consists of obtaining mesoscopic bounds on eigenvalue and eigenvector statistics of generalized Wigner matrices. Similarly to the eigenvector moment flow, we leverage eigenvalue rigidity \cite{erdos2012rigidity}, eigenvector delocalization \cite{alex2014isotropic} as well as an eigenvector fluctuation result coming from the isotropic local law \cite{alex2014isotropic}.

The second step consists of analyzing the dynamics of eigenvectors along the Dyson Brownian motion whose definition we now recall.
\begin{definition}\label{def:dbm}
    Let $B$ be an $N\times N$ matrix such that $B_{ij}=B_{ji}$ and for $i\leqslant j$, $B_{ij}$ are independent Brownian motions with variance $1+\delta_{ij}$. The Dyson Brownian motion starting at $H_0$ is the process defined by
    \[
    H_t = H_0 + \frac{1}{\sqrt{N}}B_t.
    \]
\end{definition}
The stochastic differential equations followed by its eigenvalues and eigenvectors were computed in \cites{dyson1962brownian, bru89, anderson2010introduction} and are given by the following theorem. We additionally give the generators of the eigenvector dynamics which can be found in \cite{bourgade2017eigenvector}*{Lemma 2.4}.
\begin{theorem}\label{theo:dyson}
    Let $H_t$ be the Dyson Brownian motion starting at $H_0=U_0\Lambda_0 U_0^*$. Let $\lambda_1(t)\leqslant \dots\leqslant \lambda_N(t)$ and $U_t = (\u_1(t),\dots,\u_N(t))$ be the strong solutions to the system,
    \[
    \begin{aligned}
    \d \lambda_k(t) &= \frac{\d B_{kk}(t)}{\sqrt{N}} + \frac{1}{N}\sum_{j\neq k} \frac{\d t}{\lambda_k(t)-\lambda_j(t)},\\
    \d \u_k(t) &= \frac{1}{\sqrt{N}}\sum_{\ell\neq k} \frac{\u_\ell(t)\d B_{k\ell}(t)}{\lambda_k(t)-\lambda_\ell(t)} - \frac{1}{2N}\sum_{\ell\neq k} \frac{\u_k(t)\d t}{(\lambda_k(t)-\lambda_\ell(t))^2}.
    \end{aligned}
    \]
    Then the processes $(H_t)_{t\geqslant 0}$ and $(U_t\Lambda_t U_t^*)_{t\geqslant 0}$ have the same distribution.  
    Besides, the generator acting on smooth function $f:\Rbb^{N^2}\to\Rbb$ of the eigenvector dynamics is given by
    \[
    \Lscr_t = \frac{1}{2}\sum_{i\neq j}c_{ij}(t)\Xscr_{ij}^2 \quad
    \text{with}\quad 
       c_{ij}(t) = \frac{1}{2N(\lambda_i(t)-\lambda_j(t))^2}   
    \]
    where   
    \[ 
    \Xscr_{ij} = \u_i\partial_{\u_j}-\u_j\partial_{\u_i} 
     = \sum_{\alpha=1}^N \left(\u_i(\alpha)\partial_{\u_j(\alpha)}-\u_j(\alpha)\partial_{\u_i(\alpha)}\right).
    \]
\end{theorem}
Let us describe the main idea in the simplest case of one eigenvector entry, say
\[
    \v_k(t,\alpha)=\sqrt N\,{\u_k(t,\alpha)}.
\]
We split the generator from Theorem \ref{theo:dyson} into a short-range part, corresponding to $|i-j|<\ell$, and a long-range part, corresponding to $|i-j|\geqslant \ell$ for some range $\ell \ll Nt^3$. The short-range part is kept unchanged, while the long-range part is replaced by an Ornstein--Uhlenbeck mechanism. More precisely, the reference generator is
\[
    \hat{\Lscr}_t
    =
    \frac{1}{2}\sum_{|i-j|<\ell}c_{ij}(t)\Xscr_{ij}^2
    +
    \sum_i\beta_i(t)(\partial^2_i-\v_i(\alpha)\partial_i),
    \qquad
    \beta_i(t)=\sum_{|i-j|\geqslant\ell}c_{ij}(t),
\]
    with the convention that $c_{ii}(t) =0$.   
The advantage of $\hat{\Lscr}_t$ is that its long-range part has a genuine log-Sobolev property coming from the Ornstein--Uhlenbeck structure: after evolving backward with the semigroup generated by $\hat{\Lscr}_t$, the observable $h(\v_k(\alpha))$ quantitatively relaxes to its Gaussian expectation. It remains to compare the true flow with this reference flow. Since the two generators differ only in the long-range part, which consists of weak interactions between far-away eigenvalues, this difference is somewhat perturbative and can be written as 
\[
\Lscr_t-\hat{\Lscr}_t
=
\sum_{|i-j|\geqslant \ell}c_{ij}(t)(\v_i^2(\alpha)-1)\partial_j^2
-
\sum_{|i-j|\geqslant \ell}c_{ij}(t)\v_i(\alpha)\v_j(\alpha)\partial^2_{ij}.
\] 
We write the difference of expectations through a Duhamel formula involving \(\Lscr_t-\hat{\Lscr}_t\), and control the resulting error using rigidity, delocalization and ETH fluctuations estimates for the eigenvectors. Thus the proof separates into two conceptually distinct steps: first, prove that the reference dynamics Gaussianizes because of the Ornstein--Uhlenbeck long-range component; second, show that replacing the true long-range eigenvector flow by this OU component changes the expectation only by a small error. Section \ref{sec:dysonvector} is devoted to these two steps in the dynamical part of the proof.    For the extremal process, we prove a relative version of this comparison
for events whose Gaussian probability is small. For the Kolmogorov
bound, we use the Gaussian smoothing in the reference flow to control
the Hessian of the backward observable without losing a power of $N$.    

Finally, the third step consists of removing the Gaussian component through a comparison argument. For the universality of a fixed number of eigenvector projections, we can use the Green function comparison method from \cite{knowles2013eigenvector} which is based on a Lindeberg replacement strategy   and for the extremal process we extend the comparison argument of
\cite{BL22LInfinity} to a fixed number of upper order statistics.     For the universality of a growing number of eigenvector projections, we use the reverse heat flow technique from \cites{erdHos2011universality,bourgade2022extreme}. This is done in Section \ref{sec:proofs}. We note that it would be interesting to see if some of these quantitative results could be transferred through Green function comparison by assuming some moment matching condition as in \cite{zhang2025quantitative} which was carried out for eigenvalue gaps. However, there is a clear barrier in the current Green function comparison method for eigenvectors which is that it relies on level repulsion estimates which, even in the integrable case of GOE, are not quantitative enough to allow for a union bound on $N^{\delta}$ eigenvalues for a possibly large $\delta$.

\subsection{Acknowledgements} The   first    author would like to thank P.\ Bourgade and G.\ Cipolloni for useful discussions and comments on the manuscript.   This paper was written with the assistance of large language models, which suggested arguments and contributed to revisions.   

\section{Analysis of the Dyson vector flow}\label{sec:dysonvector}
We first define the quantiles of the semicircle distribution $(\gamma_k)_{1\leqslant k\leqslant N}$ by the relation
\begin{equation}\label{eq:quantiles}
N\int_{-\infty}^{\gamma_k}\frac{1}{2\pi}\sqrt{(4-x^2)_+}\d x = k.
\end{equation}
We note that the variance of the matrix elements of the Dyson Brownian motion from Definition \ref{def:dbm} depends on time $t$. In particular, if we denote these variances as $s^2_{ij}(t)$ then we have $s^2_{ij}(t) = s^2_{ij}+\frac{(1+\delta_{ij})t}{N}$ where $s^2_{ij}$ comes from the definition of generalized Wigner matrices in Definition \ref{def:genwigner}. If we denote $a(t) = (1+(1+\frac{1}{N})t)^{-\frac{1}{2}}$ then $a(t)H_t$ is a generalized Wigner ensemble and has the same eigenvectors as $H_t$. Its eigenvalues satisfy 
\[
\frac{1}{N(\lambda_i(a(t)H_t)-\lambda_j(a(t)H_t))^2} = \frac{    1  }{    a(t)^2    N(\lambda_i(H_t)-\lambda_j(H_t))^2}
\]
and they are rigid around the quantiles of the semicircle distribution as in \eqref{eq:quantiles}. Since in the time range we are considering $a(t)\asymp 1$, this time change does not affect the estimates on the eigenvalues and eigenvectors of $H_t$ that we will use.
In the rest of the paper, it will always be understood that the time change $t\mapsto \int_0^t a(s)^{2}\d s$ is applied and we identify $\lambda_i(t)$ with $\lambda_i(a(t)H_t)$.

We first define the good events which will be used throughout the proof. 

\begin{definition}\label{def:goodevent}
    Let $\Qfr$ and $\Efr$ be two finite families of deterministic unit vectors in $\Rbb^N$. Let $\xi>0$. For $s\in[0,1]$, we denote
    \[
    \Afr_s(\xi,\Qfr,\Efr,N)
    =
    \Afr_{1,s}(\xi,\Qfr,N)\cap \Afr_{2,s}(\xi,\Efr,N),
    \]
    where
    \[
    \begin{aligned}
    \Afr_{1,s}(\xi,\Qfr,N) &=\bigcap_{i=1}^N\bigcap_{\q \in \Qfr}\left\{
        \vert \v_i(s,\q)\vert^2 \leqslant N^{\xi}
    \right\},\\
    \Afr_{2,s}(\xi,\Efr,N) &=\bigcap_{\eb,\eb'\in\Efr} \left\{
        \frac{1}{\sqrt{\vert I\vert}}\left\vert\sum_{j\in I} \left(
            \v_j(s,\eb)\v_j(s,\eb')-\scp{\eb}{\eb'}
        \right)\right\vert
        \leqslant N^\xi \text{ for all } I\subset\unn{1}{N}\text{ interval}
    \right\}.
    \end{aligned}
    \]
    We also define the event
    \[
    \Afr(\xi,\Qfr,\Efr,N)
    =
    \bigcap_{s\in[0,1]}\Afr_s(\xi,\Qfr,\Efr,N).
    \]
    We define the eigenvalue rigidity event
    \[
    \Afr_3(\xi,N)=\bigcap_{i=1}^N\left\{
        \vert \lambda_i(t)-\gamma_i\vert \leqslant N^{-\frac{2}{3}+\xi}\min\left(i,N-i+1\right)^{-\frac{1}{3}}\text{ for all }t\in[0,1]
    \right\}.
    \]
    For $\nu>0$, we define the set of good eigenvalue paths
    \[
    \Lfr(\xi,\Qfr,\Efr,\nu,N)
    =
    \left\{
        \bm{\lambda},\, \bm{\lambda}\text{ satisfies }\Afr_3(\xi,N)\text{ and }\Pds(\Afr(\xi,\Qfr,\Efr,N)\vert \bm{\lambda})\geqslant 1-N^{-\nu}
    \right\}.
    \]
    Finally, we define the stopping time
    \[
    \tau=\tau(\xi,\Qfr,\Efr,N)
    =
    \inf\left\{s\in[0,1]\,:\, \Afr_s(\xi,\Qfr,\Efr,N)\text{ fails}\right\},
    \]
    with the convention $\inf\emptyset=1$.
\end{definition}

\begin{lemma}\label{lem:good_event}
    Let $\xi>0$ and $N^{2\xi}\leqslant \ell \leqslant N^{1-\xi}$. Suppose that $\vert \Qfr\vert$ and $\vert \Efr\vert$ are bounded by $N^2$. Then for any $D>0$ we have 
    \[
    \Pds\left(\Afr(\xi,\Qfr,\Efr,N)\right)\geqslant 1-N^{-D},
    \qquad
    \Pds\left(\Afr_3(\xi,N)\right)\geqslant 1-N^{-D},
    \]
    and
    \[
    \Pds\left(\Lfr(\xi,\Qfr,\Efr,\nu,N)\right)\geqslant 1-N^{-D}.
    \]
    Additionally, for any $\bm{\lambda}\in \Lfr(\xi,\Qfr,\Efr,\nu,N)$, we have 
    \[
    \Pds(    \tau <  1\vert    \bm{\lambda})\leqslant N^{-\nu}.
    \]
\end{lemma}
\begin{proof}
    For a fixed time $s$, rigidity was obtained in \cite{erdos2012rigidity}, and the isotropic local law and isotropic delocalization in \cite{alex2014isotropic}. It remains to prove the estimate in $\Afr_{2,s}$.

    Fix $\eb,\eb'\in\Efr$ and define
    \[
        A_{\eb,\eb'}=\sqrt{N}\eb\eb'^\top,
        \qquad
        \mathring A_{\eb,\eb'}=A_{\eb,\eb'}-\langle A_{\eb,\eb'}\rangle \Id_N,
    \]
    Then
    \[
        \v_j(s,\eb)\v_j(s,\eb')-\scp{\eb}{\eb'}
        =
        \sqrt N\langle \u_j(s),\mathring A_{\eb,\eb'}\u_j(s)\rangle .
    \]
    Therefore, for an interval $I\subset\unn{1}{N}$,
    \[
        \frac{1}{\sqrt{|I|}}
        \sum_{j\in I}
        \left(
            \v_j(s,\eb)\v_j(s,\eb')-\scp{\eb}{\eb'}
        \right)
        =
        \sqrt{\frac{N}{|I|}}
        \sum_{j\in I}
        \langle \u_j(s),\mathring A_{\eb,\eb'}\u_j(s)\rangle .
    \]
    If $P_I$ denotes the spectral projection onto the eigenvectors with indices in $I$, this is
    \[
        \frac{1}{\sqrt{\vert I\vert}}\sum_{j\in I}\left(
            \v_j(s,\eb)\v_j(s,\eb')-\scp{\eb}{\eb'}
        \right)
        =\frac{N^{\frac{3}{2}}}{\sqrt{|I|}}\langle P_I \mathring A_{\eb,\eb'}\rangle .
    \]

    We first consider mesoscopic intervals, namely $N^\varepsilon\leqslant |I|\leqslant N^{1-\varepsilon}$, where $\varepsilon>0$ will be chosen sufficiently small with respect to $\xi$.   Let $I=\unn{a}{b}$ and set $\rho_I=(|I|/N)^{2/3}$.
Let $J$ be the deterministic spectral interval defined by    
    \[
        J=[E_-,E_+],
        \qquad
        E_-=
        \begin{cases}
          -2-\rho_I, & a=1,    \\
        \frac{\gamma_{a-1}+\gamma_a}{2},& a>1,
        \end{cases}
        \qquad
        E_+=
        \begin{cases}
          2+\rho_I,& b=N,    \\
        \frac{\gamma_b+\gamma_{b+1}}{2},& b<N.
        \end{cases}
    \]
    By rigidity, the projections $P_I$ and $\mathbf{1}_{J}(H_s)$ differ only by eigenvectors whose indices lie in a boundary layer of size at most $N^{c\varepsilon}$. Using isotropic delocalization, this gives
    \[
        \frac{N^{\frac{3}{2}}}{\sqrt{\vert I\vert}}
        \left|
            \langle P_I \mathring A_{\eb,\eb'}\rangle
            -
            \langle \mathbf{1}_{J}(H_s)\mathring A_{\eb,\eb'}\rangle
        \right|
        \leqslant
        \frac{N^{c\varepsilon}}{\sqrt{\vert I\vert}}
    \]
    with overwhelming probability. This error is $O(N^\xi)$ by choosing $\varepsilon$ small enough with respect to $\xi$.

    We now estimate the deterministic spectral projection. Let $G_s(z)=(H_s-z)^{-1}$ and let $m_{\mathrm{sc}}(z)$ be the Stieltjes transform of the semicircle law. We first record the local law input. Since
    \[
        \langle G_s(z)\mathring A_{\eb,\eb'}\rangle
        =
        \frac{1}{\sqrt N}
        \left(
            \langle \eb',G_s(z)\eb\rangle
            -
            \scp{\eb}{\eb'}\langle G_s(z)\rangle
        \right),
    \]
    the deterministic term $m_{\mathrm{sc}} (z)$ cancels, and the isotropic and averaged local laws from \cite{alex2014isotropic} give, for every $\chi>0$ and $D>0$, with probability at least $1-N^{-D}$,
    \[
        \left|
           \sqrt{N} \langle G_s(z)\mathring A_{\eb,\eb'}\rangle
        \right|
        \leqslant
        N^\chi
        \left(
            \sqrt{\frac{\Im m_{\mathrm{sc}}(z)}{N\eta}}
            +
            \frac{1}{N\eta}
        \right),
        \qquad z=E+\i\eta,
    \]
    uniformly for $|E|\leqslant 3$ and $\eta\geqslant N^{-1+\varepsilon/10}$.

    We apply Pleijel's representation formula for sharp spectral cut-offs, as in \cite{cipolloni2023functional}*{Proof of Theorem 2.3}, and previously used in a similar context in \cite{erdos2016fluctuations}. Let $\eta_0=N^{-1+\varepsilon/10}$ and let $\Gamma_I$ be the rectangular contour around $J$ with horizontal height $\vert J\vert$, truncated at distance $\eta_0$ from the real axis. Then
    \[
        \frac{N^{\frac32}}{\sqrt{\vert I\vert}}
        \langle \mathbf{1}_{J}(H_s)\mathring A_{\eb,\eb'}\rangle
        =
        \frac{N^{\frac32}}{2\pi i\sqrt{\vert I\vert}}
        \int_{\Gamma_I}
        \langle G_s(z)\mathring A_{\eb,\eb'}\rangle \d z
        +
        \O{\frac{N\eta_0}{\sqrt{\vert I\vert}}}.
    \]
    The last error is $O(N^\xi)$ by our choice of $\eta_0$ and since $\vert I\vert\geqslant N^\varepsilon$.

    Using the previous local law on $\Gamma_I$, together with the elementary semicircle estimates
    \[
        \Im m_{\mathrm{sc}}(E+\i\eta)\leqslant C\sqrt{\kappa(E)+\eta},
        \qquad
        \kappa(E)=\mathrm{dist}(E,\{-2,2\}),
    \]
    and the relation between $\vert I\vert$ and $\vert J\vert$, we get for any $\chi>0$,
    \[
        \frac{N^{\frac{3}{2}}}{\sqrt{\vert I\vert}}
        \int_{\Gamma_I}
        \left|
            \langle G_s(z)\mathring A_{\eb,\eb'}\rangle
        \right|
        |\d z|
        \leqslant
        N^{\chi}.
    \]
    Indeed, in the bulk, $\vert J\vert\asymp \frac{\vert I\vert}{N}$, and the horizontal and vertical parts of the contour are bounded by, up to $N^\chi$ terms coming from the local laws,
    \[
        \frac{N^{\frac32}}{\sqrt{\vert I\vert}}
        \vert J\vert\frac{1}{N\sqrt{\vert J\vert}}
        \leqslant C,
        \qquad
        \frac{N^{\frac32}}{\sqrt{\vert I\vert}}
        \int_{\eta_0}^{\vert J\vert}\frac{\d\eta}{N\sqrt\eta}
        \leqslant C.
    \]
    At the edge, $\vert J\vert\asymp N^{-\frac23}\vert I\vert^{\frac23}$ and $\Im m_{\mathrm{sc}}(E+\i\eta)\leqslant C \vert J\vert^{1/2}$ on the relevant contour, giving similarly
    \[
        \frac{N^{\frac32}}{\sqrt{\vert I\vert}}
        \vert J\vert
        \frac{1}{N}\sqrt{\frac{\vert J\vert^{1/2}}{\vert J\vert}}
        \leqslant C,
        \qquad
        \frac{N^{\frac32}}{\sqrt{\vert I\vert}}
        \int_{\eta_0}^{\vert J\vert}
        \frac{1}{N}\sqrt{\frac{\vert J\vert^{1/2}}{\eta}}
        \d\eta
        \leqslant C.
    \]
    We note that the intermediate energy regime works similarly.
    The contribution of the averaged local law term $N^{-3/2}\eta^{-1}$ is smaller and is absorbed into $N^{\chi}$. Choosing $\chi>0$ sufficiently small, we obtain
    \[
        \sup_{\substack{I\subset\unn{1}{N}\,\mathrm{interval}\\
        N^\varepsilon\leqslant |I|\leqslant N^{1-\varepsilon}}}
        \frac{1}{\sqrt{|I|}}
        \left|
        \sum_{j\in I}
        \left(
            \v_j(s,\eb)\v_j(s,\eb')-\scp{\eb}{\eb'}
        \right)
        \right|
        \leqslant N^{\xi/2}
    \]
    with overwhelming probability. A union bound over at most $N^2$ intervals and at most $N^4$ pairs $\eb,\eb'\in\Efr$ keeps the estimate overwhelming.

    It remains to treat microscopic and macroscopic intervals. If $|I|\leqslant N^\varepsilon$, then by isotropic delocalization,
    \[
        \left|
            \v_j(s,\eb)\v_j(s,\eb')-\scp{\eb}{\eb'}
        \right|
        \leqslant N^{\xi/2}
    \]
    uniformly in $j,\eb,\eb'$, with overwhelming probability, provided $\varepsilon$ is chosen small enough. Hence
    \[
        \frac{1}{\sqrt{|I|}}
        \left|
        \sum_{j\in I}
        \left(
            \v_j(s,\eb)\v_j(s,\eb')-\scp{\eb}{\eb'}
        \right)
        \right|
        \leqslant
        N^{\xi/2}|I|^{1/2}
        \leqslant N^\xi .
    \]
    If $|I|\geqslant N^{1-\varepsilon}$, we partition $I$ into consecutive intervals $I_\alpha$ of size between $N^\varepsilon$ and $N^{1-\varepsilon}$, up to one remainder of size at most $N^\varepsilon$. The number of intervals is at most $N^{2\varepsilon}$. Using the mesoscopic estimate on the intervals $I_\alpha$, the microscopic estimate on the possible remainder, and Cauchy--Schwarz, we get
    \[
    \begin{aligned}
        \frac{1}{\sqrt{|I|}}
        \left|
        \sum_{j\in I}
        \left(
            \v_j(s,\eb)\v_j(s,\eb')-\scp{\eb}{\eb'}
        \right)
        \right|
        \leqslant
        \frac{N^{\xi/2}}{\sqrt{|I|}}
        \sum_\alpha \sqrt{|I_\alpha|}
        +
        N^\xi |I|^{-\frac12}  
        \leqslant
        N^{\xi/2}N^\varepsilon+N^\xi |I|^{-\frac12}
        \leqslant N^\xi,
    \end{aligned}
    \]
    after decreasing $\varepsilon$ if necessary. This proves $\Afr_{2,s}$ with overwhelming probability for fixed $s$.

    The estimates above are uniform for $s$ in a polynomial net of $[0,1]$. The extension from the net to all $s\in[0,1]$ follows by the same H\"older regularity argument for the resolvent as in \cite{bourgade2017eigenvector}*{Lemma 4.2}. This gives
    \[
        \Pds\left(\Afr(\xi,\Qfr,\Efr,N)\right)\geqslant 1-N^{-D}
    \]
    for any $D>0$, after choosing the overwhelming-probability exponents in the local laws sufficiently large. The same time-net argument gives the time-uniform rigidity estimate
    \[
        \Pds\left(\Afr_3(\xi,N)\right)\geqslant 1-N^{-D}.
    \]

    Finally, the estimate on $\Lfr(\xi,\Qfr,\Efr,\nu,N)$ follows from Markov's inequality. Indeed,
    \[
    \begin{aligned}
    \Pds\left(\Lfr(\xi,\Qfr,\Efr,\nu,N)^c\right)
    &\leqslant
    \Pds\left(\Afr_3(\xi,N)^c\right)
    +
    \Pds\left(
        \Pds\left(\Afr(\xi,\Qfr,\Efr,N)^c\vert \bm{\lambda}\right)>N^{-\nu}
    \right)\\
    &\leqslant
    \Pds\left(\Afr_3(\xi,N)^c\right)
    +
    N^\nu\Pds\left(\Afr(\xi,\Qfr,\Efr,N)^c\right).
    \end{aligned}
    \]
    Applying the previous overwhelming-probability estimates with exponent $D+\nu$ gives
    \[
        \Pds\left(\Lfr(\xi,\Qfr,\Efr,\nu,N)\right)\geqslant 1-N^{-D}.
    \]
      If $\bm{\lambda}\in\Lfr(\xi,\Qfr,\Efr,\nu,N)$, then 
    $\{\tau < 1\}=\Afr(\xi,\Qfr,\Efr,N)^c$, and therefore
    \[
        \Pds(\tau < 1\vert \bm{\lambda})
        \leqslant N^{-\nu}.
    \]   
\end{proof}
We have the following consequence of rigidity estimates. Since this is classical, we omit the proof.
\begin{lemma}\label{lem:rigidity}
  Let $\xi,\nu>0$ and $N^{2\xi}\leqslant \ell \leqslant N^{1-\xi}$. Suppose that $\bm{\lambda}\in \Lfr(\xi,\Qfr,\Efr,\nu,N)$. There exist $c,C>0$, such that
  \[
    c\left(\frac{N}{\ell}\right)^{\frac{1}{3}}\leqslant \inf_{t\in[0,1]}\inf_{i\in\unn{1}{N}}\beta_i(t) \leqslant \sup_{t\in[0,1]}\sup_{i\in\unn{1}{N}}\beta_i(t)  \leqslant C\frac{N}{\ell} 
  \]
  where $\beta_i(t)\coloneqq \sum_{j,\vert i-j\vert \geqslant \ell} c_{ij}(t)$.
\end{lemma}
Let $p\geqslant 1$, let $k_1,\dots,k_p\in\unn{1}{N}$ be distinct indices, let $a_1,\dots,a_p\geqslant 1$ and set $Q=\sum_{i=1}^p a_i$. For every $1\leqslant i\leqslant p$, let $(\q_{i\alpha})_{1\leqslant \alpha\leqslant a_i}$ be a family of deterministic unit vectors in $\Rbb^N$. We set
    \[
    \Jcal=\left\{(i,\alpha),\,1\leqslant i\leqslant p,\ 1\leqslant \alpha\leqslant a_i\right\},
    \qquad
    \q_{(i,\alpha)}=\q_{i\alpha},
    \qquad
    \Qfr=(\q_a)_{a\in\Jcal}.
    \]
    We denote
    \[
    \v_k(t,\q_{i\alpha})=\scp{\q_{i\alpha}}{\v_k(t)}.
    \]
    For each $1\leqslant i\leqslant p$, we define the Gram matrix
    \[
    \Gamma_i=\left(\scp{\q_{i\alpha}}{\q_{i\beta}}\right)_{1\leqslant \alpha,\beta\leqslant a_i}.
    \]
    Set 
    \[
    E \coloneqq \mathrm{Span}\{
        \q_{i\alpha}\text{ for }(i,\alpha)\in\Jcal
    \}\quad\text{and}\quad  
    R \coloneqq \dim(E)\leqslant Q.
    \]
    We fix an orthonormal basis $\Efr \coloneqq \{\eb_1,\dots,\eb_R\}$ of $E$. We have the following quantitative universality result for the Dyson vector flow.
\begin{theorem}\label{theo:maindyn}
    Let $\xi>0$, $\nu>0$ and $N^{2\xi}\leqslant \ell \leqslant N^{1-\xi}$. Let $h:\Rbb^Q\to\Rbb$ be a bounded $\Ccal^2(\Rbb^Q)$ function, which can depend on $N$. Then for   $t\in[0,1)$    such that $\frac{Nt^3}{\ell}\geqslant c_0$, we have for every $\bm{\lambda} \in \Lfr(\xi,\Qfr,\Efr,\nu,N)$,
    \begin{multline*}
    \left\vert
    \Eds\left[
        h\left(\left(\v_{k_i}(t,\q_{i\alpha})\right)_{(i,\alpha)\in\Jcal}\right)\middle\vert \bm{\lambda}
    \right]
    -
    \Eds\left[
        h(G)
    \right]
    \right\vert
    \\\leqslant  
    C\Vert \nabla^2 h\Vert_\infty \frac{N^{1+\xi}t}{\ell}
    \left(
        \frac{RQ}{\sqrt{\ell}}+\frac{Q^2}{\ell}
    \right)
    +
     C\sqrt{R}\Vert h\Vert_\infty \sqrt{N}\exp\left(-c\left(\frac{Nt^3}{\ell}\right)^\frac{1}{3}\right)
    +C\Vert h\Vert_\infty N^{-\nu}
    \end{multline*}
\noindent where $\Vert \nabla^2 h\Vert_\infty=\sup_{x\in\Rbb^Q}\Vert \nabla^2 h(x)\Vert_{\mathrm{op}}$ and $G=(G_{i\alpha})_{1\leqslant i\leqslant p,\,1\leqslant \alpha\leqslant a_i}$ is the centered Gaussian vector in $\Rbb^Q$ with covariance
\[
\Eds\left[G_{i\alpha}G_{j\beta}\right]
=
    \delta_{ij}\scp{\q_{i\alpha}}{\q_{j\beta}}.    
\]
\end{theorem}

For each $i\in\unn{1}{p}$, there exists a matrix $O_i\in \Rbb^{a_i\times R}$ such that 
\[
\q_{i\alpha} = \sum_{k=1}^R (O_i)_{\alpha k}\eb_k\quad\text{for}\quad \alpha\in\unn{1}{a_i}.
\]
By definition of the Gram matrix $\Gamma_i$, we have that  $O_iO_i^\top=\Gamma_i.$
Moreover, since each $\q_{i\alpha}$ is a unit vector, we have
\[\Vert O_i\Vert_{\mathrm{HS}}^2=a_i,
\qquad
\Vert O_i\Vert_{\mathrm{op}}^2=\Vert \Gamma_i\Vert_{\mathrm{op}}\leqslant a_i.
\]
For each $j\in\unn{1}{N}$, we define the vector 
\[
\z_j(t) \coloneqq (\v_j(t,\eb_m))_{m=1}^R\in  \Rbb^R
\]
so that we have $(\v_{k_i}(t,\q_{i\alpha}))_{\substack{1\leqslant i\leqslant p\\1\leqslant \alpha\leqslant a_i}}
=\left(
    O_i\z_{k_i}(t)
\right)_{1\leqslant i\leqslant p}.$
We define the function 
\[
\tilde{h}(\x_1,\dots,\x_p) \coloneqq  h(O_1\x_1,\dots,O_p\x_p)\quad\text{for}\quad \x_i\in \Rbb^R
\]
so that we have $h\left(
    (\v_{k_i}(t,\q_{i\alpha}))_{(i,\alpha)\in\Jcal}  
\right)
=
\tilde{h}\left(
    \z_{k_1}(t),\dots,\z_{k_p}(t)
\right)$
and by construction we have the bound $\Vert \tilde{h}\Vert_\infty \leqslant \Vert h\Vert_\infty$. If $Z_1,\dots,Z_p$ are i.i.d $\mathcal{N}(0,\Id_R)$, then the vector $(O_iZ_i)_{1\leqslant i\leqslant p}$ is centered Gaussian and satisfies
\[
\Eds\left[(O_iZ_i)_\alpha (O_jZ_j)_\beta\right]
=
\delta_{ij}(O_iO_j^\top)_{\alpha\beta}
=
\delta_{ij}\scp{\q_{i\alpha}}{\q_{j\beta}}.   
\]
Therefore $(O_iZ_i)_{1\leqslant i\leqslant p}$ has the same law as $G$, and we have $\Eds\left[\tilde{h}(Z_1,\dots,Z_p)\right] = \Eds[h(G)].$
We thus have that it suffices to prove the theorem for the reduced observable $\tilde{h}\left(\z_{k_1},\dots,\z_{k_p}\right)$.

\subsection{Decomposition of the flow}
We define the deterministic observable $\F_{\k,\q}:\Rbb^{N\times N}\to\Rbb$
\[
\F_{\k,\q}(V)=\tilde{h}\left(\left(\langle \v_{k_i},\eb_m\rangle\right)_{\substack{1\leqslant i\leqslant p\\1\leqslant m\leqslant R}}\right)
=
\tilde{h}(\z_{k_1},\dots, \z_{k_p})
\]
where we denote similarly as above $\z_{k_i} = (\langle \v_{k_i},\eb_m\rangle)_{m=1}^R\in \Rbb^R$. 
We decompose $\mathscr{L}_t$ into its long-range and short-range parts
\[
\Lscr_t = \Lscr^L_t + \Lscr^S_t 
\]
with 
\[
\Lscr^S_t = \frac{1}{2}\sum_{\vert i-j\vert <\ell}c_{ij}(t)\Xscr_{ij}^2\quad\text{and}\quad 
\Lscr^L_t = \frac{1}{2}\sum_{\vert i-j\vert \geqslant \ell}c_{ij}(t)\Xscr_{ij}^2.
\]
If $g=g((\z_i)_{1\leqslant i\leqslant N})$ is a smooth function on $(\Rbb^R)^N$, we denote by $\nabla_i g\in\Rbb^R$ its gradient with respect to $\z_i$, by $\nabla_i^2 g\in\Rbb^{R\times R}$ the corresponding Hessian, by $\nabla_{ij}^2 g\in\Rbb^{R\times R}$ the mixed Hessian and by $\Delta_i g=\Tr(\nabla_i^2 g)$ the Laplacian in the variable $\z_i$. We have that
\[
\Xscr_{ij}^2 g
=
\z_j^\top \nabla_i^2 g\,\z_j + \z_i^\top \nabla_j^2 g\,\z_i -   2\z_j^\top \nabla_{ij}^2 g\,\z_i    - \z_i\cdot\nabla_i g - \z_j\cdot\nabla_j g.
\]
We can therefore rewrite the contribution of the long-range part as
\[
\mathscr{X}_{ij}^2g      
=
\left(\Delta_i-\z_i\cdot\nabla_i\right)g+\left(\Delta_j-\z_j\cdot\nabla_j\right)g
+
\left\langle \z_j\z_j^\top-\Id_R,\nabla_i^2 g\right\rangle_{\mathrm{HS}}    
+
\left\langle \z_i\z_i^\top-\Id_R,\nabla_j^2 g\right\rangle_{\mathrm{HS}}
    -2\z_j^\top\nabla_{ij}^2g\,\z_i.  
\]
We are replacing the long-range part by a generator of an Ornstein--Uhlenbeck process plus some error terms at the cost of controlling interval sums of $\z_i\z_i^\top-\Id_R$ and decorrelations between $\z_i$ and $\z_j$ for $i\neq j$.
  
This suggests writing the full generator as the sum of a reference generator with a remainder term. 
\begin{lemma}\label{lem:decomp_flow}
      We have    
    \[
    \Lscr_t = \hat{\Lscr}_t + \Dscr_t - \Oscr_t,
    \]
      where,    for any smooth function $g$ of the variables $(\z_i)_{1\leqslant i\leqslant N}\in(\Rbb^R)^N$
    \[
    \begin{gathered}
    \hat{\Lscr}_t g = \frac{1}{2}\sum_{\vert i-j\vert <\ell} c_{ij}(t)\Xscr_{ij}^2 g + \sum_{i=1}^N\beta_i(t)\left(
        \Delta_i-\z_i\cdot\nabla_i
    \right)g,\quad \beta_i(t) = \sum_{j,\vert i-j\vert \geqslant \ell} c_{ij}(t)\\
    \Dscr_tg=\sum_{\vert i-j\vert \geqslant \ell}
    c_{ij}(t)\left\langle \z_i\z_i^\top-\Id_R,\nabla_j^2 g\right\rangle_{\mathrm{HS}} 
    \quad\text{and}\quad 
    \Oscr_tg = \sum_{\vert i-j\vert \geqslant \ell} c_{ij}(t)    \z_j^\top\nabla_{ij}^2g\,\z_i   
    \end{gathered}
    \]
\end{lemma}
\begin{proof}
    By the previous identity, after symmetrization we get
    \[
    \Lscr_t g = \sum_{i\neq j} c_{ij}(t)\left\langle \z_i\z_i^\top,\nabla_j^2 g\right\rangle_{\mathrm{HS}} -\sum_{i\neq j} c_{ij}(t)    \z_j^\top\nabla_{ij}^2g\,\z_i    - \sum_{i=1}^N\left(\sum_{j\neq i} c_{ij}(t)\right)\z_i\cdot\nabla_i g .
    \]
    If we denote
    \[
    \hat{\Lscr}_t = \frac{1}{2}\sum_{\vert i-j\vert <\ell} c_{ij}(t)\Xscr_{ij}^2 + \sum_{i=1}^N\beta_i(t)\left(
        \Delta_i-\z_i\cdot\nabla_i
    \right)
    \]
    and using $\beta_i(t) = \sum_{j,\vert i-j\vert \geqslant \ell} c_{ij}(t)$ we get the final result.
\end{proof}
If we denote $\Pscr_{s,t}$ the propagator of $\Lscr_t$ and similarly $\hat{\Pscr}_{s,t}$ the propagator of $\hat{\Lscr}_t$ then we can write
\begin{equation}\label{eq:decomp_flow}
\left\vert
\Eds\left[
    \F_{\mathbf{k},\mathbf{q}}(V_t)\middle\vert \bm{\lambda}
\right]
-
\Eds\left[
    h(G)
\right]
\right\vert
\leqslant    
\left\vert
\Eds\left[
    \Pscr_{0,t}\F_{\mathbf{k},\mathbf{q}}(V_0)\middle\vert \bm{\lambda}
\right]
-
\Eds\left[
    \hat{\Pscr}_{0,t}\F_{\mathbf{k},\mathbf{q}}(V_0)\middle\vert \bm{\lambda}
\right]
\right\vert
+
\left\vert
\Eds\left[
    \hat{\Pscr}_{0,t}\F_{\mathbf{k},\mathbf{q}}(V_0)\middle\vert \bm{\lambda}
\right]
-
\Eds\left[
    h(G)
\right]
\right\vert
\end{equation}

\subsection{Analysis of the reference generator} We now consider only the reference generator $\hat{\Lscr}_t$. 
In this subsection, we are going to focus on the second term of the decomposition above. The first term will be treated in the next subsection as a perturbation of the reference dynamics. We first need to understand the invariant measure of $\hat{\Lscr}_t$. We have the following lemma.
\begin{lemma}
    The measure 
    \[
    \phi_{N,R}(\d \z) = \prod_{i=1}^N \frac{1}{(2\pi)^{\frac{R}{2}}}\e^{-\frac{\Vert \z_i\Vert^2}{2}}\d \z_i
    \]
    is invariant for the dynamics generated by $\hat{\Lscr}_t$.
\end{lemma}
\begin{proof} 
For the part involving $\Xscr_{ij}^2$, since the $c_{ij}(t)$ do not depend on $\z$, we note that since $\phi_{N,R}$ has a density only involving $\sum_{j=1}^N\Vert \z_j\Vert^2$ we have $\Xscr_{ij}\phi_{N,R}=0$ since 
\[
\Xscr_{ij}\left(
    \sum_{k=1}^N\Vert \z_k\Vert^2
\right)
=
2\z_i\cdot\z_j - 2\z_j\cdot\z_i = 0.
\]
We note also that we can write $\Xscr_{ij} = \b_{ij}\cdot \nabla$ with $\b_{ij}$ a vector field such that $\mathrm{div}(\b_{ij})=0$ so that we can integrate by parts to get
\[
\int \Xscr_{ij} f \, \phi_{N,R}(\d \z) = 
-\int f \mathrm{div}(\b_{ij}\phi_{N,R})\d \z =
-\int f \mathrm{div}(\b_{ij})\phi_{N,R}\d \z - \int f (\b_{ij}\cdot\nabla) \phi_{N,R}\d \z =0 
\]
for any smooth function $f$. If now we replace $f$ by $\Xscr_{ij} g$ for some smooth function $g$, we get
\[
\int \Xscr_{ij}^2 g \, \phi_{N,R}(\d \z) = 0.
\]
For the second part of the generator, we can simply use Gaussian integration by parts to get
\[
\int \left(\Delta_i-\z_i\cdot\nabla_i\right) f \, \phi_{N,R}(\d \z) = 0.
\]
This shows that $\phi_{N,R}$ is invariant under the dynamics generated by $\hat{\Lscr}_t$.

\end{proof}
We are going to prove the following proposition. 
\begin{proposition}\label{prop:refgen}
    With our choice of parameters, conditionally on $\bm{\lambda}\in \Lfr(\xi,\Qfr,\Efr,\nu,N)$, we have 
    \[
    \left\vert
    \hat{\Pscr}_{0,{t}}\F_{\mathbf{k},\mathbf{q}}(V_0) - \Eds\left[\tilde{h}(Z_1,\dots,Z_p)\right]
    \right\vert   
    \leqslant C\sqrt{R}\Vert h\Vert_\infty\sqrt{N\left(1-\log\left({1-\e^{-c\left(\frac{N}{\ell}\right)^{\frac{1}{3}}t}}\right)\right)}\e^{-c\left(\frac{Nt^3}{\ell}\right)^{\frac{1}{3}}}.
    \]
\end{proposition}
The proof of Proposition \ref{prop:refgen} relies on the following lemma.
\begin{lemma}\label{lem:mixentropy}
    We have for a smooth function $\psi$ of $\x = (x_i)_{1\leqslant i\leqslant N}\in(\Rbb^R)^N$,
    \[
    \left\vert    
        \hat{\Pscr}_{0,t}\psi(\x)-\Eds_{\phi_{N,R}}\left[\psi(X)\right]
    \right\vert
    \leqslant \Vert \psi\Vert_\infty  \sqrt{2\Hbf_{\phi_{N,R}}(q_{0,t}(\x,\cdot))},
    \]
    where $q_{0,t}$ is the density at time $t$ of the reference flow started at $\x$ at time $0$ with respect to $\phi_{N,R}$, in particular we have 
    \[
    \hat{\Pscr}_{0,t}\psi(\x) = \int \psi(\u)q_{0,t}(\x,\u)\phi_{N,R}(\d \u)
    \]
    which exists since $\phi_{N,R}$ is invariant for the dynamics generated by $\hat{\Lscr}_t$ and $\Hbf_{\phi_{N,R}}$ is the relative entropy with respect to $\phi_{N,R}$ defined as
    \[
    \Hbf_{\phi_{N,R}}(q) = \int q(\u)\log q(\u)\phi_{N,R}(\d \u).
    \]
\end{lemma}
\begin{proof}
    We can write 
    \[
    \hat{\Pscr}_{0,t}\psi(\x) = \int \psi(\u)q_{0,t}(\x,\u)\phi_{N,R}(\d \u)
    \]
    and thus by triangle inequality,
    \[
    \left\vert    
        \hat{\Pscr}_{0,t}\psi(\x)-\Eds_{\phi_{N,R}}\left[\psi(X)\right]
    \right\vert
    =
    \left\vert
        \int \psi(\u)(q_{0,t}(\x,\u)-1)\phi_{N,R}(\d \u)
    \right\vert
    \leqslant
    \Vert \psi\Vert_\infty  \int \vert q_{0,t}(\x,\u)-1\vert\phi_{N,R}(\d \u).
    \]
    Now we can use Pinsker's inequality to get 
    \[
    \int \vert q_{0,t}(\x,\u)-1\vert\phi_{N,R}(\d \u) 
    = \Vert q_{0,t}(\x,\cdot)-1\Vert_{L^1(\phi_{N,R})} 
    \leqslant \sqrt{2\Hbf_{\phi_{N,R}}(q_{0,t}(\x,\cdot))}.
    \]
\end{proof}
The following lemma gives us an entropy decay estimate for the reference dynamics.
\begin{lemma}\label{lem:entropy_decay}
    Let $f_u$ solve the forward equation $\partial_u f_u = \hat{\Lscr}_u f_u$\footnote{We note that $\hat{\Lscr}_u$ is symmetric in $L^2(\phi_{N,R})$ so that the forward equation is driven by $\hat{\Lscr}=\hat{\Lscr}^*$.}. Then for $u\leqslant v$ we have conditionally on $\bm{\lambda}\in \Lfr(\xi,\Qfr,\Efr,\nu,N)$,
    \[
    \Hbf_{\phi_{N,R}}(f_{v})\leqslant \e^{-c\left(\frac{N(v-u)^3}{\ell}\right)^{\frac{1}{3}}}\Hbf_{\phi_{N,R}}(f_{u})
    \]
\end{lemma}
\begin{proof}
    We can compute the derivative of the entropy as 
    \[
    \frac{\d}{\d u}\Hbf_{\phi_{N,R}}(f_u) = \int \hat{\Lscr}_u f_u \log f_u \, \phi_{N,R}(\d \w) + \int \hat{\Lscr}_u f_u \, \phi_{N,R}(\d \w).
    \]
    The second term is zero since $\phi_{N,R}$ is invariant for the dynamics generated by $\hat{\Lscr}_u$. For the first term, we can use the definition of $\hat{\Lscr}_u$ to write
    \begin{align*}
    \int \hat{\Lscr}_u f_u \log f_u \, \phi_{N,R}(\d\w) &= 
      -4\sum_{\substack{i<j\\ |i-j|<\ell}} c_{ij}(u)    \int (\Xscr_{ij}\sqrt{f_u})^2\phi_{N,R}(\d\w) - 
    4\sum_{i=1}^N\beta_i(u)\int \Vert\nabla_i\sqrt{f_u}\Vert^2\phi_{N,R}(\d\w)\\
    &= -4\hat{\mathscr{E}}(\sqrt{f_u},\sqrt{f_u})
    \end{align*}
    where $\hat{\mathscr{E}}$ is the Dirichlet form associated to $\hat{\Lscr}_t$. Now we have, since $c_{ij}(t)\geqslant 0$, 
    \[
    \hat{\mathscr{E}}(f,f)
    \geqslant 
    \sum_{i=1}^N\beta_i(u)\int \Vert\nabla_i f\Vert^2\phi_{N,R}(\d\w)
    \geqslant c\left(\frac{N}{\ell}\right)^{\frac{1}{3}}\sum_{i=1}^N\int \Vert\nabla_i f\Vert^2\phi_{N,R}(\d\w)
    \geqslant \frac{1}{2}c\left(\frac{N}{\ell}\right)^{\frac{1}{3}}\Hbf_{\phi_{N,R}}(f^2)
    \]
    where we use Lemma \ref{lem:rigidity} to get the second inequality
    and we use the log-Sobolev inequality for the Gaussian measure to get the last inequality. Finally, if we apply this to $\sqrt{f_u}$ we get that 
    \[
    \frac{\d}{\d u}\Hbf_{\phi_{N,R}}(f_u)
    =
    -4\hat{\mathscr{E}}(\sqrt{f_u},\sqrt{f_u}) \leqslant -c\left(\frac{N}{\ell}\right)^{\frac{1}{3}}\Hbf_{\phi_{N,R}}(f_u)
    \]
    which finally gives, by Gronwall's lemma,
    \[
    \Hbf_{\phi_{N,R}}(f_v)\leqslant \e^{-c\left(\frac{N(v-u)^{3}}{\ell}\right)^{\frac{1}{3}}}\Hbf_{\phi_{N,R}}(f_u)
    \]
    which gives the desired result.
\end{proof}

We thus see that we need a crude bound on the entropy at a time before $t$ to get a good bound at time $t$. We have the following lemma. 
\begin{lemma}\label{lem:entropy_bound}
    We have for a path $\bm{\lambda}\in \Lfr(\xi,\Qfr,\Efr,\nu,N)$ and for any 
    $\x=(\x_i)_{1\leqslant i\leqslant N}\in(\Rbb^R)^N$ such that $\sum_{i=1}^N \Vert \x_i\Vert^2\leqslant NR,$
    \[
    \Hbf_{\phi_{N,R}}(q_{0,\frac{t}{2}}(\x,\cdot)) \leqslant CNR\left(1-\log\left({1-\e^{-c\left(\frac{N}{\ell}\right)^{\frac{1}{3}}t}}\right)\right).
    \]
\end{lemma}
\begin{proof}
    The idea is to return to the stochastic differential equation associated to the reference generator. We identify 
\(
(\Rbb^R)^N\simeq \Rbb^{N\times R}
\)
and let matrices on $\Rbb^N$ act on $\Rbb^{N\times R}$ by left multiplication. Equivalently, an $N\times N$ matrix $A$ acts as $A\otimes \Id_R$ on $\Rbb^N\otimes \Rbb^R$. For $1\leqslant i<j\leqslant N$, define the skew-symmetric matrix $A^{ij}$ by
\[
A^{ij}e_i=e_j,\qquad A^{ij}e_j=-e_i,\qquad A^{ij}e_k=0\quad\text{for }k\notin\{i,j\}.
\]
Thus, for $\x=(\x_1,\dots,\x_N)\in(\Rbb^R)^N$,
\begin{equation}\label{eq:defA}
    (A^{ij}\x)_i=-\x_j,\qquad
    (A^{ij}\x)_j=\x_i,\qquad
    (A^{ij}\x)_k=0\quad\text{for }k\notin\{i,j\}.
\end{equation}
With this convention,
\[
(A^{ij}\x)\cdot\nabla
=
-\x_j\cdot\nabla_i+\x_i\cdot\nabla_j
=
\Xscr_{ij}
\]
    and $A^{ij}$ is skew-symmetric. We then introduce an auxiliary probability space, independent of the true Dyson vector flow, with expectation denoted by $\hat{\Eds}$. The corresponding auxiliary Brownian motions are denoted by \(\hat{B}_{ij}\) and \(\hat{W}\). They should not be confused with the randomness of the true eigenvectors. If we also define the diagonal matrix $D_t$ with entries $\beta_i(t)$, then we can write the stochastic differential equation associated to $\hat{\Lscr}_t$, in the It\^o--Stratonovich form, as 
    \[
    \d \hat{X}_t 
    =
    \sum_{i<j,\vert i-j\vert < \ell} \sqrt{2c_{ij}(t)}A^{ij}\hat{X}_t\circ \d \hat{B}_{ij}(t) -D_t\hat{X}_t\d t +\sqrt{2D_t}\d \hat{W}_t,
    \]
    where $\hat{W}_t$ is a standard Brownian motion in $\Rbb^{N\times R}$ independent of the Brownian motions $\hat{B}_{ij}(t)$, and where $D_t$ and $\sqrt{D_t}$ act by left multiplication on $\Rbb^{N\times R}$. If we condition on the path of eigenvalues, and thus on $c_{ij}(t)$ and $D_t$, and on the Brownian motions $\hat{B}_{ij}(t)$, then only the additive noise $\hat{W}_t$ remains random. If we start at $\hat{X}_{0}=\x$, then this conditional law of $\hat{X}_{\frac{t}{2}}$ is Gaussian. The conditional mean is given by 
    \[
    \mu_{0,u}(\x) = \hat{\Eds}\left[\hat{X}_u\middle\vert \hat{B}_{ij},\hat{X}_{0}=\x\right] = M_{0,u}\x     
    \]
    with $M$ the solution of the equation
    \[
    \d M_{0,u}
    =
    \sum_{i<j,\vert i-j\vert < \ell}   
    \sqrt{2c_{ij}(u)}A^{ij}M_{0,u}\circ \d \hat{B}_{ij}(u) - D_uM_{0,u}\d u,\quad M_{0,0}=\Id_N.
    \]
    If we consider the centered process $Z_u = \hat{X}_u - \mu_{0,u}(\x)$, then we see that it solves the stochastic equation
    \[
    \d Z_u = \sum_{i<j,\vert i-j\vert < \ell} \sqrt{2c_{ij}(u)}A^{ij}Z_u\circ \d \hat{B}_{ij}(u) - D_uZ_u\d u +\sqrt{2D_u}\d \hat{W}_u.
    \]
    We want to control the covariance of $\hat{X}_u$. Conditionally on the short-range Brownian motions, the $R$ columns of $Z_u$ are independent and have the same covariance matrix. We denote this $N\times N$ covariance matrix by 
    \[
    \Sigma_{0,u}
    =
    \hat{\Eds}\left[Z_u^{(m)}(Z_u^{(m)})^\top \middle\vert \hat{B}_{ij}, \hat{X}_{0}=\x\right],
    \]
    where $Z_u^{(m)}$ is any fixed column of $Z_u$. If we consider the differential of $Z_u^{(m)}(Z_u^{(m)})^\top$, we get 
    \[
    \d\left(Z_u^{(m)}(Z_u^{(m)})^\top\right) 
    =
    \sum_{i<j,\vert i-j\vert < \ell}\sqrt{2c_{ij}(u)}\left(
        A^{ij}Z_u^{(m)}(Z_u^{(m)})^\top - Z_u^{(m)}(Z_u^{(m)})^\top A^{ij}
    \right)\circ \d \hat{B}_{ij}(u)
    \]
    \[
    \hspace{3cm}
    -\left(
        D_uZ_u^{(m)}(Z_u^{(m)})^\top + Z_u^{(m)}(Z_u^{(m)})^\top D_u - 2D_u
    \right)\d u  
    +\d N_u,
    \]
    with $(N_u)_u$ a matrix-valued martingale whose conditional expectation is zero. If we now take the conditional expectation, we get that $\Sigma_{0,u}$ solves the following matrix-valued differential equation
    \[
    \d \Sigma_{0,u} = 
    \sum_{i<j,\vert i-j\vert < \ell}\sqrt{2c_{ij}(u)}\left(
        A^{ij}\Sigma_{0,u} - \Sigma_{0,u} A^{ij}
    \right)\circ \d \hat{B}_{ij}(u)    
    -\left(
        D_u\Sigma_{0,u} + \Sigma_{0,u}D_u - 2D_u
    \right)\d u,
    \quad \Sigma_{0,0} = 0.
    \]
    If we now consider $S_{0,u} = M_{0,u}M_{0,u}^\top$, we similarly get 
    \[
    \d S_{0,u} 
    =
    \sum_{i<j,\vert i-j\vert < \ell}\sqrt{2c_{ij}(u)}\left(
        A^{ij}S_{0,u} - S_{0,u}A^{ij}
    \right)\circ \d \hat{B}_{ij}(u)    
    -\left(
        D_uS_{0,u} + S_{0,u}D_u
    \right)\d u,
    \quad S_{0,0} = \Id_N.
    \]
    Thus if we consider $\tilde{S}_{0,u}=\Id_N-S_{0,u}$, we get
    \[
    \d\tilde{S}_{0,u}
    =
    \sum_{i<j,\vert i-j\vert < \ell}\sqrt{    2c_{ij   }(u)}\left(
        A^{ij}\tilde{S}_{0,u} - \tilde{S}_{0,u}A^{ij}
    \right)\circ \d \hat{B}_{ij}(u)    
    -\left(
        D_u\tilde{S}_{0,u} + \tilde{S}_{0,u}D_u-2D_u
    \right)\d u,
    \quad \tilde{S}_{0,0} = 0.
    \]
    This is the same equation as $\Sigma_{0,u}$ with the same initial condition. By uniqueness of the solution of this equation, we get that 
    \[
    \Sigma_{0,u} = \tilde{S}_{0,u} = \Id_N-M_{0,u}M_{0,u}^\top.
    \]
    Thus we need to control $M_{0,u}$ in order to control $\Sigma_{0,u}$. We can compute for $\x\in(\Rbb^R)^N$, using the Stratonovich chain rule,
    \[
    \d \Vert M_{0,u} \x \Vert_{\mathrm{HS}}^2
    =
    2\langle M_{0,u} \x, \d (M_{0,u}\x)\rangle_{\mathrm{HS}}
    =
    -2\left\langle M_{0,u} \x, D_u M_{0,u} \x\right\rangle_{\mathrm{HS}}\d u,
    \]
    since the noise is skew-symmetric and thus does not contribute to the norm. Now we can use the lower bound on $D_u$ to get
    \[
    \frac{\d}{\d u} \Vert M_{0,u} \x \Vert_{\mathrm{HS}}^2
    \leqslant -c\left(\frac{N}{\ell}\right)^{\frac{1}{3}}\Vert M_{0,u} \x\Vert_{\mathrm{HS}}^2
    \]
    and thus by Gronwall's lemma,
    \[  
    \Vert M_{0,u} \x\Vert_{\mathrm{HS}}^2 \leqslant \e^{-c\left(\frac{N}{\ell}\right)^{\frac{1}{3}}u}\Vert \x\Vert_{\mathrm{HS}}^2.
    \]
    In particular,
    \[
    M_{0,u}M_{0,u}^\top \preceq \e^{-c\left(\frac{N}{\ell}\right)^{\frac{1}{3}}u}\Id_N
    \quad\text{and thus}\quad 
    \Sigma_{0,u} \succeq \left(1-\e^{-c\left(\frac{N}{\ell}\right)^{\frac{1}{3}}u}\right)\Id_N.
    \]
     We know that the conditional law of $\hat{X}_u$ is Gaussian with mean $\mu_{0,u}(\x)$ and covariance $\Sigma_{0,u}\otimes \Id_R$. Besides, the relative entropy of a $\mathcal{N}(\mu,\Sigma)$ with respect to $\mathcal{N}(0,\Id_d)$ is given by 
    \[
    \Hbf(\Ncal(\mu,\Sigma)\vert \Ncal(0,\Id_d))
    =
    \frac{1}{2}\left(
        \Vert \mu\Vert^2+\Tr(\Sigma)-d-\log\det \Sigma
    \right).
    \]
    Therefore, if we denote by $\tilde{q}_{0,\frac{t}{2}}(\x,\cdot)$ the conditional density with respect to the short-range Brownian motions, we get
    \[
    \Hbf_{\phi_{N,R}}(\tilde{q}_{0,\frac{t}{2}}(\x,\cdot))
    =
    \frac{1}{2}\left(
        \Vert M_{0,\frac{t}{2}}\x\Vert_{\mathrm{HS}}^2
        +R\Tr \Sigma_{0,\frac{t}{2}} -NR
        -R\log\det\Sigma_{0,\frac{t}{2}}
    \right).
    \] 
    We can use the estimates obtained earlier in the proof. In particular, by the contraction property of $M$, we have
    \[
    \Vert M_{0,\frac{t}{2}}\x\Vert_{\mathrm{HS}}^2\leqslant \Vert \x\Vert_{\mathrm{HS}}^2\leqslant NR.
    \]
    In our application, this follows from Parseval since
    \[
    \sum_{i=1}^N\Vert \z_i\Vert^2
    =
    \sum_{m=1}^R\sum_{i=1}^N \v_i(\eb_m)^2
    =
    NR.
    \]
    Now, since $\Id_N \succeq \Sigma_{0,u}$, we have $\Tr(\Sigma_{0,\frac{t}{2}})\leqslant N$. Finally, since we additionally have that 
    \[
    \Sigma_{0,u}\succeq \left(1-\e^{-c\left(\frac{N}{\ell}\right)^{\frac{1}{3}}u}\right) \Id_N,
    \]
    we get that 
    \[
    -\log\det \Sigma_{0,\frac{t}{2}} \leqslant -N\log\left({1-\e^{-c\left(\frac{N}{\ell}\right)^{\frac{1}{3}}t}}\right),
    \] 
    up to changing the value of $c>0$. This finally gives 
    \[
    \Hbf_{\phi_{N,R}}\left(
        \tilde{q}_{0,\frac{t}{2}}(\x,\cdot)
    \right)
    \leqslant \frac{NR}{2} -\frac{NR}{2}\log\left({1-\e^{-c\left(\frac{N}{\ell}\right)^{\frac{1}{3}}t}}\right) 
    \leqslant CNR\left(1-\log\left({1-\e^{-c\left(\frac{N}{\ell}\right)^{\frac{1}{3}}t}}\right)\right).
    \]
    We then finish using the convexity of the relative entropy to write, with $\tilde{\Eds}$ the expectation over the short-range Brownian motions,
    \[
    \Hbf_{\phi_{N,R}}\left(
        q_{0,\frac{t}{2}}(\x,\cdot)
    \right)
    =
    \Hbf_{\phi_{N,R}}\left(
        \tilde{\Eds}\left[
            \tilde{q}_{0,\frac{t}{2}}(\x,\cdot)
        \right]
    \right)
    \leqslant    
    \tilde{\Eds}\left[
        \Hbf_{\phi_{N,R}}\left(
            \tilde{q}_{0,\frac{t}{2}}(\x,\cdot)
        \right)
    \right]
    \leqslant CNR\left(1-\log\left({1-\e^{-c\left(\frac{N}{\ell}\right)^{\frac{1}{3}}t}}\right)\right).
    \]
\end{proof}
We are now ready to prove Proposition \ref{prop:refgen}.
\begin{proof}[Proof of Proposition \ref{prop:refgen}]
    By Lemma \ref{lem:mixentropy}, applied to the observable 
    \[
    \F_{\mathbf{k},\mathbf{q}}(\z)=\tilde{h}(\z_{k_1},\dots,\z_{k_p}),
    \]
    we have that
    \[
    \left\vert    
        \hat{\Pscr}_{0,t}\F_{\mathbf{k},\mathbf{q}}(\x) - \Eds_{\phi_{N,R}}\left[\F_{\mathbf{k},\mathbf{q}}(Z)\right]
    \right\vert
    \leqslant \Vert \tilde{h}\Vert_\infty  \sqrt{2\Hbf_{\phi_{N,R}}(q_{0,t}(\x,\cdot))}.
    \]
    Since under $\phi_{N,R}$ the random vectors $Z_{k_1},\dots,Z_{k_p}$ are independent standard Gaussian vectors in $\Rbb^R$, we have
    \[
    \Eds_{\phi_{N,R}}\left[\F_{\mathbf{k},\mathbf{q}}(Z)\right]
    =
    \Eds\left[\tilde{h}(Z_1,\dots,Z_p)\right],
    \]
    where $Z_1,\dots,Z_p$ are i.i.d. $\Ncal(0,\Id_R)$. Using Lemmas \ref{lem:entropy_decay} and \ref{lem:entropy_bound}, we get that
    \[
    \Hbf_{\phi_{N,R}}(q_{0,t}(\x,\cdot))
    \leqslant 
    \e^{-c\left(\frac{Nt^3}{8\ell}\right)^{\frac{1}{3}}}\Hbf_{\phi_{N,R}}(q_{0,\frac{t}{2}}(\x,\cdot))
    \leqslant CNR\left(1-\log\left({1-\e^{-c\left(\frac{N}{\ell}\right)^{\frac{1}{3}}t}}\right)\right)\e^{-c\left(\frac{Nt^3}{\ell}\right)^{\frac{1}{3}}}.
    \]
    Since $\Vert \tilde{h}\Vert_\infty\leqslant \Vert h\Vert_\infty$, we get the bound
    \[
    \left\vert    
        \hat{\Pscr}_{0,t}\F_{\mathbf{k},\mathbf{q}}(\x) - \Eds\left[\tilde{h}(Z_1,\dots,Z_p)\right]
    \right\vert
    \leqslant C\sqrt{R}\Vert h\Vert_\infty\sqrt{N\left(1-\log\left({1-\e^{-c\left(\frac{N}{\ell}\right)^{\frac{1}{3}}t}}\right)\right)}\e^{-c\left(\frac{Nt^3}{\ell}\right)^{\frac{1}{3}}}.
    \]
\end{proof}
\subsection{Comparison of the full dynamics with the reference dynamics} 
We now consider the first term in \eqref{eq:decomp_flow}. In this subsection, we work conditionally on a path 
\(
    \bm{\lambda}\in \Lfr(\xi,\Qfr,\Efr,\nu,N).
\)
In particular, the rigidity estimates from Lemma \ref{lem:rigidity} hold deterministically along the path $\bm{\lambda}$. Throughout this section, we use the backward flow with respect to the reference generator given in the following definition. 

\begin{definition}\label{def:backwardflow}
Fix a path \(\bm{\lambda}\in \Lfr(\xi,\Qfr,\Efr,\nu,N)\), \(t\in[0,1]\), and \(\mathbf{k}=(k_1,\dots,k_p)\), \(\a=(a_1,\dots,a_p)\) of the theorem. For \(0\le u\le t\), we define
\[
        g_u(\x):=\hat{\Pscr}_{u,t}\F_{\mathbf{k},\mathbf{q}}(\x),
\]
that is, \(g_u\) is the solution of the backward equation
\[
        \partial_u g_u+\hat{\Lscr}_u g_u=0,
        \qquad
        g_t(\x)=\F_{\mathbf{k},\mathbf{q}}(\x)
        =
        \tilde{h}(\x_{k_1},\dots,\x_{k_p}).
\]
Equivalently, \(g_u\) admits the following stochastic representation. Let
\((\hat{B}_{ij})_{i<j}\) and \(\hat{W}\) be auxiliary independent Brownian motions,
independent of the true Dyson vector flow. Let \(\hat{X}^{u,\x}_s\), \(s\in[u,t]\),
solve the reference SDE
\[
\begin{aligned}
        \d \hat{X}^{u,\x}_s
        &=
        \sum_{i<j,\,{\vert i-j|<\ell}}
        \sqrt{2c_{ij}(s)}\,A^{ij}\hat{X}^{u,\x}_s
        \circ \d \hat{B}_{ij}(s)
        -
        D_s\hat{X}^{u,\x}_s\,\d s
        +
        \sqrt{2D_s}\,\d \hat{W}_s,  \quad
        \hat{X}^{u,\x}_u=\x,
\end{aligned}
\]
where \(D_s:=\operatorname{diag}(\beta_1(s),\dots,\beta_N(s))\), acting by left multiplication on \((\Rbb^R)^N\simeq \Rbb^{N\times R}\), and \(A^{ij}\)
is defined in \eqref{eq:defA}. Then
\[
        g_u(\x)
        =
        \hat{\Eds}\left[
        \tilde{h}\left(
            (\hat{X}^{u,\x}_t)_{k_1},\dots,(\hat{X}^{u,\x}_t)_{k_p}
        \right)
        \right],
\]
where \(\hat{\Eds}\) denotes expectation only over the auxiliary
reference-flow randomness. Moreover, conditionally on the auxiliary Brownian motions
\((\hat{B}_{ij})_{i<j}\), the reference flow is affine Gaussian:
\[
        \hat{X}^{u,\x}_t
        =
        M_{u,t}\x+\Xi_{u,t}.
\]
Here \(M_{u,s}\), \(s\in[u,t]\), is an \(N\times N\) matrix acting on \((\Rbb^R)^N\) by left multiplication and solves
\[
\begin{aligned}
        \d M_{u,s}
        &=
        \sum_{i<j,|i-j|<\ell}
        \sqrt{2c_{ij}(s)}\,A^{ij}M_{u,s}
        \circ \d \hat{B}_{ij}(s)
        -
        D_sM_{u,s}\,\d s, \quad
        M_{u,u}=\mathrm{Id}.
\end{aligned}
\]
Conditionally on \((\hat{B}_{ij})_{i<j}\), \(\Xi_{u,t}\) is a centered
Gaussian vector in \((\Rbb^R)^N\), independent of \(\x\), with covariance
\[
        \Sigma_{u,t}\otimes \Id_R,
        \qquad
        \Sigma_{u,t}=I-M_{u,t}M_{u,t}^{\top}.
\]
Consequently,
\[
        g_u(\x)
        =
        \hat{\Eds}\left[
        \tilde{h}\left(
            (M_{u,t}\x+\Xi_{u,t})_{k_1},\dots,
            (M_{u,t}\x+\Xi_{u,t})_{k_p}
        \right)
        \right].
\]
\end{definition}

We emphasize that the auxiliary Brownian motions \((\hat{B}_{ij})_{i<j}\) and \(\hat{W}\) are used only to represent the deterministic function \(g_u=\hat{\Pscr}_{u,t}\F_{\mathbf{k},\mathbf{q}}\). They are independent of the true Dyson vector flow. In particular, the events \(\Afr_s(\xi,\Qfr,\Efr,N)\), \(\Afr(\xi,\Qfr,\Efr,N)\), and the stopping time \(\tau\) are defined only in terms of the true flow and do not involve the auxiliary reference randomness.

We first use the Duhamel principle to bound the difference of the two semigroups.
\begin{lemma}\label{lem:duhamel_stopped}
      For $t\in[0,1)$ and    $\bm{\lambda}\in\Lfr(\xi,\Qfr,\Efr,\nu,N)$,
    \[
    \left\vert
    \Eds\left[
        \F_{\mathbf{k},\mathbf{q}}(\z(t))\middle\vert \bm{\lambda}
    \right]
    -
    \Eds\left[
        g_0(\z(0))\middle\vert \bm{\lambda}
    \right]
    -
    \int_0^t \Eds\left[
        \mathds{1}_{\{u<\tau\}}
        \left(
            \Lscr_u-\hat{\Lscr}_u
        \right)g_u\left(
            \z(u)
        \right)\middle\vert\bm{\lambda}
    \right]
    \d u
    \right\vert
    \leqslant 2\Vert h\Vert_\infty N^{-\nu}.
    \]
    Equivalently, using Lemma \ref{lem:decomp_flow},
    \[
    \left\vert
    \Eds\left[
        \F_{\mathbf{k},\mathbf{q}}(\z(t))\middle\vert \bm{\lambda}
    \right]
    -
    \Eds\left[
        g_0(\z(0))\middle\vert \bm{\lambda}
    \right]
    -
    \int_0^t \Eds\left[
        \mathds{1}_{\{u<\tau\}}
        \left(
            \Dscr_u-\Oscr_u
        \right)g_u\left(
            \z(u)
        \right)\middle\vert\bm{\lambda}
    \right]
    \d u
    \right\vert
    \leqslant 2\Vert h\Vert_\infty N^{-\nu}.
    \]
    Here $\z(u)=(\z_i(u))_{1\leqslant i\leqslant N}$ is the projected true Dyson vector flow, namely
    \[
        \z_i(u)=\bigl(\v_i(u,\eb_m)\bigr)_{m=1}^R\in\Rbb^R.
    \]
\end{lemma}
\begin{proof}
    We consider the stopped process
    \[
        u\mapsto g_{u\wedge\tau}\bigl(\z(u\wedge\tau)\bigr),
    \]
    where $\z(u)$ is the solution to the stochastic differential equation associated to $\Lscr_u$, i.e. the true Dyson vector flow projected on the basis $\Efr$. By definition of $g_u$, we have on the event $\{u<\tau\}$ that
    \[
    \d g_u(\z(u))
    = \partial_u g_u(\z(u))\d u + \Lscr_u g_u(\z(u))\d u + \d N_u
    = (\Lscr_u-\hat{\Lscr}_u)g_u(\z(u))\d u + \d N_u,
        \]
    where $N_u$ is a martingale and we used the fact that
    \[
        (\partial_u+\hat{\Lscr}_u)g_u=0.
    \]
    Therefore,
    \[
    \d g_{u\wedge\tau}\bigl(\z(u\wedge\tau)\bigr)
    =
    \mathds{1}_{\{u<\tau\}}
    (\Lscr_u-\hat{\Lscr}_u)g_u(\z(u))\d u
    +
    \mathds{1}_{\{u<\tau\}}\d N_u.
    \]
    Integrating from $0$ to $t$ and taking conditional expectation with respect to $\bm{\lambda}$ gives
    \[
    \Eds\left[
        g_{t\wedge\tau}\bigl(\z(t\wedge\tau)\bigr)
        \middle\vert \bm{\lambda}
    \right]
    -
    \Eds\left[
        g_0(\z(0))\middle\vert \bm{\lambda}
    \right]
    =
    \int_0^t \Eds\left[
        \mathds{1}_{\{u<\tau\}}
        (\Lscr_u-\hat{\Lscr}_u)g_u(\z(u))
        \middle\vert\bm{\lambda}
    \right]\d u.
    \]
    We now compare the stopped terminal value with the unstopped terminal value. Since $g_t=\F_{\mathbf{k},\mathbf{q}}$, we have on the event $\{\tau> t\}$, 
    \[
        g_{t\wedge\tau}\bigl(\z(t\wedge\tau)\bigr)
        =
        g_t(\z(t))
        =
        \F_{\mathbf{k},\mathbf{q}}(\z(t)).
    \]
    On the complementary event $\{\tau\leqslant t\}$, we use the trivial bounds
    \[
        \left\vert \F_{\mathbf{k},\mathbf{q}}(\z(t))\right\vert
        \leqslant \Vert h\Vert_\infty,
        \qquad
        \left\vert g_{t\wedge\tau}\bigl(\z(t\wedge\tau)\bigr)\right\vert
        \leqslant \Vert h\Vert_\infty,
    \]
    the second one following from the Markov property of $\hat{\Pscr}_{u,t}$ and the bound $\Vert g_u\Vert_\infty\leqslant \Vert h\Vert_\infty$. Hence, for $\bm{\lambda}\in\Lfr(\xi,\Qfr,\Efr,\nu,N)$,
    \[
    \left\vert
    \Eds\left[
        \F_{\mathbf{k},\mathbf{q}}(\z(t))\middle\vert\bm{\lambda}
    \right]
    -
    \Eds\left[
        g_{t\wedge\tau}\bigl(\z(t\wedge\tau)\bigr)
        \middle\vert\bm{\lambda}
    \right]
    \right\vert
    \leqslant
    2\Vert h\Vert_\infty
    \Pds(\tau\leqslant t\vert\bm{\lambda})
    \leqslant
    2\Vert h\Vert_\infty N^{-\nu}.
    \]
    Combining this estimate with the stopped Duhamel identity gives the first statement. The second statement follows from the decomposition
    \[
        \Lscr_u-\hat{\Lscr}_u=\Dscr_u-\Oscr_u
    \]
    from Lemma \ref{lem:decomp_flow}.
\end{proof}
It is important to note that while $g_u$ is defined from the reference flow, we evaluate it on the full projected flow $\z(u)$ and not along the reference flow. This is crucial because we only have control of the actual eigenvectors along the exact Dyson vector flow and not on the reference flow. The stopping time $\tau$ will ensure that whenever the error terms are evaluated, the estimates from the good event are available.

We are now going to bound the diagonal term in Lemma \ref{lem:decomp_flow}.
\begin{lemma}\label{lem:diagonal_term}
    Let $\xi,\nu>0$ and $N^{2\xi}\leqslant \ell \leqslant N^{1-\xi}$. We have that for $\bm{\lambda}\in\Lfr(\xi,\Qfr,\Efr,\nu,N)$,
    \[
    \left\vert \int_{0}^t\Eds\left[\mathds{1}_{\{u<\tau\}}\Dscr_ug_u(\z(u))\mid \bm{\lambda}\right]\d u \right\vert 
    \leqslant C RQ\Vert \nabla^2 h\Vert_\infty\frac{N^{1+\xi}t}{\ell^{\frac{3}{2}}}.
   \]
\end{lemma}
\begin{proof}
    We recall that we have 
    \[
    \Dscr_ug_u(\z(u))
    =
    \sum_{\vert i-j\vert \geqslant \ell}
    c_{ij}(u)\left\langle \z_i(u)\z_i(u)^\top-\Id_R,\nabla_j^2g_u(\z(u))\right\rangle_{\mathrm{HS}}.
    \]
    We therefore define
    \[
    \Delta_j(u)
    =
    \sum_{i,\vert i-j\vert \geqslant \ell}
    c_{ij}(u)\left(\z_i(u)\z_i(u)^\top-\Id_R\right)
    \]
    so that
    \[
    \vert \Dscr_ug_u(\z(u))\vert
    \leqslant
    \sup_{j\in\unn{1}{N}}\Vert \Delta_j(u)\Vert_{\mathrm{HS}}
    \sum_{j=1}^N\Vert \nabla_j^2g_u(\z(u))\Vert_{\mathrm{HS}}.
    \]
    To control the first term on the event $\{u<\tau\}$, we control each matrix entry of $\Delta_j(u)$. For $a,b\in\unn{1}{R}$, we have
    \[
    \left(\Delta_j(u)\right)_{ab}
    =
    \sum_{i,\vert i-j\vert \geqslant \ell}
    c_{ij}(u)\left(\v_i(u,\eb_a)\v_i(u,\eb_b)-\delta_{ab}\right).
    \]
    We perform a dyadic decomposition and define 
    \[
    A_m(j) = \{i,\,2^m\ell\leqslant \vert i-j\vert<  2^{m+1}\ell\}.
    \]
    We note that $A_m(j)$ is the union of two integer intervals of size at most $2^m\ell$ and that on each interval, since $j$ is fixed and the sum is over $i$, $c_{ij}(u)$ is a monotone sequence. Thus we can write $A_m(j)$ as the union of two sets $A_m^+(j)$ and $A_m^-(j)$ defined as
\[
A_m^+(j) = \{i,\, 2^m\ell \leqslant i-j\leqslant 2^{m+1}\ell\}\quad\text{and}\quad  
A_m^-(j) = \{i,\, 2^m\ell \leqslant j-i\leqslant 2^{m+1}\ell\}.
\]
Near the spectral edges we implicitly intersect these intervals with \(\unn{1}{N}\); this only shortens the intervals and does not affect the estimate.
Thus we can use a discrete integration by parts to write, for the interval $A_m^+(j)$,
\begin{align*}
\left\vert\sum_{i=j+2^m\ell}^{j+2^{m+1}\ell}c_{ij}(u)\left(\v_i(u,\eb_a)\v_i(u,\eb_b)-\delta_{ab}\right)\right\vert
\leqslant
C\max_{i\in A_m^+(j)}c_{ij}(u)\max_{n\in A_m^+(j)}\left\vert\sum_{k=j+2^m\ell}^{n} \left(\v_k(u,\eb_a)\v_k(u,\eb_b)-\delta_{ab}\right)\right\vert.
\end{align*}
The same inequality holds for the sum over $A_m^-(j)$. Now, on the event $\{u<\tau\}$, we can use the quantum ergodicity estimate from Definition \ref{def:goodevent} to say that 
\[
\max_{n\in A_m^+(j)}\left\vert\sum_{k=j+2^m\ell}^{n} \left(\v_k(u,\eb_a)\v_k(u,\eb_b)-\delta_{ab}\right)\right\vert
\mathds{1}_{\{u<\tau\}} \leqslant N^\xi2^{\frac{m}{2}}\sqrt{\ell}.
\]
Using eigenvalue rigidity since $\bm{\lambda}\in\Lfr(\xi,\Qfr,\Efr,\nu,N)$, we can also bound, for $\vert i-j\vert \geqslant \ell$, $c_{ij}(u)\leqslant C\frac{N}{\vert i-j\vert^2}$. Finally, we get the bound, uniform in $j\in \unn{1}{N}$, $u\in[0,1]$, and $a,b\in\unn{1}{R}$,
\[
\left\vert 
    \sum_{i,\vert i-j\vert \geqslant \ell} c_{ij}(u)
    \left(\v_i(u,\eb_a)\v_i(u,\eb_b)-\delta_{ab}\right)
\right\vert \mathds{1}_{\{u<\tau\}}
\leqslant 
C\sum_{m\geqslant 0} \frac{N}{2^{2m}\ell^2} N^{\xi} 2^{\frac{m}{2}}\sqrt{\ell} 
\leqslant 
CN^{1+\xi} \ell^{-\frac{3}{2}}.
\]
Consequently,
\[
\sup_{j\in\unn{1}{N}}\Vert \Delta_j(u)\Vert_{\mathrm{HS}}\mathds{1}_{\{u<\tau\}}
\leqslant C R N^{1+\xi}\ell^{-\frac{3}{2}}.
\]

We now need the $\ell^1$ bound 
\[
\sum_{j=1}^N \Vert \nabla_j^2 g_u(\z(u))\Vert_{\mathrm{HS}}
\leqslant Q\Vert \nabla^2 h\Vert_\infty.
\]
Using Definition \ref{def:backwardflow}, we can write
\[
g_u(\x)
=
\hat{\Eds}\left[
h\left(
\left(O_i(M_{u,t}\x+\Xi_{u,t})_{k_i}\right)_{1\leqslant i\leqslant p}
\right)
\right].
\]
For each $j\in\unn{1}{N}$, define the linear map $T_j(u):\Rbb^R\to\Rbb^Q$ by
\[
T_j(u)\y
=
\left(
    (M_{u,t})_{k_i j}O_i\y
\right)_{1\leqslant i\leqslant p}.
\]
Equivalently, $T_j(u)$ is the $Q\times R$ block matrix whose $i$-th block is $(M_{u,t})_{k_i j}O_i\in\Rbb^{a_i\times R}$.
Then, by the chain rule,
\[
\nabla_j^2 g_u(\x)
=
\hat{\Eds}\left[
T_j(u)^\top
\nabla^2 h\left(
\left(O_i(M_{u,t}\x+\Xi_{u,t})_{k_i}\right)_{1\leqslant i\leqslant p}
\right)
T_j(u)
\right].
\]
Thus
\[
\Vert \nabla_j^2 g_u(\x)\Vert_{\mathrm{HS}}
\leqslant
\Vert \nabla^2 h\Vert_\infty
\hat{\Eds}\!\left[\Vert T_j(u)\Vert_{\mathrm{HS}}^2\right].  
\]
Moreover,
\[
\Vert T_j(u)\Vert_{\mathrm{HS}}^2
=
\sum_{i=1}^p (M_{u,t})_{k_i j}^2\Vert O_i\Vert_{\mathrm{HS}}^2
=
\sum_{i=1}^p a_i (M_{u,t})_{k_i j}^2.
\]
For $Y_s = M_{u,s} \x$ we have
\[
\d \Vert Y_s\Vert_2^2 = -2\left\langle Y_s, D_sY_s\right\rangle \d s \leqslant 0
\]
since $A^{ij}$ is skew-symmetric and does not contribute to the norm. Thus we get that $\Vert M_{u,s} \x\Vert^2 \leqslant \Vert \x\Vert^2$ and therefore $\Vert M_{u,s}\Vert_\mathrm{op}\leqslant 1$. This finally gives that 
\begin{equation}\label{eq:boundMus}
\sum_{j=1}^N (M_{u,t})_{k_i j}^2  = \Vert M_{u,t}^\top \mathbf{e}_{k_i}\Vert^2 \leqslant 1.
\end{equation}
Therefore,
\[
\sum_{j=1}^N \Vert T_j(u)\Vert_{\mathrm{HS}}^2
=
\sum_{i=1}^p a_i\sum_{j=1}^N (M_{u,t})_{k_i j}^2
\leqslant
\sum_{i=1}^p a_i
=
Q,
\]
and hence
\[
\sum_{j=1}^N \Vert \nabla_j^2 g_u(\z(u))\Vert_{\mathrm{HS}}
\leqslant Q\Vert \nabla^2 h\Vert_\infty.
\]
Combining the two estimates, we get
\[
\left\vert
\Eds\left[
\mathds{1}_{\{u<\tau\}}\Dscr_ug_u(\z(u))\middle\vert \bm{\lambda}
\right]
\right\vert
\leqslant
C RQ\Vert \nabla^2 h\Vert_\infty N^{1+\xi}\ell^{-\frac{3}{2}}.
\]
Integrating over $u\in[0,t]$ gives
\[
\left\vert \int_{0}^t\Eds\left[\mathds{1}_{\{u<\tau\}}\Dscr_ug_u(\z(u))\mid \bm{\lambda}\right]\d u \right\vert 
\leqslant C RQ\Vert \nabla^2 h\Vert_\infty\frac{N^{1+\xi}t}{\ell^{\frac{3}{2}}}
\]
as desired.
\end{proof}
The following lemma controls the off-diagonal term in Lemma \ref{lem:decomp_flow}.
\begin{lemma}\label{lem:offdiagonal_term}
   Let $\xi,\nu>0$ and $N^{2\xi}\leqslant \ell \leqslant N^{1-\xi}$. We have that for $\bm{\lambda}\in \Lfr(\xi,\Qfr,\Efr,\nu,N)$, 
    \[
    \left\vert\int_{0}^t \Eds\left[ \mathds{1}_{\{u<\tau\}}\Oscr_u g_u(\z(u))\middle\vert \bm{\lambda}\right] \d u \right\vert
    \leqslant C Q^2\frac{N^{1+\xi}t}{\ell^2}\Vert \nabla^2 h\Vert_\infty.
    \]
\end{lemma}
\begin{proof}
    We recall that we have 
    \[
    \Oscr_ug_u(\z(u))   
    =
    \sum_{\vert i-j\vert \geqslant \ell} c_{ij}(u)    \z_j(u)^\top\nabla_{ij}^2g_u(\z(u))\z_i(u)  .
    \]
    Using Definition \ref{def:backwardflow}, we can compute the mixed second derivatives of $g_u$ in the following way. Recall that for each $j\in\unn{1}{N}$, the linear map $T_j(u):\Rbb^R\to\Rbb^Q$ is given by
    \[
        T_j(u)\y
        =
        \left(
            (M_{u,t})_{k_i j}O_i\y
        \right)_{1\leqslant i\leqslant p}.
    \]
    Thus
    \[
    \nabla_{ij}^2 g_u(\x)
    =
    \hat{\Eds}\left[
        T_i(u)^\top
        \nabla^2h\left(
            \left(O_r(M_{u,t}\x+\Xi_{u,t})_{k_r}\right)_{1\leqslant r\leqslant p}
        \right)
        T_j(u)
    \right].
    \]
    Consequently,
    \[
    \Oscr_ug_u(\z(u))
    =
    \hat{\Eds}\left[
        \sum_{\vert i-j\vert \geqslant \ell} c_{ij}(u)
          \left(T_i(u)\z_j(u)\right)^\top
        \nabla^2h(Y_{u,t})
        \left(T_j(u)\z_i(u)\right)    
    \right],
    \]
    where we have set
    \[
        Y_{u,t}
        =
        \left(O_r(M_{u,t}\z(u)+\Xi_{u,t})_{k_r}\right)_{1\leqslant r\leqslant p}
        \in\Rbb^Q.
    \]
    For $a=(r,\alpha)\in\Jcal$, define $\mu_i^a(u)=(M_{u,t})_{k_r i},$
    since
    \[
          \left(T_i(u)\z_j(u)\right)_a
        =
        \mu_i^a(u)\v_j(u,\q_a),   
    \]
    we can rewrite the off-diagonal term as
    \[
    \Oscr_ug_u(\z(u))
    =
    \hat{\Eds}\left[
        \sum_{a,b\in\Jcal}
        \partial_{ab}h(Y_{u,t})
        S_{ab}(u)
    \right],
    \]
    where
    \[
    S_{ab}(u)
    =
    \sum_{\vert i-j\vert \geqslant \ell}
      c_{ij}(u)\mu_j^a(u)\mu_i^b(u)   
    \v_i(u,\q_a)\v_j(u,\q_b).
    \]
    Hence, by Cauchy--Schwarz and Fubini's theorem,
    \[
    \left\vert    
        \Eds\left[
            \mathds{1}_{\{u<\tau\}}\Oscr_ug_u(\z(u))\middle\vert \bm{\lambda}
        \right]  
    \right\vert
    \leqslant
    \Vert \nabla^2 h\Vert_\infty
    \sum_{a,b\in\Jcal}
    \hat{\Eds}\left[
        \left(
            \Eds\left[
                \mathds{1}_{\{u<\tau\}}\vert S_{ab}(u)\vert^2
                \middle\vert \bm{\lambda}
            \right]
        \right)^{\frac{1}{2}}
    \right].
    \]
    We now fix the auxiliary reference randomness temporarily. Then the coefficients $\mu_i^a(u)$ and $\mu_j^b(u)$ are deterministic with respect to the expectation over the true Dyson vector flow and the Rademacher signs. We define
    \[
        \Theta_{ij}^{ab}(u)
        =
          c_{ij}(u)\mu_j^a(u)\mu_i^b(u)   \mathds{1}_{\{\vert i-j\vert\geqslant \ell\}}.
    \]
    If $V(u)=(\v_1(u)\vert \dots\vert \v_N(u))$ is the whole renormalized eigenbasis, then
    \[
    S_{ab}(u)
    =
    \q_a^\top V(u)\Theta^{ab}(u)V(u)^\top\q_b.
    \]
    Denote $\E_\tau[X]=\Eds\left[X\mathds{1}_{\{u<\tau\}}\middle\vert \bm{\lambda}\right].$
    We can compute
    \[
    \begin{aligned}
    \E_\tau\left[
        \left\vert \q_a^\top V(u)\Theta^{ab}(u)V(u)^\top\q_b\right\vert^2
    \right]
    &=
    \sum_{i,j,k,l=1}^N
    \E_\tau\left[
        \v_i(u,\q_a)\v_j(u,\q_b)\v_k(u,\q_a)\v_l(u,\q_b)
    \right]
    \Theta_{ij}^{ab}(u)\Theta_{kl}^{ab}(u).
    \end{aligned}
    \]
    Note that $\Theta_{ii}^{ab}(u)=0$ by definition. The event $\{u<\tau\}$ depends only on quantities which are invariant under the Rademacher signs. Using the symmetry of the Rademacher random variables, the only non-zero terms in the sum are those for which the indices $i,j,k,l$ are equal in pairs. Thus we get
    \[
    \begin{aligned}
    \E_\tau\left[
        \left\vert \q_a^\top V(u)\Theta^{ab}(u)V(u)^\top\q_b\right\vert^2
    \right]
    &=
    \sum_{i\neq j}
    \E_\tau\left[
        \v_i(u,\q_a)^2\v_j(u,\q_b)^2
    \right]
    \left(\Theta_{ij}^{ab}(u)\right)^2\\
    &\quad+
    \sum_{i\neq j}
    \E_\tau\left[
        \v_i(u,\q_a)\v_i(u,\q_b)\v_j(u,\q_a)\v_j(u,\q_b)
    \right]
    \Theta_{ij}^{ab}(u)\Theta_{ji}^{ab}(u).
    \end{aligned}
    \]
    On the event $\{u<\tau\}$, we can use the delocalization estimate from Definition \ref{def:goodevent}. Thus,  we have
    \[
    \left\vert
        \E_\tau\left[
            \v_i(u,\q_a)^2\v_j(u,\q_b)^2
        \right]
    \right\vert
    +
    \left\vert
        \E_\tau\left[
            \v_i(u,\q_a)\v_i(u,\q_b)\v_j(u,\q_a)\v_j(u,\q_b)
        \right]
    \right\vert
    \leqslant   2N^{2\xi}.   
    \]
    Hence
    \[
    \E_\tau\left[
        \left\vert \q_a^\top V(u)\Theta^{ab}(u)V(u)^\top\q_b\right\vert^2
    \right]
    \leqslant
    C N^{2\xi}
    \sum_{i\neq j}
    \left(
        \left(\Theta_{ij}^{ab}(u)\right)^2
        +
        \left\vert\Theta_{ij}^{ab}(u)\Theta_{ji}^{ab}(u)\right\vert
    \right).
    \]
    Using $2|xy|\leqslant x^2+y^2$ and symmetry of the summation, we get
    \[
    \E_\tau\left[
        \left\vert \q_a^\top V(u)\Theta^{ab}(u)V(u)^\top\q_b\right\vert^2
    \right]
    \leqslant
    C N^{2\xi}
    \sum_{\vert i-j\vert\geqslant \ell}
    c_{ij}^2(u)(\mu_i^a(u))^2(\mu_j^b(u))^2.
    \]
    Therefore,
    \[
    \left\vert    
        \Eds\left[
            \mathds{1}_{\{u<\tau\}}\Oscr_ug_u(\z(u))\middle\vert \bm{\lambda}
        \right]  
    \right\vert
    \leqslant
    C N^\xi\Vert \nabla^2 h\Vert_\infty
    \sum_{a,b\in\Jcal}
    \hat{\Eds}\left[
        \left(
            \sum_{\vert i-j\vert \geqslant \ell}
            c_{ij}^2(u)(\mu_i^a(u))^2(\mu_j^b(u))^2
        \right)^{\frac{1}{2}}
    \right].
    \]
    Now we can use the fact that, for any $a=(r,\alpha)\in\Jcal$,
    \[
    \sum_{i=1}^N (\mu_i^a(u))^2
    =
    \sum_{i=1}^N (M_{u,t})_{k_r i}^2
    =
    \Vert M_{u,t}^\top \mathbf{e}_{k_r}\Vert^2
    \leqslant 1
    \]
    from \eqref{eq:boundMus}. Using also the rigidity bound, valid since $\bm{\lambda}\in \Lfr(\xi,\Qfr,\Efr,\nu,N)$,
    \[
    c_{ij}(u)\leqslant C\frac{N}{\vert i-j\vert^2}
    \leqslant C\frac{N}{\ell^2}
    \quad\text{for }\vert i-j\vert\geqslant \ell,
    \]
    we obtain
    \[
    \sum_{\vert i-j\vert \geqslant \ell}
    c_{ij}^2(u)(\mu_i^a(u))^2(\mu_j^b(u))^2
    \leqslant
    C\frac{N^2}{\ell^4}
    \left(\sum_{i=1}^N(\mu_i^a(u))^2\right)
    \left(\sum_{j=1}^N(\mu_j^b(u))^2\right)
    \leqslant C\frac{N^2}{\ell^4}.
    \]
    This gives
    \[
    \left\vert    
        \Eds\left[
            \mathds{1}_{\{u<\tau\}}\Oscr_ug_u(\z(u))\middle\vert \bm{\lambda}
        \right]  
    \right\vert
    \leqslant
    C Q^2 \frac{N^{1+\xi}}{\ell^2}\Vert \nabla^2 h\Vert_\infty.
    \]
    Finally, integrating over $u\in[0,t]$ gives the final bound
    \[
    \left\vert\int_{0}^t \Eds\left[ \mathds{1}_{\{u<\tau\}}\Oscr_u g_u(\z(u))\middle\vert \bm{\lambda}\right] \d u \right\vert
    \leqslant C Q^2\frac{N^{1+\xi}t}{\ell^2}\Vert \nabla^2 h\Vert_\infty.
    \]
\end{proof}
We are now ready to prove Theorem \ref{theo:maindyn}.
\begin{proof}[Proof of Theorem \ref{theo:maindyn}]
     By Lemma \ref{lem:duhamel_stopped}, we have
    \begin{multline*}
    \left\vert
    \Eds\left[
        \F_{\mathbf{k},\mathbf{q}}(\z(t))\middle\vert \bm{\lambda}
    \right]
    -
    \Eds\left[
        g_0(\z(0))\middle\vert \bm{\lambda}
    \right]
    \right\vert\\
    \qquad\leqslant
    \left\vert
    \int_0^t
    \Eds\left[
        \mathds{1}_{\{u<\tau\}}\Dscr_u g_u(\z(u))
        \middle\vert \bm{\lambda}
    \right]\d u
    \right\vert
    +
    \left\vert
    \int_0^t
    \Eds\left[
        \mathds{1}_{\{u<\tau\}}\Oscr_u g_u(\z(u))
        \middle\vert \bm{\lambda}
    \right]\d u
    \right\vert
    +
    C\Vert h\Vert_\infty N^{-\nu}.
    \end{multline*}
    Applying Lemma \ref{lem:diagonal_term}, we get
    \[
    \left\vert
    \int_0^t
    \Eds\left[
        \mathds{1}_{\{u<\tau\}}\Dscr_u g_u(\z(u))
        \middle\vert \bm{\lambda}
    \right]\d u
    \right\vert
    \leqslant
    C RQ\Vert \nabla^2 h\Vert_\infty
    \frac{N^{1+\xi}t}{\ell^{\frac{3}{2}}}.
    \]
    Similarly, by Lemma \ref{lem:offdiagonal_term},
    \[
    \left\vert
    \int_0^t
    \Eds\left[
        \mathds{1}_{\{u<\tau\}}\Oscr_u g_u(\z(u))
        \middle\vert \bm{\lambda}
    \right]\d u
    \right\vert
    \leqslant
    C Q^2\Vert \nabla^2 h\Vert_\infty
    \frac{N^{1+\xi}t}{\ell^2}.
    \]
    Therefore,
    \[
    \left\vert
    \Eds\left[
        \F_{\mathbf{k},\mathbf{q}}(\z(t))\middle\vert \bm{\lambda}
    \right]
    -
    \Eds\left[
        g_0(\z(0))\middle\vert \bm{\lambda}
    \right]
    \right\vert
    \leqslant
    C\Vert \nabla^2 h\Vert_\infty
    \frac{N^{1+\xi}t}{\ell}
    \left(
        \frac{RQ}{\sqrt{\ell}}+\frac{Q^2}{\ell}
    \right)
    +
    C\Vert h\Vert_\infty N^{-\nu}.
    \]

    It remains to compare the reference flow with its invariant Gaussian distribution. By Proposition \ref{prop:refgen}, applied pointwise to the initial projected vector $\z(0)$, and then taking conditional expectation with respect to $\bm{\lambda}$, we obtain
    \[
    \left\vert
    \Eds\left[
        g_0(\z(0))\middle\vert \bm{\lambda}
    \right]
    -
    \Eds\left[\tilde{h}(Z_1,\dots,Z_p)\right]
    \right\vert
    \leqslant
    C\sqrt{R}\Vert h\Vert_\infty
    \sqrt{N\left(1-\log\left({1-\e^{-c\left(\frac{N}{\ell}\right)^{\frac{1}{3}}t}}\right)\right)}
    \e^{-c\left(\frac{Nt^3}{\ell}\right)^{\frac{1}{3}}}.
    \]
    Since $\frac{Nt^3}{\ell}\geqslant c_0$, the logarithmic factor can be absorbed into the exponential term, up to changing the constants $c,c_0,C$. Thus
    \[
    \left\vert
    \Eds\left[
        g_0(\z(0))\middle\vert \bm{\lambda}
    \right]
    -
    \Eds\left[\tilde{h}(Z_1,\dots,Z_p)\right]
    \right\vert
    \leqslant
    C\sqrt{R}\Vert h\Vert_\infty
    \sqrt{N}
    \exp\left(-c\left(\frac{Nt^3}{\ell}\right)^{\frac{1}{3}}\right).
    \]
and we get the final result by combining the two estimates.
\end{proof}
\section{Proof of main results}\label{sec:proofs}
\subsection{Eigenvector universality for generalized Wigner matrices}
The proof of Theorem \ref{theo:finite} is now a direct consequence of Theorem \ref{theo:maindyn} combined with the density property of the Dyson Brownian motion \cite{erdHos2017dynamical}*{Lemma 16.2} and the comparison scheme from \cite{knowles2013eigenvector}.

\begin{proof}[Proof of Theorem \ref{theo:finite}]
    Let $t=N^{-\theta}$ with $\theta\in(0,\frac{1}{3})$. We set $\ell=N^{1-4\theta}$,
    so that $\frac{Nt^3}{\ell}=N^\theta.$
    We will choose $\theta>0$ sufficiently small below.

    By the density of Dyson Brownian motion from \cite{erdHos2017dynamical}*{Lemma 16.2}, there exists a generalized Wigner matrix $H_0$ such that, if $H_t$ denotes the Dyson Brownian motion started from $H_0$ and   $T=a(t) H_t$,    
    then the entries of $T$ match the entries of $W$ up to order three and satisfy
    \[
    \Eds\left[
        w_{ij}^k
    \right]
    =
    \Eds\left[
        t_{ij}^k
    \right]
    \quad\text{for }k=1,2,3,
    \quad\text{and}\quad   
    \left\vert\Eds\left[
        w_{ij}^4
    \right]-\Eds\left[
        t_{ij}^4
    \right]
    \right\vert
    \leqslant CN^{-2}t.
    \]
      Since multiplication by the scalar $a(t)$ does not change eigenvectors,
the eigenvectors of $T$ are the eigenvectors of $H_t$.    

    We denote by $X_N^{(t)}$ the analogue of $X_N$ constructed from the eigenvectors of $H_t$. Let $\Qfr$ be the family of vectors appearing in the theorem, and let $\Efr$ be an orthonormal basis of their span, as in Theorem \ref{theo:maindyn}. Since $Q$ is fixed, we have $R\leqslant Q=O(1)$ and the hypotheses of Lemma \ref{lem:good_event} are satisfied.

    We apply Theorem \ref{theo:maindyn} to the Dyson vector flow started from $H_0$. For every $\bm{\lambda}\in\Lfr(\xi,\Qfr,\Efr,\nu,N)$,
    we get
    \[
    \left\vert
    \Eds\left[
        h(X_N^{(t)})\middle\vert \bm{\lambda}
    \right]
    -
    \Eds\left[h(G_N)\right]
    \right\vert
    \leqslant
    C_h\frac{N^{1+\xi}t}{\ell}
    \left(
        \frac{RQ}{\sqrt{\ell}}+\frac{Q^2}{\ell}
    \right)
    +
    C_h\sqrt{RN}\exp\left(-c\left(\frac{Nt^3}{\ell}\right)^{\frac13}\right)
    +
    C_hN^{-\nu}.
    \]
    Since $Q$ and $R$ are fixed, and since $t=N^{-\theta}$ and $\ell=N^{1-4\theta}$, the first term, and the exponential term for $N$ large enough, are bounded by
    \[
    C_h\left(
        N^{-\frac12+\xi+5\theta}
        +
        N^{-1+\xi+7\theta}
    \right).
    \]
Choosing for instance $\xi=\theta$ and $\theta=\frac{1}{13}$, we obtain
    \[
    -\frac12+\xi+5\theta
    =
    -\frac12+\frac{6}{13}
    =
    -\frac{1}{26}
    \quad\text{and}\quad 
    -1+\xi+7\theta
    =
    -1+\frac{8}{13}
    =
    -\frac{5}{13}\leqslant -\frac{1}{26}.
    \]
    Taking $\nu$ large enough, we therefore obtain, for every
    $\bm{\lambda}\in\Lfr(\xi,\Qfr,\Efr,\nu,N)$,
    \[
    \left\vert
    \Eds\left[
        h(X_N^{(t)})\middle\vert \bm{\lambda}
    \right]
    -
    \Eds\left[h(G_N)\right]
    \right\vert
    \leqslant C_hN^{-\frac{1}{26}}.
    \]
    By Lemma \ref{lem:good_event}, we know that
    \[
        \Pds\left(\Lfr(\xi,\Qfr,\Efr,\nu,N)\right)\geqslant 1-N^{-D}
    \]
    for any fixed $D>0$, by choosing the overwhelming probability exponent in Lemma \ref{lem:good_event} sufficiently large. Hence, taking $D$ large enough and using the bound $\Vert h\Vert_\infty\leqslant C_h$, we get
    \[
    \left\vert
    \Eds\left[
        h(X_N^{(t)})
    \right]
    -
    \Eds\left[h(G_N)\right]
    \right\vert
    \leqslant C_hN^{-\frac{1}{26}}.
    \]
   
It remains to compare the original matrix $W$ with the Gaussian divisible matrix $T$. We use the eigenvector comparison argument of \cite{knowles2013eigenvector}, together with its isotropic extension in \cite{bourgade2017eigenvector}*{Theorem 5.2}. Although the latter result is stated for squared eigenvector projections, this suffices after averaging over the independent eigenvector signs. The resulting observable depends only on products of pairs of projections, which are recovered from squared projections by polarization using the directions $\q_{i\alpha}\pm\q_{i\beta}$. Applying the comparison argument to this enlarged fixed family, followed by a standard smoothing approximation, only changes the  error exponent; see also \cite{knowles2013eigenvector}*{Section 4}. This gives

    \[
    \left\vert
    \Eds\left[
        h(X_N)
    \right]
    -
    \Eds\left[
        h(X_N^{(t)})
    \right]
    \right\vert
    \leqslant C_h N^{-a}
    \]
    for some $a>0$. The required eigenvalue repulsion estimates are available for generalized Wigner matrices throughout the spectrum, for example by \cite{BL22LInfinity}*{Proposition B.17}.

    Combining the previous two estimates gives
    \[
    \left\vert
    \Eds\left[
        h(X_N)
    \right]
    -
    \Eds\left[h(G_N)\right]
    \right\vert
    \leqslant
    C_hN^{-\frac{1}{26}}
    +
    C_hN^{-a}.
    \]
    This concludes the proof by setting $\varepsilon=\min\left(\frac{1}{26},a\right)>0$.
\end{proof}

\subsection{The extremal process of eigenvector entries}

We start with proving a relative version of our dynamical comparison which will be useful for events whose Gaussian probability is small. Throughout this subsection, we set
\[
    t=N^{-\frac{1}{48}},
    \qquad
    \ell=N^{\frac{11}{12}},
    \qquad
    \xi=\frac{1}{48}.
\]
For $T,w>0$, define
\[
    \chi_{T,w}(x)
    =
    \frac{1}{1+\e^{-\frac{x-T}{w}}}
    +
    \frac{1}{1+\e^{-\frac{-x-T}{w}}}.
\]
Then $0<\chi_{T,w}\leqslant 1$ and
\[
    \vert \chi_{T,w}'\vert\leqslant w^{-1}\chi_{T,w},
    \qquad
    \vert \chi_{T,w}''\vert\leqslant w^{-2}\chi_{T,w}.
\]
Consequently, for fixed $r$, defining
\[
    h(\x)=\prod_{a=1}^r\chi_{T_a,w}(x_a),
\]
we have
\begin{equation}\label{eq:self_bounding_hessian_extremal}
    \Vert\nabla^2h(\x)\Vert_{\mathrm{op}}
    \leqslant C_rw^{-2}h(\x).
\end{equation}
Moreover, if $T\to+\infty$ and $wT\to0$, then, for $g\sim\Ncal(0,1)$,
\begin{equation}\label{eq:gaussian_logistic_tail_extremal}
    \Eds[\chi_{T,w}(g)]
    =
    (1+O(wT))\Pds(\vert g\vert>T).
\end{equation}
The last estimate follows directly from the Gaussian density and Mills's ratio, which we recall is
\[
    \Pds(g>T)
    =
    \frac{\phi(T)}{T}
    \left(1+O\left(T^{-2}\right)\right),
    \qquad
    T\to+\infty.
\]

\begin{proposition}\label{prop:relative_comparison}
Let $r\geqslant1$ be fixed and let $\alpha_1,\dots,\alpha_r$ be distinct coordinates. Let $h_N:\Rbb^r\to[0,1]$ be smooth and assume that
\[
    \Vert\nabla^2h_N(\x)\Vert_{\mathrm{op}}
    \leqslant L_Nh_N(\x),
    \qquad
    L_N\leqslant C_0(\log N)^{A_0},
\]
for some fixed constants $C_0,A_0>0$.
Let
\[
    \z_i(u)=\left(\v_i(u,\eb_{\alpha_a})\right)_{a=1}^r
\]
and let $G_r\sim\Ncal(0,\Id_r)$. Then there exists $c>0$ such that, for every $D>0$, there exists $C>0$ such that 
\begin{equation}\label{eq:relative_comparison_extremal}
    \left\vert
        \Eds[h_N(\z_k(t))]-\Eds[h_N(G_r)]
    \right\vert
    \leqslant
    C \left(
        N^{-c}\Eds[h_N(G_r)]+N^{-D}
    \right),
\end{equation}
uniformly in $k$, in the choice of
$\alpha_1,\dots,\alpha_r$, and in $h_N$ satisfying the above
assumptions.
\end{proposition}

\begin{proof}
We set
\[
    \Qfr=\Efr=\{\eb_{\alpha_1},\dots,\eb_{\alpha_r}\}
\]
and fix $\bm\lambda\in\Lfr(\xi,\Qfr,\Efr,\nu,N)$. In the  representation of the backward reference flow in Definition~\ref{def:backwardflow}, we take $p=1$, $a_1=r$, $k_1=k$, $O_1=\Id_r$ and $\widetilde h=h_N$. We can write
\[
    g_u(\x)=\hat{\Eds}[h_N(Y_u)],
    \qquad
    Y_u=\sum_{i=1}^N\mu_i(u)\x_i+\Xi_u,
    \qquad
    \mu_i(u)=(M_{u,t})_{ki},
\]
where $\Xi_u=(\Xi_{u,t})_k$. By \eqref{eq:boundMus},
\begin{equation}\label{eq:relative_mu_contraction_extremal}
    \sum_{i=1}^N\mu_i(u)^2\leqslant1.
\end{equation}
We set
\[
    \Phi(u)
    =
    \Eds\left[
        g_{u\wedge\tau}(\z(u\wedge\tau))
        \middle\vert\bm\lambda
    \right].
\]
For later use below, we note that since $0\leqslant h_N\leqslant1$ and the reference evolution is Markov,
we have $0\leqslant g_u\leqslant1$ and therefore $0\leqslant\Phi(u)\leqslant1$. 
Repeating the  It\^o calculation from the proof of
Lemma~\ref{lem:duhamel_stopped} on an arbitrary interval
$[s,v]\subset[0,t]$, we obtain
\[
    \Phi(v)-\Phi(s)
    =
    \int_s^v
    \Eds\left[
        \mathds 1_{\{u<\tau\}}
        (\Dscr_u-\Oscr_u)g_u(\z(u))
        \middle|\bm\lambda
    \right]\d u.
\]
This implies that for almost every $u$,
\begin{equation}\label{eq:relative_phi_derivative_extremal}
    \Phi'(u)
    =
    \Eds\left[
        \mathds 1_{\{u<\tau\}}
        (\Dscr_u-\Oscr_u)g_u(\z(u))
        \middle|\bm\lambda
    \right].
\end{equation}

By the chain rule,
\[
    \nabla_j^2g_u(\x)
    =
    \hat{\Eds}\left[
        \mu_j(u)^2\nabla^2h_N(Y_u)
    \right].
\]
Using the assumption on $h_N$ and \eqref{eq:relative_mu_contraction_extremal}, we obtain
\[
    \sum_{j=1}^N\Vert\nabla_j^2g_u(\x)\Vert_{\mathrm{HS}}
    \leqslant
    C L_Ng_u(\x).
\]
Combining this with the bound on $\Delta_j(u)$ from the proof of Lemma \ref{lem:diagonal_term} gives
\begin{equation}\label{eq:relative_diagonal_extremal}
\left\vert
    \Eds\left[
        \mathds 1_{\{u<\tau\}}\Dscr_ug_u(\z(u))
        \middle\vert\bm\lambda
    \right]
\right\vert
\leqslant
CL_NN^{1+\xi}\ell^{-\frac{3}{2}}\Phi(u).
\end{equation}

For the off-diagonal term, we use the same decomposition as in the proof of Lemma \ref{lem:offdiagonal_term}. Namely,
\[
    \Oscr_ug_u(\z(u))
    =
    \hat{\Eds}\left[
        \sum_{a,b=1}^r
        \partial_{ab}h_N(Y_u)S_{ab}(u)
    \right],
\]
where 
\[
    S_{ab}(u)
    =
    \sum_{|i-j|\geqslant\ell}
    c_{ij}(u)\mu_i(u)\mu_j(u)
    \v_i(u,\eb_{\alpha_a})\v_j(u,\eb_{\alpha_b}).
\]
Let $\Gcal_u$ be the $\sigma$-field generated by the eigenvalue path, the
eigenvectors modulo their independent Rademacher signs, and all the
auxiliary reference-flow randomness. We condition on $\Gcal_u$ and
write $\Eds_\varepsilon$ for expectation with respect to the remaining
Rademacher signs. Since the event $\{u<\tau\}$ is invariant under these
signs, it is $\Gcal_u$-measurable. Under this conditional law,
$S_{ab}(u)$ is a degree-two polynomial in the independent Rademacher
signs. The sign-pairing computation in the proof of Lemma
\ref{lem:offdiagonal_term} gives
\[
    \mathds 1_{\{u<\tau\}}
    \Eds_\varepsilon\left[|S_{ab}(u)|^2\right]
    \leqslant
    C N^{2+2\xi}\ell^{-4}.
\]
Hence, by the Bonami--Beckner inequality \cites{Bon70,Bec75}, for every
$q\geqslant2$,
\begin{equation}\label{eq:relative_S_Lq_extremal}
    \mathds 1_{\{u<\tau\}}
    \left(
        \Eds_\varepsilon\left[|S_{ab}(u)|^q\right]
    \right)^{\frac1q}
    \leqslant
    CqN^{1+\xi}\ell^{-2}.
\end{equation}
Set
\[
    m_u
    =
    \Eds_\varepsilon\left[h_N(Y_u)\right].
\]
Writing $q'=q/(q-1)$, conditional H\"older's inequality,
\eqref{eq:relative_S_Lq_extremal}, and
$\vert\partial_{ab}h_N\vert\leqslant L_Nh_N$ give
\[
\begin{aligned}
&\Eds_\varepsilon\left[
    \mathds 1_{\{u<\tau\}}
    \vert\partial_{ab}h_N(Y_u)S_{ab}(u)\vert
\right]\leqslant
CqL_NN^{1+\xi}\ell^{-2}
\mathds 1_{\{u<\tau\}}m_u^{1-\frac1q},
\end{aligned}
\]
where we used $h_N^{q'}\leqslant h_N$.  Applying H\"older's inequality once more and then averaging over the
remaining randomness gives
\[
\begin{aligned}
    \Eds\left[
        \mathds 1_{\{u<\tau\}}m_u^{1-\frac1q}
        \middle\vert\bm\lambda
    \right]
    &\leqslant
    \left(
        \Eds\left[
            \mathds 1_{\{u<\tau\}}m_u
            \middle\vert\bm\lambda
        \right]
    \right)^{1-\frac1q} =
    \left(
        \Eds\left[
            \mathds 1_{\{u<\tau\}}g_u(\z(u))
            \middle\vert\bm\lambda
        \right]
    \right)^{1-\frac1q}
    \leqslant
    \Phi(u)^{1-\frac1q}.
\end{aligned}
\]
Consequently,
\begin{equation}\label{eq:relative_offdiagonal_extremal}
\left\vert
    \Eds\left[
        \mathds 1_{\{u<\tau\}}\Oscr_ug_u(\z(u))
        \middle\vert\bm\lambda
    \right]
\right\vert
\leqslant
CqL_NN^{1+\xi}\ell^{-2}\Phi(u)^{1-\frac1q}.
\end{equation}

We choose $q=\lceil\log N\rceil$. For every fixed $K>0$,
\[
    x^{1-\frac{1}{q}}\leqslant C_K(x+N^{-K}),
    \qquad 0\leqslant x\leqslant1.
\]
It follows from \eqref{eq:relative_phi_derivative_extremal}--\eqref{eq:relative_offdiagonal_extremal} that
\[
    \vert\Phi'(u)\vert
    \leqslant
    \Lambda_N(\Phi(u)+N^{-K}),
\]
where
\[
    \Lambda_N
    =
    C_{K,r}L_N\left(
        N^{1+\xi}\ell^{-\frac{3}{2}}
        +qN^{1+\xi}\ell^{-2}
    \right).
\]
With our choice of $t,\ell$ and $\xi$,
\[
    tN^{1+\xi}\ell^{-\frac{3}{2}}=N^{-\frac{3}{8}},
    \qquad
    tN^{1+\xi}\ell^{-2}=N^{-\frac{5}{6}}.
\]
Since $L_N$ is bounded by a power of $\log N$, Gronwall's lemma gives, for some $c>0$,
\[
    \vert\Phi(t)-\Phi(0)\vert
    \leqslant
    CN^{-c}\Phi(0)+C_DN^{-D}.
\]
Finally, Proposition \ref{prop:refgen} gives
\[
    \Phi(0)=\Eds[h_N(G_r)]+O(N^{-D})
\]
for every $D>0$, since $(Nt^3/\ell)^{1/3}=N^{1/144}$. Moreover, since
$g_t(\z(t))=h_N(\z_k(t))$ on $\{\tau>t\}$ and both quantities are
bounded by one,
\[
    \left\vert
        \Eds\left[h_N(\z_k(t))\middle\vert\bm\lambda\right]
        -\Phi(t)
    \right\vert
    \leqslant
    \Pds(\tau\leqslant t\mid\bm\lambda)
    \leqslant N^{-\nu}.
\]
Combining these estimates with the bound on
$\vert\Phi(t)-\Phi(0)\vert$, taking $\nu$ sufficiently large, and then
taking the overwhelming probability exponent in Lemma
\ref{lem:good_event} sufficiently large and averaging over the
eigenvalue trajectory proves the proposition.
\end{proof}

We will also use the same comparison scheme as in the proof of Theorem \ref{theo:finite}. By \cite{erdHos2017dynamical}*{Lemma 16.2}, there exists a generalized Wigner matrix $H_0$ such that, if $H_t$ is the Dyson Brownian motion started from $H_0$ and   $T=a(t)H_t$    
then the entries of $T$ match those of $W$ up to order three and
\begin{equation}\label{eq:extremal_moment_matching}
    \Eds[t_{ij}^a]=\Eds[w_{ij}^a],
    \quad a=1,2,3,
    \qquad
    \left\vert\Eds[t_{ij}^4]-\Eds[w_{ij}^4]\right\vert
    \leqslant CN^{-2}t.
\end{equation}
As in the comparison argument of \cite{BL22LInfinity}, the construction can be taken inside a uniform subexponential generalized Wigner class.

To compare the extremal process for $W$ and $T$, we need a smooth approximation of a fixed number of upper order statistics. If $\x=(x_1,\dots,x_N)$, write $x_{(1)}\geqslant\cdots\geqslant x_{(N)}$ for its decreasing rearrangement and define, for $j\geqslant1$,
\[
    \Scal_{j,\beta}(\x)
    =
    \frac{1}{\beta}
    \log\sum_{\substack{I\subset\unn{1}{N}\\ \vert I\vert=j}}
    \e^{\beta\sum_{i\in I}x_i},
    \qquad
    \Scal_{0,\beta}=0,
    \qquad
    \Tcal_{j,\beta}=\Scal_{j,\beta}-\Scal_{j-1,\beta}.
\]

\begin{lemma}\label{lem:soft_order_statistics}
Let $\beta\geqslant1$. For every $1\leqslant j\leqslant N$,
\[
    x_{(j)}-\frac{1}{\beta}\log\binom{N}{j-1}
    \leqslant
    \Tcal_{j,\beta}(\x)
    \leqslant
    x_{(j)}+\frac{1}{\beta}\log\binom{N}{j}.
\]
Moreover, for every fixed $m$, every $\Psi\in\Ccal^5(\Rbb^m)$, and
$1\leqslant d\leqslant5$, if
\[
    F_{m,\beta}(\x)
    =
    \Psi\left(
        \Tcal_{1,\beta}(\x),\dots,\Tcal_{m,\beta}(\x)
    \right),
\]
then
\begin{equation}\label{eq:soft_order_derivative_extremal}
    \sum_{i_1,\dots,i_d=1}^N
    \left\vert
        \partial_{i_1}\cdots\partial_{i_d}F_{m,\beta}(\x)
    \right\vert
    \leqslant
    C_{m,d}\beta^{d-1}
    \max_{1\leqslant q\leqslant d}\Vert D^q\Psi\Vert_\infty.
\end{equation}
\end{lemma}

\begin{proof}
The largest term in the sum defining $\Scal_{j,\beta}$ is obtained by taking the $j$ largest coordinates. Thus
\[
    \sum_{a=1}^jx_{(a)}
    \leqslant
    \Scal_{j,\beta}(\x)
    \leqslant
    \sum_{a=1}^jx_{(a)}+\frac{1}{\beta}\log\binom{N}{j},
\]
and the first claim follows by subtracting the corresponding bounds for $j$ and $j-1$.

For the derivative estimate, let $\mu_{j,\beta,\x}$ be the probability measure on the subsets $I\subset\unn{1}{N}$ with $\vert I\vert=j$ whose weights are proportional to
\[
    \e^{\beta\sum_{i\in I}x_i},
\]
and set $X_i(I)=\mathds 1_{\{i\in I\}}$. Differentiating the logarithm gives
\[
    \partial_{i_1}\cdots\partial_{i_d}\Scal_{j,\beta}(\x)
    =
    \beta^{d-1}\kappa_{\mu_{j,\beta,\x}}(X_{i_1},\dots,X_{i_d}).
\]
Since $\sum_iX_i=j$, the moment--cumulant formula gives
\[
    \sum_{i_1,\dots,i_d=1}^N
    \left\vert
        \partial_{i_1}\cdots\partial_{i_d}\Scal_{j,\beta}(\x)
    \right\vert
    \leqslant
    C_dj^d\beta^{d-1}.
\]
The same bound, with a constant depending on $m$, holds for
$\Tcal_{j,\beta}$, $1\leqslant j\leqslant m$. Every $d$-th derivative of
$F_{m,\beta}$ is a finite sum of terms involving $q$ derivatives of
$\Psi$ and derivatives of the $\Tcal_{j,\beta}$ whose orders sum to
$d$. After summing over the coordinate indices, such a term is bounded
by
\[
    C_{m,d}\beta^{d-q}\Vert D^q\Psi\Vert_\infty
    \leqslant
    C_{m,d}\beta^{d-1}\Vert D^q\Psi\Vert_\infty,
\]
where we used $\beta\geqslant1$. This proves
\eqref{eq:soft_order_derivative_extremal}.
\end{proof}
We then have the following comparison statement. 
\begin{proposition}\label{prop:fixed_order_statistics_comparison}
For a matrix $H$, let
\[
    Y_{N,k}^{(1)}(H)\geqslant\cdots\geqslant Y_{N,k}^{(N)}(H)
\]
be the decreasing rearrangement of
\[
    \left(N\vert\u_k(H,\alpha)\vert^2-b_N\right)_{\alpha=1}^N.
\]
Then, for every fixed $m$ and every bounded smooth $h:\Rbb^m\to\Rbb$ with bounded derivatives through order five,
\begin{equation}\label{eq:fixed_order_comparison_extremal}
\sup_{k\in\unn{1}{N}}
\left\vert
    \Eds h\left(Y_{N,k}^{(1)}(W),\dots,Y_{N,k}^{(m)}(W)\right)
    -
    \Eds h\left(Y_{N,k}^{(1)}(T),\dots,Y_{N,k}^{(m)}(T)\right)
\right\vert
\xrightarrow[N\to\infty]{}0.
\end{equation}
\end{proposition}

\begin{proof}
We use the regularized eigenvector projections from \cite{BL22LInfinity}*{Definition 4.2}. For $H\in\{W,T\}$, set
\[
    \w_k^{\mathrm{reg}}(H)
    =
    \left(
        Nv_k^{\mathrm{reg}}(H,\eb_\alpha)-b_N
    \right)_{\alpha=1}^N.
\]
We choose the regularization parameters as in the proof of \cite{BL22LInfinity}*{Theorem 1.3} and let $\beta=N^\delta$, where $\delta>0$ will be chosen sufficiently small. The estimates from \cite{BL22LInfinity}*{Lemma 5.5} and
\cite{BL22LInfinity}*{Proposition 5.7} imply that, for some $c>0$,
uniformly in $k$ and $H\in\{W,T\}$,
\[
    \Pds\left(
        \left\Vert
            \w_k^{\mathrm{reg}}(H)
            -
            \left(
                N\vert\u_k(H,\alpha)\vert^2-b_N
            \right)_{\alpha=1}^N
        \right\Vert_\infty
        >N^{-c}
    \right)
    \leqslant N^{-c}.
\]
By Lemma \ref{lem:soft_order_statistics}, it is therefore enough to compare
\[
    h\left(
        \Tcal_{1,\beta}(\w_k^{\mathrm{reg}}(H)),\dots,
        \Tcal_{m,\beta}(\w_k^{\mathrm{reg}}(H))
    \right)
\]
for $H=W$ and $H=T$.

We use the replacement argument from the proof of \cite{BL22LInfinity}*{Theorem 1.3}. Set
\[
    \Fcal_k(H)
    =
    h\left(
        \Tcal_{1,\beta}(\w_k^{\mathrm{reg}}(H)),\dots,
        \Tcal_{m,\beta}(\w_k^{\mathrm{reg}}(H))
    \right).
\]
The derivative estimates from \cite{BL22LInfinity}*{Equation (4.61)}, together with \eqref{eq:soft_order_derivative_extremal}, show that all derivatives through order five which appear in the replacement argument are bounded by $N^{C_m\delta}$ (on a high-probability event whose complement is manifestly negligible). A standard moment-matching argument based on a fourth-order Taylor expansion and \eqref{eq:extremal_moment_matching} therefore gives
\[
    \sup_{k\in\unn{1}{N}}
    \left\vert
        \Eds\Fcal_k(W)-\Eds\Fcal_k(T)
    \right\vert
    \leqslant
    C\left(
        N^{-\frac{1}{48}+C_m\delta}
        +N^{-\frac{1}{2}+C_m\delta}
    \right)
    +o(1).
\]
Choosing $\delta>0$ sufficiently small proves the proposition.
\end{proof}

We now study the extremal process along the Dyson Brownian motion
$H_t$ introduced above. For $x\in\Rbb$, let
\[
    T_N(x)=\sqrt{b_N+x},
    \qquad
    q_N(x)=\Pds(\vert g\vert>T_N(x)),
    \qquad g\sim\Ncal(0,1).
\]
By Mills's ratio, uniformly for $x$ in compact sets,
\begin{equation}\label{eq:gumbel_gaussian_tail_extremal}
    Nq_N(x)\xrightarrow[N\to\infty]{}\e^{-\frac{x}{2}}.
\end{equation}

\begin{proposition}\label{prop:gumbel_joint_tails}
Let $r\geqslant1$ be fixed and let $K\subset\Rbb$ be compact. Set
\[
    X_{\alpha,k}=\v_k(t,\eb_\alpha).
\]
Then
\[
\sup_{\substack{k\in\unn{1}{N},\ x_1,\dots,x_r\in K\\
\alpha_1,\dots,\alpha_r\ \text{distinct}}}
\left\vert
    \frac{
        \Pds\left(
            \vert X_{\alpha_a,k}\vert>T_N(x_a)
            \text{ for every }a\in\unn{1}{r}
        \right)
    }{
        \prod_{a=1}^rq_N(x_a)
    }
    -1
\right\vert
\xrightarrow[N\to\infty]{}0.
\]
\end{proposition}

\begin{proof}
Write $\x=(x_1,\dots,x_r)$ and set $w=(\log N)^{-3}$. We define
\[
    h_{N,\x}(\y)
    =
    \prod_{a=1}^r\chi_{T_N(x_a),w}(y_a).
\]
By \eqref{eq:self_bounding_hessian_extremal},
\[
    \Vert\nabla^2h_{N,\x}(\y)\Vert_{\mathrm{op}}
    \leqslant
    C_r(\log N)^6h_{N,\x}(\y).
\]
Proposition \ref{prop:relative_comparison} and \eqref{eq:gaussian_logistic_tail_extremal} give, uniformly in the parameters appearing in the statement,
\begin{equation}\label{eq:gumbel_logistic_relative_extremal}
\Eds\left[
    \prod_{a=1}^r
    \chi_{T_N(x_a),w}(X_{\alpha_a,k})
\right]
=
(1+o(1))\prod_{a=1}^rq_N(x_a).
\end{equation}
The error $N^{-D}$ in Proposition \ref{prop:relative_comparison} is negligible by taking $D$ sufficiently large, since $\prod_aq_N(x_a)=N^{-r+o(1)}$.

It remains to remove the smooth cutoffs. Let
\[
    \rho_N=(3r+10)\log N,
    \qquad
    \delta_N=\rho_Nw.
\]
We apply \eqref{eq:gumbel_logistic_relative_extremal} with thresholds $T_N(x_a)\pm\delta_N$. Since $T_N(x)\delta_N\to0$ uniformly for $x\in K$, the corresponding
Gaussian probabilities are $(1+o(1))q_N(x)$. Indeed, writing
$A_a=\{\vert X_{\alpha_a,k}\vert>T_N(x_a)\}$, we have
\[
    \chi_{T_N(x_a)-\delta_N,w}\geqslant1-\e^{-\rho_N}
    \quad\text{on }A_a,
    \qquad
    \chi_{T_N(x_a)+\delta_N,w}\leqslant2\e^{-\rho_N}
    \quad\text{on }A_a^c.
\]
Since $0\leqslant\chi_{T,w}\leqslant1$, it follows that
\[
\begin{aligned}
    \Eds\left[
        \prod_{a=1}^r
        \chi_{T_N(x_a)+\delta_N,w}(X_{\alpha_a,k})
    \right]
    -2\e^{-\rho_N}
    &\leqslant
    \Pds\left(\bigcap_{a=1}^r A_a\right)\le 
    (1-\e^{-\rho_N})^{-r}
    \Eds\left[
        \prod_{a=1}^r
        \chi_{T_N(x_a)-\delta_N,w}(X_{\alpha_a,k})
    \right].
\end{aligned}
\]
Since $\e^{-\rho_N}=o(\prod_aq_N(x_a))$, the shifted comparison
estimates imply the desired conclusion.
\end{proof}

\begin{proposition}\label{prop:extremal_process_flow}
Define
\[
    \Xi_{N,k}(t)
    =
    \sum_{\alpha=1}^N
    \delta_{\vert X_{\alpha,k}\vert^2-b_N}.
\]
Then, uniformly in $k\in\unn{1}{N}$, $\Xi_{N,k}(t)$ converges in distribution in $\Mcal_p(\Rbb)$ to a Poisson point process of intensity
\[
    \frac{1}{2}\e^{-\frac{x}{2}}\d x.
\]
Moreover, every fixed number of largest atoms converge jointly and uniformly in $k$.
\end{proposition}

\begin{proof}
Let $I_1,\dots,I_s$ be pairwise disjoint bounded intervals with $I_j=(a_j,b_j]$ and set
\[
    Z_{N,k}(I_j)=\Xi_{N,k}(t)(I_j).
\]
Fix $r_1,\dots,r_s\in\Nbb_0$ and let $r=\sum_jr_j$. We write $(m)_q=m(m-1)\cdots(m-q+1)$. For every family of distinct indices
\[
    \alpha_{j,a},
    \qquad
    j\in\unn{1}{s},\quad a\in\unn{1}{r_j},
\]
the inclusion--exclusion principle and Proposition \ref{prop:gumbel_joint_tails} give
\[
\begin{aligned}
&\Pds\left(
    \vert X_{\alpha_{j,a},k}\vert^2-b_N\in I_j
    \text{ for every }j\in\unn{1}{s},\ a\in\unn{1}{r_j}
\right)=
(1+o(1))
\prod_{j=1}^s
\left(q_N(a_j)-q_N(b_j)\right)^{r_j},
\end{aligned}
\]
uniformly in all the indices. Summing over the $(N)_r$ ordered choices and using \eqref{eq:gumbel_gaussian_tail_extremal}, we obtain
\[
    \Eds\prod_{j=1}^s\left(Z_{N,k}(I_j)\right)_{r_j}
    \xrightarrow[N\to\infty]{}
    \prod_{j=1}^s
    \left(
        \int_{I_j}\frac{1}{2}\e^{-\frac{x}{2}}\d x
    \right)^{r_j}
\]
uniformly in $k$. The same argument applies when one of the intervals is a right half-line. The convergence of the mixed factorial moments gives the independent Poisson limits for the count vectors. Approximating a non-negative compactly supported continuous function by step functions gives the convergence of the Laplace functionals. The right half-line version also gives the joint convergence of every fixed number of largest atoms.
\end{proof}

We are now ready to prove Theorem \ref{theo:extremal_process}.
\begin{proof}[Proof of Theorem \ref{theo:extremal_process}]
Let $T$ be the Gaussian divisible matrix defined above \eqref{eq:extremal_moment_matching} and use the notation $Y_{N,k}^{(j)}(H)$ from Proposition \ref{prop:fixed_order_statistics_comparison}. Let
\[
    \eta^{(1)}>\eta^{(2)}>\cdots
\]
denote the atoms of the limiting Poisson point process in decreasing order. Proposition \ref{prop:extremal_process_flow} gives, for every fixed $m$,
\[
    \left(
        Y_{N,k}^{(1)}(T),\dots,Y_{N,k}^{(m)}(T)
    \right)
    \xrightarrow[N\to\infty]{(d)}
    \left(
        \eta^{(1)},\dots,\eta^{(m)}
    \right)
\]
uniformly in $k$. Proposition \ref{prop:fixed_order_statistics_comparison}, applied to smooth approximations of bounded continuous functions, gives the same convergence with $W$ in place of $T$.

It remains to prove the convergence of the point process. Fix a non-negative $f\in\Ccal_c(\Rbb)$ and choose $x_0$ below the support of $f$. We truncate $\int f\d\Xi_{N,k}$ after the first $m$ atoms. The resulting quantity is a bounded continuous function of the first $m$ order statistics. The truncation can differ from the original functional only if $Y_{N,k}^{(m+1)}(W)>x_0$. By the convergence of the first $m+1$ order statistics,
\[
    \limsup_{N\to\infty}\sup_{k\in\unn{1}{N}}
    \Pds\left(Y_{N,k}^{(m+1)}(W)>x_0\right)
    \leqslant
    \Pds\left(\eta^{(m+1)}>x_0\right),
\]
and the right-hand side tends to zero as $m\to\infty$. The same argument applies to the limiting point process. This proves the convergence of the Laplace functionals and hence the theorem.
\end{proof}

\begin{proof}[Proof of Corollary \ref{theo:gumbel}]
The largest atom of $\Xi_{N,k}$ is $
    N\Vert\u_k\Vert_\infty^2-b_N$. 
The probability that the limiting Poisson process has no point in $(x,\infty)$ is
\[
    \e^{-\int_x^\infty\frac{1}{2}\e^{-\frac{y}{2}}\d y}
    =
    \e^{-\e^{-\frac{x}{2}}}.
\]
The result follows from Theorem \ref{theo:extremal_process}.
\end{proof}

\subsection{Quantitative eigenvector universality for smooth generalized Wigner matrices} To transfer our quantitative dynamical estimate to smooth generalized Wigner matrices, we use the quantitative reverse heat flow technique from \cite{bourgade2022extreme}*{Lemma 4.1} based on \cite{erdHos2011universality}*{Proposition 4.1}.
\begin{lemma}[\cite{bourgade2022extreme}]
    \label{lem:reverse_heat_flow}
    Let $W$ be a  smooth   real symmetric    generalized Wigner matrix as in Definition \ref{def:smooth}. If $t=N^{-\theta}$ with $\theta\in(0,1),$ then for any $D>0$ there exists $C>0$ and a generalized Wigner matrix $H_0$ such that   
    \[ 
      \d_{TV}\left(
    W,
    \left(1+\frac{t}{N}\right)^{-\frac12}
    \left(\sqrt{1-t}H_0+\sqrt{t}\mathrm{GOE}_N\right)
\right)\leqslant CN^{-D}.    
    \]
\end{lemma}
    
\begin{remark}
We note that the additional subexponential decay assumption in \cite{bourgade2022extreme}*{Lemma 4.1} is not needed here. Indeed, in the proof of that lemma it is used only to control the contribution outside the cutoff introduced in the reverse heat flow construction, and to ensure that the resulting density satisfies the same tail assumption. Under Definition \ref{def:genwigner}, for every fixed $a,M>0$ and every rescaled entry $X=\sqrt{N}w_{ij}$,
\[
    \Eds\left[
        (1+\vert X\vert)^a
        \mathds 1_{\{\vert X\vert>N^\varepsilon\}}
    \right]
    \leqslant C_{a,M}N^{-\varepsilon M}.
\]
Then the same proof, choosing the cutoff exponent sufficiently small and then $M$ sufficiently large, gives an arbitrarily large polynomial gain in the cutoff errors. The normalization and first two moment corrections can be handled as in the cited proof, and the resulting entries retain uniformly bounded moments of every fixed order. After taking the entrywise error smaller than $N^{-D-2}$ and tensorizing over the matrix entries, this gives the stated $N^{-D}$ total variation bound.
\end{remark}
  
We are now ready to prove Theorem \ref{theo:growing}. 
\begin{proof}[Proof of Theorem \ref{theo:growing}]
    It suffices to consider $\kappa$ values satisfying 
$0<\kappa<b-\delta_2$, since a bound with any smaller positive value
of $\kappa$ implies the claim for larger $\kappa$. Choose $\theta>0$
sufficiently small so that
\[
    8\theta<\min\{\kappa,b\},
    \qquad
    \theta<\frac{1}{6}.
\]
  
We set
\[
    t=N^{-\theta},
    \qquad
    \ell=N^{1-4\theta},
    \qquad
    \xi=\theta
\]
so the hypotheses of Theorem \ref{theo:maindyn} are satisfied for $N$ large enough.

We first prove the estimate for the Gaussian divisible matrix. Let $H_0$ be a generalized Wigner matrix and let $H_t$ be the corresponding Dyson Brownian motion at time $t=N^{-\theta}$. Denote its eigenvectors by $\u_k(t)$, and set
\[
    X_N(t)
    =
    \left(
        \langle \q_{i\alpha},\v_{k_i}(t)\rangle
    \right)_{\substack{1\leqslant i\leqslant p\\1\leqslant \alpha\leqslant a_i}}.
\]
Applying Theorem \ref{theo:maindyn}, and then removing the conditioning on the eigenvalue trajectory using the overwhelming-probability estimate for the good set $\Lfr$, gives
\[
\begin{aligned}
\left|
    \Eds\left[      
        h(X_N(t))
        \right]-\Eds [h(G_N)]
\right|
&\leqslant
C\Vert \nabla^2h\Vert_\infty
\frac{N^{1+\xi}t}{\ell}
\left(
    \frac{RQ}{\sqrt\ell}+\frac{Q^2}{\ell}
\right)  \\
&\quad
+
C\sqrt R\,\Vert h\Vert_\infty\sqrt N
\exp\left(-c\left(\frac{Nt^3}{\ell}\right)^{1/3}\right)
+
C\Vert h\Vert_\infty N^{-\nu}
+
C\Vert h\Vert_\infty N^{-D},
\end{aligned}
\]
where \(D,\nu>0\) can be chosen arbitrarily large.

We now estimate the first term, with our choice of parameters, we get
\[
\frac{N^{1+\xi}t}{\ell}
\frac{RQ}{\sqrt\ell}
\leqslant
N^{4\theta}N^{-\frac{1}{2}+2\theta}N^{\frac{1}{2}-b}
=
N^{-b+6\theta}
\quad\text{and}\quad
\frac{N^{1+\xi}t}{\ell}
\frac{Q^2}{\ell}
\leqslant
N^{4\theta}N^{-1+4\theta}N^{1-2b}
=
N^{-2b+8\theta}.
\]
Using \(\Vert\nabla^2h\Vert_\infty\leqslant N^{\delta_2}\), this yields
\[
\Vert \nabla^2h\Vert_\infty
\frac{N^{1+\xi}t}{\ell}
\left(
    \frac{RQ}{\sqrt\ell}+\frac{Q^2}{\ell}
\right)
\leqslant
C\left(
    N^{\delta_2-b+6\theta} + N^{\delta_2-2b+8\theta}
\right).
\]
   
By the choice of $\theta$, this is bounded by $
    C N^{-b+\delta_2+\kappa}$. 
The relaxation term is smaller than any power of $N$. Choosing
\[
    \nu,D>\delta_1+b-\delta_2+1,
\]
the last two error terms are also
$O(N^{-b+\delta_2+\kappa})$. Therefore
\[
    \left|
        \Eds h(X_N(t))-\Eds h(G_N)
    \right|
    \leqslant
    C_0 N^{-b+\delta_2+\kappa}.
\]

It remains to pass from the Gaussian divisible matrix back to $W$. By Lemma \ref{lem:reverse_heat_flow}, for the above choice $t=N^{-\theta}$ and for any $D>0$, there exists a generalized Wigner matrix $H_0$ such that
\[
   \d_{\mathrm{TV}}\left(
    W,
    \left(1+\frac{t}{N}\right)^{-\frac12}
    \left(\sqrt{1-t}\,H_0+\sqrt t\,\mathrm{GOE}_N\right)
\right)
\leqslant
N^{-D}.    
\]
Up to the harmless deterministic normalization of the flow and renormalizing by $(1-t)^{-\frac{1}{2}}$, the matrix $\sqrt{1-t}\,H_0+\sqrt t\,\mathrm{GOE}_N$
has the same eigenvectors as the Gaussian divisible matrix \(H_s\) for $s={\frac{t}{1-t}}\asymp t$ considered above. Hence, writing \(X_N(W)\) for the vector of projections associated with \(W\),
\[
\left|
      \Eds [h(X_N(W))]-\Eds [h(X_N(s))]   
\right|
\leqslant
2\Vert h\Vert_\infty N^{-D}
\leqslant
2N^{\delta_1-D}.
\]
    
Choosing $D>\delta_1+b-\delta_2+1$,
this error is $O(N^{-b+\delta_2+\kappa})$.

Combining the Gaussian divisible estimate with the reverse heat flow
comparison gives
\[
    \left|
        \Eds h(X_N(W))-\Eds h(G_N)
    \right|
    \leqslant
    C_0 N^{-b+\delta_2+\kappa},
\]
after increasing the value of $C_0$. 
This proves the theorem.    
\end{proof}
   
We now prove the Kolmogorov bound for the Dyson Brownian motion by
refining the comparison argument in the proof of Theorem
\ref{theo:maindyn}.   
We first state Nazarov's inequality, which is an anti-concentration inequality of Gaussian vectors.
\begin{lemma}[\cite{nazarov2003maximal}]
\label{lem:nazarov}
Let \(G=(G_1,\dots,G_Q)\) be a centered Gaussian vector in \(\Rbb^Q\), and assume that $\Eds[G_a^2]\geqslant \sigma^2$ for every $a\in\unn{1}{Q}$
for some \(\sigma>0\). Then, for every \(\eta>0\),
\[
    \sup_{\y\in\Rbb^Q}
    \Pds\left(
        \left|\max_{1\leqslant a\leqslant Q}(G_a-\y_a)\right|\leqslant \eta
    \right)
    \leqslant
    \frac{2\eta}{\sigma}\left(\sqrt{2\log Q}+2\right).
\]
\end{lemma}
   
We also require the following Gaussian smoothing estimate, which allows us to leverage the Gaussian regularization from the reference dynamics.  
\begin{lemma}\label{lem:gaussian_polyhedron_smoothing}
Let $A\subset\Rbb^d$ be a Borel set and, for $w>0$, define
\[
    H_{A,w}(\x)
    =
    \Eds\left[\mathds 1_A(\x+wZ)\right],
    \qquad Z\sim\Ncal(0,\Id_d).
\]
Let $\Gamma$ be a positive semidefinite matrix such that
$\Gamma\succeq v\Id_d$ for some $v\geqslant0$. If $P_\Gamma$
denotes convolution with the centered Gaussian distribution of
covariance $\Gamma$, then there exists $C>0$, independent of $d$, $A$, $w$, and $\Gamma$, such that 
\[
    \sup_{\x\in\Rbb^d}
    \left\Vert
        \nabla^2 P_\Gamma H_{A,w}(\x)
    \right\Vert_{\mathrm{op}}
    \leqslant
    \frac{C}{w^2+v}.
\]
\end{lemma}

\begin{proof}
We have
\[
    P_\Gamma H_{A,w}
    =
    P_{\Gamma+w^2\Id_d}\mathds 1_A.
\]
Put $C_0=\Gamma+w^2\Id_d$, so that
$C_0\succeq(w^2+v)\Id_d$. For unit vectors
$\mathbf a,\mathbf b\in\Rbb^d$ and $Y\sim\Ncal(0,C_0)$, Gaussian
integration by parts gives
\[
\begin{aligned}
    \partial_{\mathbf a}\partial_{\mathbf b}
    P_{C_0}\mathds 1_A(\x)
    =
    \Eds\Big[\mathds 1_A(\x+Y)
    \big(&\langle \mathbf a,C_0^{-1}Y\rangle
          \langle \mathbf b,C_0^{-1}Y\rangle
          -\langle \mathbf a,C_0^{-1}\mathbf b\rangle\big)\Big].
\end{aligned}
\]
By Cauchy--Schwarz,
\[
    \Eds\left[
        \left|\langle \mathbf a,C_0^{-1}Y\rangle
        \langle \mathbf b,C_0^{-1}Y\rangle\right|
    \right]
    \leqslant
    \frac{1}{w^2+v},
\]
and the deterministic term is bounded in the same way. Taking the
supremum over $\mathbf a$ and $\mathbf b$ proves the lemma.
\end{proof}
  
We thus can obtain the following Kolmogorov bound for the Dyson Brownian motion.
    
\begin{proposition}
\label{prop:kolmogorov_from_dynamic}
Let $H_t$ be the Gaussian divisible matrix considered in Theorem
\ref{theo:maindyn}. Suppose $t,\ell,\xi$ are admissible in Theorem~\ref{theo:maindyn} and such that $N t^3 \ge N^{\kappa} \ell$ for some $\kappa>0$. Define
\[
    \Kcal_N
    =
    RQ\,N^{\frac23+\xi}\ell^{-\frac76}
    +
    Q^2N^{\frac23+\xi}\ell^{-\frac53}.
\]
Then, for every $D>0$, if $\Kcal_N\leqslant\frac12$,
\[
\begin{aligned}
    \d_{\mathrm K}(X_N(t),G_N)
    \leqslant{}&
    C\Kcal_N\log\left(\frac{\e Q}{\Kcal_N}\right)
    +
    C\sqrt{RN}
    \exp\left(-c\left(\frac{Nt^3}{\ell}\right)^{\frac13}\right)
    +N^{-D}.
\end{aligned}
\]
\end{proposition}

\begin{proof}
Recall the matrices
$O_i\in\Rbb^{a_i\times R}$ introduced after Theorem
\ref{theo:maindyn}, and define $
    r_i=\operatorname{rank}(O_i)$. 
Choose a matrix $P_i\in\Rbb^{r_i\times R}$ whose rows form an
orthonormal basis of the row space of $O_i$, and choose a  matrix
$L_i\in\Rbb^{a_i\times r_i}$ such that $O_i=L_iP_i$. Since the rows of $O_i$ have
norm one and the rows of $P_i$ are orthonormal, the rows of $L_i$ also
have norm one, and $L_iL_i^\top = O_iO_i^\top$. 
Define 
\[
    d=\sum_{i=1}^p r_i,
    \qquad
    L=\operatorname{diag}(L_1,\dots,L_p),
\]
and define the linear transformation 
$\Pi_{\mathbf k}:(\Rbb^R)^N\to\Rbb^d$ by $\Pi_{\mathbf k}\x = \left(P_i\x_{k_i}\right)_{1\leqslant i\leqslant p}$.
Then
\[
    X_N(t)=L\Pi_{\mathbf k}\z(t).
\]
Moreover, if $Z_d\sim\Ncal(0,\Id_d)$, then $LZ_d$ has the same
distribution as $G_N$ by the fact that $L_iL_i^\top = O_iO_i^\top$.

Fix $\y\in\Rbb^Q$ and consider
\[
    A_\y
    =
    \left\{\x\in\Rbb^d:L\x\leqslant\y\right\},
\]
which satisfies
\[
\Pds ( X_N(t) \le \y) = \Pds ( \Pi_{\mathbf k}\z(t) \in A_\y).
\]
For $w>0$, let $H_{\y,w}=H_{A_\y,w}$ be the function from
Lemma \ref{lem:gaussian_polyhedron_smoothing}, and set
\[
    F_{\y,w}(\z)
    =
    H_{\y,w}(\Pi_{\mathbf k}\z).
\]
The rows of the matrices $P_i$, expressed in the orthonormal basis
$\Efr=\{\eb_1,\dots,\eb_R\}$, determine at most $Q$ additional
deterministic unit vectors in $E$ (defined before Theorem~\ref{theo:maindyn}). Throughout this proof, we enlarge
$\Qfr$ by these vectors. Lemma \ref{lem:good_event} and all estimates
below remain unchanged.

Let
\[
    g_u=\hat{\Pscr}_{u,t}F_{\y,w}
\]
be the backward reference evolution and put
\[
    a_u=\exp\left(-c\left(\frac{N}{\ell}\right)^{\frac13}(t-u)\right).
\]
We adopt the notation of the proof of Lemma \ref{lem:entropy_bound}. That proof, applied on the interval
$[u,t]$, gives 
\begin{equation}\label{e:Mut}
    M_{u,t}M_{u,t}^\top\preceq a_u\Id_N,
    \qquad
    \Sigma_{u,t}\succeq(1-a_u)\Id_N.
\end{equation}
Conditionally on the auxiliary short-range Brownian motions, define $\Ucal_{u,t}=\Pi_{\mathbf k}(M_{u,t}\otimes\Id_R)$, 
and let $\Gamma_{u,t}=\Pi_{\mathbf k}(\Sigma_{u,t}\otimes\Id_R)\Pi_{\mathbf k}^\top$
be the conditional covariance of $\Pi_{\mathbf k}\Xi_{u,t}$.
Since the indices $k_i$ are distinct and the rows of each $P_i$ are
orthonormal,
\[
    \Pi_{\mathbf k}\Pi_{\mathbf k}^\top
    =
    \operatorname{diag}(P_1P_1^\top,\dots,P_pP_p^\top)
    =
    \Id_d.
\]
Hence
\[
\begin{aligned}
    \Ucal_{u,t}\Ucal_{u,t}^\top
    &=
    \Pi_{\mathbf k}
    \bigl(M_{u,t}M_{u,t}^\top\otimes\Id_R\bigr)
    \Pi_{\mathbf k}^\top
    \preceq a_u\Id_d,\\
    \Gamma_{u,t}
    &=
    \Pi_{\mathbf k}
    \bigl(\Sigma_{u,t}\otimes\Id_R\bigr)
    \Pi_{\mathbf k}^\top
    \succeq (1-a_u)\Id_d.
\end{aligned}
\]

We now fix the auxiliary Brownian motions temporarily. Conditionally on
this randomness, the backward observable is
\[
    P_{\Gamma_{u,t}}H_{\y,w}(\Ucal_{u,t}\x).
\]
Write $\Ucal_j(u):\Rbb^R\to\Rbb^d$ for the $j$-th block of
$\Ucal_{u,t}$. By Lemma \ref{lem:gaussian_polyhedron_smoothing}
and the chain rule,
\[
\begin{aligned}
    \sum_{j=1}^N
    \left\Vert\nabla_j^2g_u(\x)\right\Vert_{\mathrm{HS}}
    &\leqslant
    \frac{C}{w^2+1-a_u}
    \sum_{j=1}^N\Vert\Ucal_j(u)\Vert_{\mathrm{HS}}^2\leqslant
    CQ\frac{a_u}{w^2+1-a_u}.
\end{aligned}
\]
Since this estimate holds uniformly in the auxiliary Brownian motions,
averaging over this randomness gives the same bound for $g_u$. 
Combining this estimate with the bound on $\Delta_j(u)$ from the proof
of Lemma \ref{lem:diagonal_term}, and bounding $\Dscr_ug_u(\z(u))$ as in that proof, we obtain
\[
\left|
    \Eds\left[
        \mathds 1_{\{u<\tau\}}\Dscr_ug_u(\z(u))
        \middle|\bm\lambda
    \right]
\right|
\leqslant
CRQ N^{1+\xi}\ell^{-\frac{3}{2}}
\frac{a_u}{w^2+1-a_u}.
\]

For the off-diagonal term, we repeat the sign-pairing argument from the
proof of Lemma \ref{lem:offdiagonal_term}, now using the rows of the
$P_i$ as the projection directions. Index the coordinates of
$\Rbb^d$ by $a$, and let $\widetilde{\q}_a\in E$ be the deterministic
unit vector corresponding to the relevant row of $P_r$ when $a$
belongs to the $r$-th block. Set
\[
    \mu_i^a(u)=(M_{u,t})_{k_r i}
\]
and
\[
    S_{ab}(u)
    =
    \sum_{|i-j|\geqslant\ell}
      c_{ij}(u)\mu_j^a(u)\mu_i^b(u)    
    \v_i(u,\widetilde{\q}_a)\v_j(u,\widetilde{\q}_b).
\]
The only change from the proof of Lemma \ref{lem:offdiagonal_term} is
that we use the first estimate in \eqref{e:Mut}. 
This shows that for every coordinate $a$,
\[
    \sum_{i=1}^N(\mu_i^a(u))^2\leqslant a_u,
\]
in place of the bound by $1$ used there. Repeating the sign-pairing
estimate from the proof of Lemma \ref{lem:offdiagonal_term} with this
improved bound gives
\[
    \left(
        \Eds\left[
            \mathds{1}_{\{u<\tau\}}|S_{ab}(u)|^2
            \middle\vert\bm\lambda
        \right]
    \right)^{\frac12}
    \leqslant
    CN^{1+\xi}\ell^{-2}a_u.
\]
Further, Lemma~\ref{lem:gaussian_polyhedron_smoothing} gives
\[
    \max_{a,b}
    |\partial_{ab}P_{\Gamma_{u,t}}H_{\y,w}|
    \leqslant
    \frac{C}{w^2+1-a_u}.
\]
Since $d\leqslant Q$, it follows from the proof of Lemma~\ref{lem:offdiagonal_term} that
\[
\left|
    \Eds\left[
        \mathds 1_{\{u<\tau\}}\Oscr_ug_u(\z(u))
        \middle|\bm\lambda
    \right]
\right|
\leqslant
CQ^2N^{1+\xi}\ell^{-2}
\frac{a_u}{w^2+1-a_u}.
\]

The time integral of the last factor satisfies
\[
    \int_0^t
    \frac{a_u}{w^2+1-a_u}\,\d u
    \leqslant
      C\left(\frac{\ell}{N}\right)^{\frac{1}{3}}
\log(1+w^{-2}).    
\]
Using Lemma \ref{lem:duhamel_stopped}, Proposition~\ref{prop:refgen}, and Lemma \ref{lem:good_event}, and taking $\nu$ and the
overwhelming-probability exponent sufficiently large, we obtain
uniformly in $\y$,
\[
\left|
    \Eds H_{\y,w}(\Pi_{\mathbf k}\z(t))
    -
    \Eds H_{\y,w}(Z_d)
\right|\leqslant
C\Kcal_N\log(1+w^{-2})
+
C\sqrt{RN}
\exp\left(-c\left(\frac{Nt^3}{\ell}\right)^{\frac13}\right)
+N^{-D}.
\]

It remains to remove the Gaussian convolution. Let
$\mathbf n_1,\dots,\mathbf n_Q$ be the rows of $L$, so that
$\Vert\mathbf n_a\Vert=1$. For $\Lambda>0$ and an independent
$Z\sim\Ncal(0,\Id_d)$, set $\varepsilon_\Lambda = 2Q\e^{-\Lambda^2/2}$ and note that 
\[
    \Pds\left(
        \max_{1\leqslant a\leqslant Q}
        |\langle\mathbf n_a,Z\rangle|>\Lambda
    \right)
    \leqslant
    2Q\e^{-\Lambda^2/2}. 
\]
Put $\eta=w\Lambda$, let $\mathbf 1=(1,\dots,1)\in\Rbb^Q$, and write
\[
    F_X(\y)
    =
    \Pds\left(L\Pi_{\mathbf k}\z(t)\leqslant\y\right),
    \qquad
    F_G(\y)
    =
    \Pds\left(LZ_d\leqslant\y\right).
\]
On the event
$\max_a|\langle\mathbf n_a,Z\rangle|\leqslant\Lambda$, adding $wZ$ changes each inequality in the definition of $A_\y$ by at most $\eta$. Hence
\[
\begin{aligned}
    F_X(\y)
    &\leqslant
    \frac{\Eds H_{\y+\eta\mathbf 1,w}
    (\Pi_{\mathbf k}\z(t))}{1-\varepsilon_\Lambda},\\
    F_X(\y)
    &\geqslant
    \Eds H_{\y-\eta\mathbf 1,w}
    (\Pi_{\mathbf k}\z(t))
    -\varepsilon_\Lambda,
\end{aligned}
\]
while
\[
\begin{aligned}
    \Eds H_{\y+\eta\mathbf 1,w}(Z_d)
    &\leqslant
    F_G(\y+2\eta\mathbf 1)+\varepsilon_\Lambda,\\
    \Eds H_{\y-\eta\mathbf 1,w}(Z_d)
    &\geqslant
    F_G(\y-2\eta\mathbf 1)-\varepsilon_\Lambda.
\end{aligned}
\]
Assuming $\varepsilon_\Lambda\leqslant\frac12$ and using
Lemma \ref{lem:nazarov}, we get
\[
    \d_{\mathrm K}(X_N(t),G_N)
    \leqslant
    C\Kcal_N\log(1+w^{-2})
    +Cw\Lambda\sqrt{1+\log Q}
    +C\varepsilon_\Lambda+
    C\sqrt{RN}
    \exp\left(-c\left(\frac{Nt^3}{\ell}\right)^{\frac13}\right)
    +N^{-D}.
\]
Finally, choose
\[
    \Lambda
    =
    \sqrt{2\log\left(\frac{4Q}{\Kcal_N}\right)},
    \qquad
    w
    =
    \frac{\Kcal_N}{\Lambda\sqrt{1+\log Q}}.
\]
Then $\varepsilon_\Lambda\leqslant\Kcal_N/2$ and
\[
    w\Lambda\sqrt{1+\log Q}=\Kcal_N,
    \qquad
    \log(1+w^{-2})
    \leqslant
    C\log\left(\frac{\e Q}{\Kcal_N}\right).
\]
Substituting these estimates proves the proposition.
\end{proof}

We are now ready to prove Theorem \ref{theo:kolmogorov}.
    
\begin{proof}[Proof of Theorem \ref{theo:kolmogorov}]
It suffices to consider $0<\kappa<b$, since for $\kappa\geqslant b$
the conclusion follows from $\d_{\mathrm K}\leqslant1$, after
increasing $C$. Choose $\theta>0$ so small that
\[
    \theta<\frac16,
    \qquad
    \frac{17}{3}\theta<\frac{\kappa}{4},
    \qquad
    \frac{23}{3}\theta<b+\frac{\kappa}{4}.
\]
Set
\[
    t=N^{-\theta},
    \qquad
    \ell=N^{1-4\theta},
    \qquad
    \xi=\theta.
\]
Then $Nt^3/\ell=N^\theta$, and the quantity in Proposition
\ref{prop:kolmogorov_from_dynamic} satisfies
\[
    \Kcal_N
    \leqslant
    C\left(
        RQ\,N^{-\frac12+\frac{17}{3}\theta}
        +
        Q^2N^{-1+\frac{23}{3}\theta}
    \right).
\]

By Lemma \ref{lem:reverse_heat_flow}, for this $t$ and every $D>0$,
there exists a generalized Wigner matrix $H_0$ such that
\[
     \d_{\mathrm{TV}}\left(
    W,
    \left(1+\frac{t}{N}\right)^{-\frac12}
    \left(\sqrt{1-t}\,H_0+\sqrt t\,\mathrm{GOE}_N\right)
\right)
\leqslant
N^{-D}.    
\]
The deterministic scalar normalization does not change eigenvectors,
and the corresponding Dyson flow time is comparable to $t$. Since
$R\geqslant1$ and
$RQ\leqslant CN^{\frac12-b}$,
\[
    Q\leqslant RQ,
    \qquad
    Q^2\leqslant C^2N^{1-2b},
\]
and therefore
\[
    \Kcal_N
    \leqslant
    C\left(
        N^{-b+\frac{17}{3}\theta}
        +
        N^{-2b+\frac{23}{3}\theta}
    \right).
\]
By the choice of $\theta$, 
$\Kcal_N\leqslant\frac12$ for all sufficiently large $N$.
Proposition \ref{prop:kolmogorov_from_dynamic} and the total-variation
comparison therefore give
\[
\begin{aligned}
    \d_{\mathrm K}(X_N,G_N)
    \leqslant{}&
    C\Kcal_N\log\left(\frac{\e Q}{\Kcal_N}\right)
    +
    C\sqrt{RN}\exp(-cN^{\theta/3})
    +
    2N^{-D}\le 
    C N^{-b+\kappa},
\end{aligned}
\]
where $D$ is chosen sufficiently large. This proves the theorem.
\end{proof}

\bibliographystyle{abbrv}
\bibliography{eigvect.bib}
\appendix   
\appendixswitch
\section{Main results for complex Hermitian random matrices}
\label{app:hermitian}
    
In this appendix we record the analogues of the main results for complex Hermitian generalized Wigner matrices. The proofs follow the same strategy, replacing the orthogonal Dyson vector flow by its unitary analogue and Rademacher signs by uniform phases. Throughout this appendix we assume, in addition to Definition \ref{def:genwigner}, that there exists $c>0$ such that
\[
    \Eds\left[
        \begin{pmatrix}
            \Re w_{ij}\\
            \Im w_{ij}
        \end{pmatrix}
        \begin{pmatrix}
            \Re w_{ij}&\Im w_{ij}
        \end{pmatrix}
    \right]
    \geqslant \frac{c}{N}\Id_2,
    \qquad i<j,
\]
uniformly in $N$, where the inequality is in the sense of positive semidefinite matrices.

Let $W$ be a complex Hermitian generalized Wigner matrix. We write $U=(\u_1,\dots,\u_N)$ for an orthonormal eigenbasis of $W$. Since the eigenvectors are only defined up to multiplication by phases, we let $(\omega_k)_{k=1}^N$ be independent random variables, independent of $W$, uniformly distributed on $[0,2\pi)$, and define
\[
    \v_k=\sqrt N\,\e^{i\omega_k}\u_k.
\]
The limiting Gaussian variables are standard complex Gaussians: if $g$ is standard complex Gaussian, then
\[
    \Eds[g]=0,
    \qquad
    \Eds[|g|^2]=1,
    \qquad
    \Eds[g^2]=0.
\]

\begin{theorem}
\label{theo:finite_complex}
Let $W$ be a complex Hermitian generalized Wigner matrix. Let $p\geqslant 1$ and $a_1,\dots,a_p\geqslant 1$ be fixed integers, and set $Q=\sum_{i=1}^p a_i.$
Let $k_1,\dots,k_p\in\unn{1}{N}$ be distinct indices, and for every $1\leqslant i\leqslant p$, let
$\q_{i1},\dots,\q_{ia_i}$
be deterministic unit vectors in $\Cbb^N$. We set
\[
    X_N
    =
    \left(
        \scp{\q_{i\alpha}}{\v_{k_i}}
    \right)_{\substack{1\leqslant i\leqslant p\\1\leqslant \alpha\leqslant a_i}}
    \in\Cbb^Q.
\]
Let $G_N=(G_{i\alpha})_{1\leqslant i\leqslant p,\,1\leqslant \alpha\leqslant a_i}$ be the centered complex Gaussian vector in $\Cbb^Q$ with covariance
\[
    \Eds\left[G_{i\alpha}\overline{G_{j\beta}}\right]
    =
      \delta_{ij}\scp{\q_{i\alpha}}{\q_{j\beta}},    
    \qquad
    \Eds\left[G_{i\alpha}G_{j\beta}\right]=0.
\]
Then there exists $\varepsilon=\varepsilon(Q)>0$ such that for every smooth function
$h:\Cbb^Q\simeq\Rbb^{2Q}\to\Rbb$ for which there exist $K_1,K_2>0$ with
\[
    \Vert h\Vert_\infty,\,
    \Vert \nabla^2 h\Vert_\infty
    \leqslant K_1,
    \qquad
    |\partial^n h(\x)|
    \leqslant
    K_2(1+|\x|)^{K_2}
\]
for $|n|\leqslant 5$, we have
\[
    \left|
    \Eds[h(X_N)]
    -
    \Eds[h(G_N)]
    \right|
    \leqslant
    C_hN^{-\varepsilon},
\]
where $C_h=C_h(K_1,K_2,Q)$.
\end{theorem}

The quantitative nature of our dynamical result also allows us to identify the extremal process of the coordinates of an eigenvector. We use the uniform subexponential decay assumption and the space $\Mcal_p(\Rbb)$ defined before Theorem \ref{theo:extremal_process}.

\begin{theorem}
\label{theo:extremal_process_complex}
Let $W$ be a complex Hermitian generalized Wigner matrix whose entries have uniformly subexponential decay. For $k\in\unn{1}{N}$, define
\[
    \Xi_{N,k}
    =
    \sum_{\alpha=1}^N
    \delta_{N\vert \u_k(\alpha)\vert^2-b_N},
    \qquad
    b_N=\log N.
\]
Uniformly in $k$, $\Xi_{N,k}$ converges in distribution in $\Mcal_p(\Rbb)$ to a Poisson point process $\Xi$ with intensity
\[
    \nu(\d x)=\e^{-x}\,\d x.
\]
More precisely, for every nonnegative $f\in\Ccal_c(\Rbb)$,
\[
\sup_{k\in\unn{1}{N}}
\left\vert
    \Eds\left[\exp\left(-\int f\,\d\Xi_{N,k}\right)\right]
    -
    \exp\left(-\int_{\Rbb}(1-\e^{-f(x)})\e^{-x}\,\d x\right)
\right\vert
\xrightarrow[N\to\infty]{}0.
\]
Moreover, for every fixed $m\in\Nbb$, the $m$ largest atoms of $\Xi_{N,k}$ converge jointly in distribution to the $m$ largest atoms of $\Xi$, uniformly in $k$.
\end{theorem}

As a direct consequence, we obtain the Gumbel distribution for the largest entry of an eigenvector.

\begin{corollary}
\label{theo:gumbel_complex}
\label{theo:lower_complex}
Under the assumptions of Theorem \ref{theo:extremal_process_complex}, for every $x\in\Rbb$,
\[
\sup_{k\in\unn{1}{N}}
\left\vert
    \Pds\left(
        N\Vert \u_k\Vert_\infty^2-\log N
        \leqslant x
    \right)
    -\e^{-\e^{-x}}
\right\vert
\xrightarrow[N\to\infty]{}0.
\]
\end{corollary}

This corollary implies that under the assumptions of Theorem \ref{theo:extremal_process_complex} and for every $\varepsilon>0$,
\[
    \sup_{k\in\unn{1}{N}}
    \Pds\left(
        \left\vert
        \sqrt{\frac N{\log N}}\Vert \u_k\Vert_\infty-1
        \right\vert>\varepsilon
    \right)
    \xrightarrow[N\to\infty]{}0.
\]
The centering and the limiting constant differ from the real case because the tail of a standard complex Gaussian satisfies
\[
    \Pds(\vert g\vert>x)=\e^{-x^2},
    \qquad x\geqslant0.
\]

We now state the growing-dimensional version. We say that $W$ is smooth if the real and imaginary parts of the rescaled entries $\sqrt N w_{ij}$, $i<j$, and the rescaled real diagonal entries have smooth densities satisfying the analogue of Definition \ref{def:smooth}. For the off-diagonal entries, we further require that the pair
$(\sqrt N\Re w_{ij},\sqrt N\Im w_{ij})$
has a joint density proportional to $\e^{-V_{ij}(\x)}$, where $V_{ij}$ are smooth functions on $\Rbb^2$ such that for every $k\in\Nbb$, there exists $C_k>0$ with
\[
    \sup_{i<j}\vert\partial^n V_{ij}(\x)\vert
    \leqslant C_k(1+\vert\x\vert)^{C_k},
    \qquad
    \vert n\vert=k,\quad \x\in\Rbb^2.
\]
For functions on $\Cbb^Q\simeq\Rbb^{2Q}$, derivatives are taken in the real coordinates, and $\Vert\nabla^2h\Vert_\infty$ denotes the supremum of the operator norm of the real Hessian.
   
\begin{theorem}
\label{theo:growing_complex}
Let $W$ be a smooth complex Hermitian generalized Wigner matrix in the sense above. Let $p\geqslant 1$ and $a_1,\dots,a_p\geqslant 1$, and set $Q=\sum_{i=1}^p a_i.$
Let $k_1,\dots,k_p\in\unn{1}{N}$ be distinct indices, and for every $1\leqslant i\leqslant p$, let
$\q_{i1},\dots,\q_{ia_i}$
be deterministic unit vectors in $\Cbb^N$. We denote
\[
    R
    =
    \dim_{\Cbb}
    \left(
        \mathrm{Span}_{\Cbb}
        \{\q_{i\alpha},\,i\in\unn{1}{p},\alpha\in\unn{1}{a_i}\}
    \right).
\]
We set
\[
    X_N
    =
    \left(
        \scp{\q_{i\alpha}}{\v_{k_i}}
    \right)_{\substack{1\leqslant i\leqslant p\\1\leqslant \alpha\leqslant a_i}}
    \in\Cbb^Q.
\]
Let $G_N=(G_{i\alpha})_{1\leqslant i\leqslant p,\,1\leqslant \alpha\leqslant a_i}$ be the centered complex Gaussian vector in $\Cbb^Q$ with covariance
\[
    \Eds\left[G_{i\alpha}\overline{G_{j\beta}}\right]
    =
    \delta_{ij}\scp{\q_{i\alpha}}{\q_{j\beta}},
    \qquad
    \Eds\left[G_{i\alpha}G_{j\beta}\right]=0.
\]
Assume that there exist $b\in(0,\frac{1}{2}]$ and $C>0$ such that
\[
    Q\leqslant RQ\leqslant CN^{\frac{1}{2}-b}.
\]
Let $h:\Cbb^Q\simeq\Rbb^{2Q}\to\Rbb$ be a smooth function such that, for some $\delta_1>0$ and $\delta_2<b$,
\[
    \Vert h\Vert_\infty\leqslant N^{\delta_1}
    \quad\text{and}\quad
    \Vert\nabla^2h\Vert_\infty\leqslant N^{\delta_2}.
\]
For every $\kappa>0$, there exists $C_0>0$ such that
\[
    \left\vert
    \Eds\left[h(X_N)\right]
    -
    \Eds\left[h(G_N)\right]
    \right\vert
    \leqslant C_0N^{-b+\delta_2+\kappa}.
\]
Here $C_0$ may depend on $b,C,\delta_1,\delta_2,\kappa$ and the constants in the assumptions on $W$, but is uniform in the projection family and $h$.
\end{theorem}

Finally, we record the corresponding Kolmogorov bound. For complex vectors $X,Y\in\Cbb^Q$, define
\[
    \d_{\mathrm K}^{\Cbb}(X,Y)
    =
    \d_{\mathrm K}
    \left(
        (\Re X,\Im X),
        (\Re Y,\Im Y)
    \right),
\]
where the Kolmogorov distance on the right-hand side is the coordinatewise rectangular Kolmogorov distance in $\Rbb^{2Q}$.

\begin{theorem}
\label{theo:kolmogorov_complex}
With the notation and assumptions on $W$ and the projection family from Theorem \ref{theo:growing_complex}, for every $\kappa>0$ there exists $C>0$ such that
\[
    \d_{\mathrm K}^{\Cbb}(X_N,G_N)
    \leqslant
    CN^{-b+\kappa}.
\]
In particular, for a fixed number of eigenvector projections, i.e.
$b=\frac{1}{2}$, we get, for every $\kappa>0$,
\[
    \d_{\mathrm K}^{\Cbb}(X_N,G_N)
    \leqslant
    CN^{-\frac{1}{2}+\kappa}.
\]
\end{theorem}

\section{A Convergence Rate Lower Bound}
\label{app:rate}

In this appendix we show that the exponent $\frac{1}{2}$ in Theorems \ref{theo:growing} and \ref{theo:kolmogorov} for a fixed number of eigenvector projections cannot be improved in general. We construct a smooth real symmetric generalized Wigner matrix whose entries have a nonzero third moment. This gives a lower bound of order $N^{-\frac{1}{2}}$ on the difference between the second moment of an eigenvector projection and its Gaussian value, from which we obtain the following estimates.

\begin{theorem}\label{theo:rate_lower_bounds}
There exist a sequence of smooth real symmetric generalized Wigner matrices $W=W_N$ and deterministic indices $k_N\in\unn{1}{N}$ such that the following holds. We set
\[
    \q_N=\frac{1}{\sqrt{N}}(1,\dots,1)^\top,
    \qquad
    X_N=\scp{\q_N}{\v_{k_N}},
\]
where we recall $\v_k=\sqrt{N}\varepsilon_k\u_k$, and let $G_N\sim\Ncal(0,1)$.
For every $0\leqslant\delta_2<\delta_1$, there exist smooth functions $h_N:\Rbb\to\Rbb$ and a constant $c>0$ such that, for all sufficiently large $N$,
\[
    \Vert h_N\Vert_\infty\leqslant N^{\delta_1}
    \quad\text{and}\quad
    \Vert h_N''\Vert_\infty\leqslant N^{\delta_2},
\]
and
\[
    \left\vert
        \Eds\left[h_N(X_N)\right]
        -\Eds\left[h_N(G_N)\right]
    \right\vert
    \geqslant cN^{-\frac{1}{2}+\delta_2}.
\]
Moreover, for every $\kappa>0$, there exists $c_\kappa>0$ such that, for all sufficiently large $N$,
\[
    \d_{\mathrm K}(X_N,G_N)
    \geqslant c_\kappa N^{-\frac{1}{2}-\kappa}.
\]
\end{theorem}

\begin{proof}
We first construct the entry distribution. Fix $\sigma>0$ and let $Y$ have the Gaussian mixture distribution
\[
    Y\sim\frac{1}{3}\Ncal(2,\sigma^2)
    +\frac{2}{3}\Ncal(-1,\sigma^2).
\]
We set $\zeta=Y/\sqrt{\sigma^2+2}$. Then
\[
    \Eds[\zeta]=0,
    \qquad
    \Eds[\zeta^2]=1,
    \qquad
    m_3=\Eds[\zeta^3]
    =\frac{2}{(\sigma^2+2)^{\frac{3}{2}}}>0.
\]
The density of $\zeta$ has the form
\[
    p(x)=C\e^{-ax^2}
    \left(c_1\e^{b_1x}+c_2\e^{b_2x}\right)
\]
for positive constants $a,c_1,c_2$ and real constants $b_1,b_2$. Thus, if $V=-\log p$, then $V'$ grows at most linearly and $V^{(j)}$ is bounded for every $j\geqslant2$, which verifies the smoothness condition in Definition \ref{def:smooth}.
Let $(\zeta_{ij})_{i\leqslant j}$ be independent copies of $\zeta$ and define the real symmetric matrix $W$ by
\[
    w_{ij}=w_{ji}=\frac{\zeta_{ij}}{\sqrt{N}},
    \qquad i\leqslant j.
\]
All the variances $s_{ij}^2$ are equal to $N^{-1}$, so every column has total variance one. Since $\zeta$ has moments of every order, $W$ is a smooth generalized Wigner matrix.

We now estimate the second moments of the eigenvector projections. Expanding the cubic moment gives
\[
    \q_N^\top W^3\q_N
    =\frac{1}{N}\sum_{a,b,i,j=1}^N w_{ai}w_{ij}w_{jb}.
\]
By independence and centering, a summand can have nonzero expectation only when all three entries correspond to the same unordered pair of indices. Each pair $\{r,s\}$ with $r\neq s$ contributes the two choices $(a,i,j,b)=(r,s,r,s)$ and $(s,r,s,r)$, while each diagonal entry contributes once. Therefore
\begin{equation}\label{eq:rate_cubic_moment}
    \Eds\left[\q_N^\top W^3\q_N\right]
    =\frac{1}{N}\left(2\binom{N}{2}+N\right)
    \frac{m_3}{N^{\frac{3}{2}}}
    =\frac{m_3}{\sqrt{N}}.
\end{equation}
By the spectral decomposition and the orthonormality of the eigenvectors,
\[
    \q_N^\top W^3\q_N
    =\frac{1}{N}\sum_{k=1}^N
    \lambda_k^3\scp{\q_N}{\v_k}^2,
    \qquad
    \frac{1}{N}\sum_{k=1}^N\scp{\q_N}{\v_k}^2=1.
\]
Recall the semicircle quantiles $\gamma_k$ from \eqref{eq:quantiles}, with $\gamma_N=2$. By the rigidity and isotropic delocalization estimates in Lemma \ref{lem:good_event}, applied at time zero, for every sufficiently small $\xi>0$ and every $D>0$, with probability at least $1-N^{-D}$,
\begin{equation}\label{eq:rate_standard_inputs}
    \max_{k\in\unn{1}{N}}\vert\lambda_k^3-\gamma_k^3\vert
    \leqslant CN^{-\frac{2}{3}+\xi},
    \qquad
    \max_{k\in\unn{1}{N}}\scp{\q_N}{\v_k}^2
    \leqslant N^\xi.
\end{equation}
It follows that
\begin{equation}\label{eq:rate_replace_eigenvalues}
    \Eds\left[\q_N^\top W^3\q_N\right]
    =\frac{1}{N}\sum_{k=1}^N\gamma_k^3
    \Eds\left[\scp{\q_N}{\v_k}^2\right]
    +O(N^{-\frac{2}{3}+\xi}).
\end{equation}
To control the complement of the event in \eqref{eq:rate_standard_inputs}, we used Cauchy--Schwarz and the bound
\[
    \Eds\left[\Vert W\Vert^6\right]
    \leqslant\Eds\left[(\Tr W^2)^3\right]
    \leqslant CN^3,
\]
which follows from the entry moment bounds, and chose $D$ sufficiently large.

We set
\[
    d_k=\Eds\left[\scp{\q_N}{\v_k}^2\right]-1.
\]
By symmetry of the semicircle distribution,
\[
    \gamma_{N-k}=-\gamma_k,
    \qquad 1\leqslant k\leqslant N-1,
    \qquad\text{and}\qquad
    \sum_{k=1}^N\gamma_k^3=8.
\]
Combining \eqref{eq:rate_cubic_moment} and \eqref{eq:rate_replace_eigenvalues} gives
\[
    \frac{1}{N}\sum_{k=1}^N\gamma_k^3d_k
    =\frac{m_3}{\sqrt{N}}-\frac{8}{N}
    +O(N^{-\frac{2}{3}+\xi}).
\]
Choosing $\xi<\frac{1}{6}$, we obtain, for all sufficiently large $N$,
\[
    \left\vert\frac{1}{N}\sum_{k=1}^N\gamma_k^3d_k\right\vert
    \geqslant\frac{m_3}{2\sqrt{N}}.
\]
Since $\vert\gamma_k\vert\leqslant2$, choosing $k_N$ to be the smallest index maximizing $\vert d_k\vert$ gives a deterministic index such that
\begin{equation}\label{eq:rate_variance_bias}
    \left\vert\Eds\left[X_N^2\right]-1\right\vert
    =\vert d_{k_N}\vert
    \geqslant\frac{m_3}{16\sqrt{N}}.
\end{equation}

We first prove the estimate for smooth functions. Fix $0\leqslant\delta_2<\delta_1$ and let $\psi:\Rbb\to\Rbb$ be a smooth even function such that
\[
    \psi(x)=x^2\quad\text{for }\vert x\vert\leqslant1,
    \qquad
    \psi(x)=4\quad\text{for }\vert x\vert\geqslant2.
\]
Choose $C_\psi\geqslant1$ such that
$\Vert\psi\Vert_\infty+\Vert\psi''\Vert_\infty\leqslant C_\psi$ and set
\[
    \rho=\frac{\delta_1-\delta_2}{4},
    \qquad T=N^\rho,
    \qquad
    h_N(x)=C_\psi^{-1}N^{\delta_2}T^2\psi(x/T).
\]
Then
\[
    \Vert h_N\Vert_\infty
    \leqslant N^{\delta_2}T^2\leqslant N^{\delta_1},
    \qquad
    \Vert h_N''\Vert_\infty\leqslant N^{\delta_2},
\]
and $h_N(x)=C_\psi^{-1}N^{\delta_2}x^2$ for $\vert x\vert\leqslant T$.
Using \eqref{eq:rate_standard_inputs} with $\xi<2\rho$ and the deterministic bound $X_N^2\leqslant N$, for every $D>0$ we get
\[
    \Eds\left[
        \left\vert h_N(X_N)-C_\psi^{-1}N^{\delta_2}X_N^2\right\vert
    \right]
    \leqslant CN^{\delta_2}(N+T^2)\Pds(\vert X_N\vert>T)
    \leqslant N^{-D},
\]
after increasing the probability exponent in \eqref{eq:rate_standard_inputs}. The corresponding Gaussian error is bounded by
$CN^{\delta_2}(1+T^2)\e^{-cT^2}$. Therefore
\[
    \Eds\left[h_N(X_N)\right]-\Eds\left[h_N(G_N)\right]
    =C_\psi^{-1}N^{\delta_2}
    \left(\Eds\left[X_N^2\right]-1\right)
    +o(N^{-\frac{1}{2}+\delta_2}).
\]
Combining this with \eqref{eq:rate_variance_bias} proves the first estimate.

We finally prove the Kolmogorov lower bound. Fix $\kappa>0$ and set
\[
    T=N^{\frac{\kappa}{2}},
    \qquad f_T(x)=\min\{x^2,T^2\}.
\]
The function $f_T$ is absolutely continuous and satisfies
$\int_{\Rbb}\vert f_T'(x)\vert\d x=2T^2$. Integration by parts therefore gives
\begin{equation}\label{eq:rate_bv_kolmogorov}
    \left\vert
        \Eds\left[f_T(X_N)\right]-\Eds\left[f_T(G_N)\right]
    \right\vert
    \leqslant2T^2\d_{\mathrm K}(X_N,G_N).
\end{equation}
Using \eqref{eq:rate_standard_inputs} with $\xi<\kappa$ and the bound $X_N^2\leqslant N$, for every $D>0$ we have
\[
    \Eds\left[(X_N^2-T^2)_+\right]
    \leqslant N\Pds(X_N^2>N^\kappa)
    \leqslant N^{1-D},
    \qquad
    \Eds\left[(G_N^2-T^2)_+\right]
    \leqslant C\e^{-cN^\kappa}.
\]
Together with \eqref{eq:rate_variance_bias} and \eqref{eq:rate_bv_kolmogorov}, these estimates imply
\[
    \frac{m_3}{16\sqrt{N}}
    \leqslant2N^\kappa\d_{\mathrm K}(X_N,G_N)
    +N^{1-D}+C\e^{-cN^\kappa}.
\]
Choosing $D>2$ and taking $N$ sufficiently large proves the second estimate.
\end{proof}

\begin{remark}
We note that the example above uses a delocalized direction for the projection. It does not rule out a faster rate for the coordinate directions $\q_N=\eb_\alpha$.
\end{remark}

\end{document}